\documentclass[10pt,reqno]{amsart}

\usepackage{amssymb}
\usepackage{enumerate}
\usepackage{verbatim}
\usepackage{hyperref}
\usepackage{cleveref}
\usepackage{fullpage}
\usepackage[mathcal,mathbf]{euler}

\usepackage[all, 2cell]{xy}
\UseAllTwocells
\usepackage[utf8]{inputenc}
\usepackage[T1]{fontenc}
\usepackage{url}
\usepackage{tikz}
\usepackage[colorinlistoftodos, shadow]{todonotes}

\usepackage[backend=biber,
            isbn=false,
            doi=false,
            maxbibnames=5,
            giveninits=true,
            style=alphabetic,
            citestyle=alphabetic]{biblatex}
\usepackage{url}
\renewbibmacro{in:}{}
\bibliography{realizations.bib}

\newtheorem{maintheorem}{Theorem}

\newtheorem{thm}{Theorem}[section]
\newtheorem{prop}[thm]{Proposition}
\newtheorem{lem}[thm]{Lemma}
\newtheorem{cor}[thm]{Corollary}

\theoremstyle{definition}
\newtheorem{df}[thm]{Definition}
\newtheorem{rmk}[thm]{Remark}
\newtheorem{ex}[thm]{Example}

\newcommand{\bA}{\mathbb{A}}
\newcommand{\bD}{\mathbf{D}}
\newcommand{\bG}{\mathbb{G}}
\newcommand{\bK}{\mathbb{K}}
\newcommand{\bL}{\mathbb{L}}
\newcommand{\bM}{\mathbb{M}}
\newcommand{\bN}{\mathbb{N}}
\newcommand{\bT}{\mathbb{T}}
\newcommand{\bV}{\mathbb{V}}
\newcommand{\Z}{\mathbb{Z}}

\newcommand{\cA}{\mathcal{A}}
\newcommand{\cB}{\mathcal{B}}
\newcommand{\cC}{\mathcal{C}}

\newcommand{\cF}{\mathcal{F}}
\newcommand{\cK}{\mathcal{K}}
\newcommand{\cL}{\mathcal{L}}
\newcommand{\cM}{\mathcal{M}}
\newcommand{\cO}{\mathcal{O}}
\newcommand{\cQ}{\mathcal{Q}}
\newcommand{\cS}{\mathcal{S}}

\newcommand{\Acl}{\mathcal{A}^{2, \mathrm{cl}}}

\newcommand{\ddr}{\mathrm{d}_{\mathrm{dR}}}
\newcommand{\uD}{\underline{\mathbf{D}}}
\newcommand{\gr}{\mathrm{gr}}
\newcommand{\gre}{\gr,\epsilon}
\newcommand{\Sym}{\mathbf{Sym}}

\newcommand{\B}{\mathrm{B}}
\newcommand{\Bifurc}{\mathrm{Bifurc}_{\mathrm{scaled}}}
\newcommand{\Bpre}{\mathrm{B}^{\mathrm{pre}}}
\newcommand{\Bun}{\mathbf{Bun}}
\newcommand{\CAlg}{\mathbf{CAlg}}
\newcommand{\cl}{\mathrm{cl}}
\newcommand{\DR}{\mathbf{DR}}
\newcommand{\dAff}{\mathbf{dAff}}
\newcommand{\dArt}{\mathbf{dArt}}
\newcommand{\desc}{\mathrm{desc}}
\newcommand{\dPSt}{\mathbf{dPSt}}
\newcommand{\dSt}{\mathbf{dSt}}
\newcommand{\Fun}{\mathbf{Fun}}
\newcommand{\Gpd}{\mathbf{Gpd}}
\newcommand{\Hom}{\mathrm{Hom}}
\newcommand{\Lag}{\mathbf{Lag}}
\newcommand{\Map}{\mathrm{Map}}
\newcommand{\bMap}{\mathbf{Map}}
\newcommand{\Mod}{\mathbf{Mod}}
\newcommand{\op}{\mathrm{op}}
\newcommand{\Or}{\mathbf{Or}}
\newcommand{\Perf}{\mathbf{Perf}}
\newcommand{\perfdef}{\mathrm{perf}\bL}
\newcommand{\pre}{\mathrm{pre}}
\newcommand{\QCoh}{\mathbf{QCoh}}
\newcommand{\Res}{\mathbf{Res}}
\newcommand{\Span}{\mathbf{Span}}
\newcommand{\Spec}{\operatorname{\mathbf{Spec}}}
\newcommand{\Tot}{\mathbf{Tot}}

\newcommand{\defterm}[1]{\textbf{\emph{#1}}}

\DeclareMathOperator*{\colim}{colim}

\newcommand{\twospan}[5]{
\vcenter{\xymatrix{
& #4 & \\
#2 \ar[ur] \ar[dr] & #1 \ar[l] \ar[r] & #3 \ar[dl] \ar[ul] \\
& #5 &
}}}

\newcommand{\twocospan}[5]{
\vcenter{\xymatrix{
& #4 \ar[dl] \ar[dr] & \\
#2 \ar[r] & #1 & #3 \ar[l] \\
& #5 \ar[ul] \ar[ur] &
}}}

\makeatletter
\providecommand\@dotsep{5}
\def\listtodoname{List of Todos}
\def\listoftodos{\@starttoc{tdo}\listtodoname}
\makeatother

\begin{document}

\title{Shifted cotangent bundles, symplectic groupoids and deformation to the normal cone}

\address{IMAG, Univ Montpellier, CNRS, Montpellier, France}
\email{damien.calaque@umontpellier.fr}
\author{Damien Calaque}

\address{School of Mathematics, University of Edinburgh, Edinburgh, UK}
\email{p.safronov@ed.ac.uk}
\author{Pavel Safronov}
\begin{abstract}
This article generalizes the theory of shifted symplectic structures to the relative context and non-geometric stacks. We describe basic constructions that naturally appear in this theory: shifted cotangent bundles and the AKSZ procedure. Along the way, we also develop the theory of shifted symplectic groupoids presenting shifted symplectic structures on quotients and define a deformation to the normal cone for shifted Lagrangian morphisms.
\end{abstract}
\maketitle

\tableofcontents


\section*{Introduction}

\subsection*{Shifted symplectic structures in the relative context}

The theory of symplectic structures on smooth algebraic varieties was generalized to derived Artin stacks in the foundational work of Pantev, To\"en, Vaqui\'e and Vezzosi \cite{PTVV}. It has since found many applications; among others, these include:
\begin{itemize}
    \item In \cite{PTVV} the authors construct a model-independent version of the so-called AKSZ construction in mathematical physics. It was later proven that this construction can be promoted to an oriented topological field theory \cite{CalaqueTFT} that is fully extended \cite{CHS}. 
    \item The Batalin--Vilkovisky formalism is closely related to $(-1)$-shifted symplectic structures and their quantization \cite{CostelloGwilliam,Pridham1}.
    \item The cohomological Donaldson--Thomas (DT) invariants categorifying numerical DT invariants are defined using the corresponding $(-1)$-shifted symplectic structure \cite{BBDJS,BBBBJ}.
    \item A virtual fundamental class for the moduli spaces of Calabi--Yau 4-folds is defined using the corresponding $(-2)$-shifted symplectic structure \cite{BorisovJoyce,OhThomas,Pridham2}.
    \item Many ``higher'' differential structures, such as Courant algebroids, Dirac structures and generalized complex structures, have a natural interpretation in shifted symplectic geometry \cite{PymSafronov,SafronovPoissonLie,BaileyGualtieri}.
\end{itemize}

The original definition of \emph{$n$-shifted symplectic structures} in \cite{PTVV} were given for derived Artin stacks $X$ defined over a classical commutative ring $\bK$. These are given by $n$-shifted closed two-forms $\omega$ which induce an isomorphism $\bT_X\rightarrow \bL_X[n]$ from the tangent complex to the cotangent complex. This was later generalized in \cite{CPTVV,CHS} to the relative setting of geometric morphisms of derived stacks $f\colon X\rightarrow S$ which admit a perfect relative cotangent complex $\bL_{X/S}$. In this case a relative $n$-shifted symplectic structure is given by a relative $n$-shifted closed two-form $\omega$ which induces an isomorphism $\bT_{X/S}\rightarrow \bL_{X/S}[n]$. Here the geometricity of the morphism $f\colon X\rightarrow S$ is essential for the theory to work, as the natural morphism $\Gamma_S(X, \Sym(\bL_{X/S}[-1]))=f_* \Sym(\bL_{X/S}[-1])\rightarrow \DR_S(X)^{\not\epsilon}$ to the underlying graded object of the relative de Rham algebra of $X\rightarrow S$ was shown to be an isomorphism in \cite{PTVV,CHS} using the induction on the degree of geometricity of $X\rightarrow S$.

There are also interesting examples of shifted symplectic structures on formal stacks (which, in particular, are not derived Artin stacks) \cite{PymSafronov}, which requires one to perform ad hoc constructions (see e.g. \cite[Warning 1.4.23]{GrataloupThesis}). The first goal of this paper is to remove the assumption of geometricity from the definition of a shifted symplectic structure. In fact, as in \cite{BouazizGrojnowski}, we simultaneously generalize the definition as follows. For a morphism of derived prestacks $f\colon X\rightarrow S$ consider a graded line bundle $\cL$ on $S$, i.e. $\cL=L[n]$ for a line bundle $L$ on $S$ and a locally constant function $n\colon S\rightarrow \Z$. As in \cite{CHS,Park} we also have the relative de Rham algebra $\DR_S(X)$, which is a graded mixed commutative dg algebra: the individual graded pieces are given by relative $p$-forms and the mixed structure is the de Rham differential $\ddr$. One may then define \emph{$\cL$-twisted (closed) $p$-forms on $X/S$} as sections of $\DR_S(X)\otimes\cL$ over $S$, see \cref{df:closedforms}. For $\cL=\cO_S[n]$ this gives back the definition of an $n$-shifted (closed) $p$-form on $X/S$ as in \cite{PTVV,CPTVV}. If $X/S$ admits a relative cotangent complex $\bL_{X/S}\in\QCoh(X)^-$ we then prove in \cref{thm:DRgr} that the natural morphism of graded commutative dg algebras on $S$
\[\Gamma_S(X, \Sym(\bL_{X/S}[-1]))\longrightarrow \DR_S(X)^{\not\epsilon}\]
is an isomorphism. In other words, the graded pieces of $\DR_S(X)$ are exactly $\Gamma_S(X, \Sym^p(\bL_{X/S}[-1]))$, as expected. This allows us to define $\cL$-twisted symplectic structures on $X/S$ whenever the relative cotangent complex $\bL_{X/S}$ is perfect: these are given by $\cL$-twisted closed 2-forms on $X/S$ such that the induced morphism $\bT_{X/S}\rightarrow \bL_{X/S}\otimes f^*\cL$ is an isomorphism.

Such a relative formalism of twisted symplectic structures is particularly suited to the twisted AKSZ formalism developed in \cite{GinzburgRozenblyum}. Namely, consider a morphism of derived prestacks $C\rightarrow S$ equipped with a graded line bundle $\cM$ on $C$. An \emph{$\cL$-orientation on $C$ over $S$} is a morphism $[C]\colon \Gamma_S(C, \cL)\rightarrow \cO_S$, a generalization of the trace morphism in Grothendieck duality, satisfying a suitable non-degeneracy condition, see \cref{df-orientation}. Now, given morphisms $X\rightarrow C\rightarrow S$ of derived prestacks, one can define the \emph{Weil restriction} $\Res_{C/S}(X)$, which is a derived prestack over $S$ parametrizing sections of $X\rightarrow C$. We then prove the following result.

\begin{maintheorem}[{\cref{thm:relativeAKSZ}}]\label{mainthm:AKSZ}
Let $S$ be a derived prestack equipped with a graded line bundle $\cM$, $g\colon C\rightarrow S$ a morphism of derived $S$-prestack equipped with a graded line bundle $\cL$ and an $\cL$-orientation and $f\colon X\rightarrow C$ is a derived $C$-prestack equipped with an $\cL\otimes g^*\cM$-twisted symplectic structure. Then the derived $S$-prestack $\Res_{C/S}(X)$ carries a natural $\cM$-twisted symplectic structure.
\end{maintheorem}

The above theorem generalizes both the AKSZ construction for mapping stacks \cite{PTVV,CHS} and the twisted AKSZ construction for the stacks of sections \cite{GinzburgRozenblyum}.

\subsection*{Shifted symplectic groupoids}

In the same way as integrations of Lie algebras are Lie groups and integrations of Lie algebroids are Lie groupoids, integrations of Poisson structures are given by \emph{symplectic groupoids} introduced in \cite{Weinstein,CDW,Karasev,Zakrzewski}. Let us recall that a symplectic groupoid is a groupoid $G_1\rightrightarrows G_0$ equipped with a multiplicative symplectic structure on $G_1$ such that the unit map $e\colon G_0\rightarrow G_1$ is Lagrangian. Such a structure on the groupoid induces a unique Poisson structure on $G_0$ such that the source map $s\colon G_1\rightarrow G_0$ is Poisson and the target map $t\colon G_1\rightarrow G_0$ is anti-Poisson.

Symplectic groupoids naturally fit in the framework of shifted symplectic geometry: by \cite[Proposition 3.32]{SafronovPoissonLie} there is a $1$-shifted symplectic structure on the quotient stack $[G_0/G_1]$ as well as a Lagrangian structure on the quotient map $G_0\rightarrow [G_0/G_1]$; conversely, any such structure endows $G_1\rightrightarrows G_0$ with the structure of a symplectic groupoid. In this paper we give a far-reaching generalization of this phenomenon: in \cref{df:symplecticgroupoid} we introduce the notion of an \emph{$n$-shifted symplectic groupoid} $G_\bullet$ in derived stacks; moreover, we prove in \cref{thm-symplectic-quotient} that its quotient $|G_\bullet|$ inherits an $(n+1)$-shifted symplectic structure and the quotient map $G_0\rightarrow |G_\bullet|$ carries a shifted Lagrangian structure. In particular, by \cite[Theorem 4.22]{MS2} this implies that $G_0$ carries an $n$-shifted Poisson structure. Motivated by the AKSZ construction, in \cref{df:orientedcogroupoid} we also introduce and study an oriented analog of shifted symplectic groupoids: \emph{oriented co-groupoids}. In \cref{thm:AKSZsymplectic} we show that they give rise to shifted symplectic groupoids as follows: if $C^\bullet$ is a $d$-oriented co-groupoid and $X$ is an $n$-shifted symplectic stack, then the mapping stack $\bMap(C^\bullet, X)$ has a natural structure of a $(n-d)$-shifted symplectic groupoid.

Let us mention similarities and differences between our notion of an $n$-shifted symplectic groupoid and shifted symplectic higher Lie groupoids introduced in \cite{CuecaZhu}. Both notions are used to present $(n+1)$-shifted symplectic structures on (derived) Artin stacks $X$. In our definition we require the atlas $G_0\rightarrow X$ to be equipped with a shifted Lagrangian structure, which is not the case for the notion from \cite{CuecaZhu}. Moreover, the latter notion is formulated in terms of hypergroupoids of smooth manifolds, while we consider ordinary groupoids, but allow the individual spaces to be derived stacks. We refer to \cite{PridhamPoisson,PridhamOutline,PridhamNote} for a description of shifted symplectic structures on derived Artin stacks presented in terms of hypergroupoids in algebraic and differential geometry.

\subsection*{Functoriality of cotangent bundles}

Suppose $f\colon X\rightarrow Y$ is a morphism of derived prestacks which admit perfect cotangent complexes. Then we have a morphism of tangent complexes $\bT_X\rightarrow f^*\bT_Y$, which induces a morphism $T[n]X\rightarrow T[n]Y$ of shifted tangent bundles. However, the dual morphism of cotangent complexes $f^*\bL_Y\rightarrow \bL_X$ gives a backwards functoriality and so instead induces a correspondence
\[
\xymatrix{
& T^*[n] Y\times_Y X \ar[dl] \ar[dr] & \\
T^*[n]Y && T^*[n]X
}
\]

In \cite{CalaqueCotangent} it was established that if $X$ is a derived Artin stack, $T^*[n] X$ is $n$-shifted symplectic. Moreover, if $f\colon X\rightarrow Y$ is a morphism of derived Artin stacks, the above correspondence is $n$-shifted Lagrangian. Our next goal is, first, remove the assumption that $X$ is Artin and, second, establish $\infty$-categorical functoriality of the cotangent bundle.

\begin{maintheorem}[{\cref{thm:cotangentnondegenerate}, \cref{thm:spancotangent}}]\label{mainthm:cotangent}
There is a symmetric monoidal functor
\[T^*[n]\colon \Span_1(\dPSt^{\perfdef})\longrightarrow \Lag_1^n\]
which sends a derived prestack $X$ which admits a perfect cotangent complex to the $n$-shifted cotangent bundle $T^*[n] X$ equipped with its natural $n$-shifted symplectic structure and a span $X\leftarrow Z\rightarrow Y$ of derived prestacks to the Lagrangian correspondence
\[
\xymatrix{
& N^*[n](Z/X\times Y) \ar[dl] \ar[dr] & \\
T^*[n] X && T^*[n] Y,
}
\]
where $N^*[n](Z/X\times Y)$ is the $n$-shifted conormal bundle of $Z\rightarrow X\times Y$.
\end{maintheorem}

Let us explain one application of \cref{mainthm:cotangent}. Given a smooth groupoid $G\rightrightarrows X$ over a smooth scheme $X$ recall its \emph{cotangent groupoid} $T^* G\rightrightarrows L^*$, where $L\rightarrow X$ is the corresponding Lie algebroid, which is a basic example of a symplectic groupoid. Consider a groupoid stack $G_\bullet$ which presents a derived stack $|G_\bullet|$. In \cref{prop:shiftedcotangentgroupoid} we construct an $n$-shifted symplectic groupoid with the stack of arrows $T^*[n] G_1$ and show that it presents the $(n+1)$-shifted cotangent bundle $T^*[n+1]|G_\bullet|$. This construction is related to \cref{mainthm:cotangent} in the following way. By \cite{Stern} a 2-Segal object in $\dSt$, such as the groupoid $G_\bullet$, is the same as an algebra object in the $\infty$-category of spans $\Span_1(\dSt)$: the unit is given by the span $pt\leftarrow G_0\rightarrow G_1$ and the multiplication is given by the span $G_1\times G_1\leftarrow G_2\rightarrow G_1$. Applying the shifted cotangent bundle functor $T^*[n]$ from \cref{mainthm:cotangent}, we obtain another algebra object $T^*[n] G_1\in \Lag_1^n$, which can be regarded again as a 2-Segal object in $\dSt$. It is shown in the upcoming paper of the first author with David Kern that the resulting object is, in fact, the 2-Segal object corresponding to the $n$-shifted cotangent groupoid.

\subsection*{Deformation to the normal cone}

Given a closed subscheme $L\subset X$ of a scheme, a basic construction in algebraic geometry is that of \emph{deformation to the normal cone}: it is a construction of a scheme $\bD_{L/X}$ flat over $\bA^1$ equipped with a closed immersion $L\times \bA^1\subset \bD_{L/X}$, such that the fiber over $\lambda\neq 0\in\bA^1$ gives back the embedding $L\subset X$ and the fiber over $\lambda=0\in\bA^1$ gives the inclusion $L\subset C_L X$ of $L$ into the normal cone of $L$ in $X$. We refer to \cite[Chapter 5]{Fulton} for more details.

The setting of derived algebraic geometry allows one to greatly generalize the framework to general morphisms $L\rightarrow X$ of derived prestacks. We refer to \cite{Simpson}, \cite[Chapter 9]{GR2}, \cite{KhanRydh} and \cite{Hekking} where these ideas were developed. As an example, the geometric incarnation of the Hodge filtration on the de Rham cohomology of $L$ arises via the deformation to the normal cone of the projection $L\rightarrow pt$. For any morphism $f\colon L\rightarrow X$ of derived prestacks admitting a relative cotangent complex, the deformation to the normal cone is given by a certain derived prestack $\bD_{L/X}\rightarrow \bA^1$ with the following properties (we refer to \cref{sect:deformationspace} for more details):
\begin{itemize}
    \item It is equipped with morphisms $L\times \bA^1\rightarrow \bD_{L/X}\rightarrow X\times \bA^1$ composing to $f\times id$.
    \item It is $\bG_m$-equivariant.
    \item The fiber over $\lambda\neq 0\in\bA^1$ is $X$.
    \item The fiber over $\lambda=0\in\bA^1$ is the shifted tangent bundle $T[1](L/X)$, the derived analog of the normal cone.
\end{itemize}

Now suppose $L\rightarrow X$ is equipped with an $n$-shifted Lagrangian structure. Using that structure, we may identify $T[1](L/X)\cong T^*[n+1] L$ which carries an $(n+1)$-shifted symplectic structure. It was conjectured in \cite[Conjecture 3.2]{CalaqueCotangent} (motivated by connections to the theory of shifted Poisson structures) that the whole deformation to the normal cone can be made compatible with shifted symplectic structures, and we prove this conjecture.

\begin{maintheorem}[{\cref{cor:calaqueconjecture}}]\label{mainthm:deformationnormalcone}
Let $L\rightarrow X$ be a morphism of derived prestacks equipped with an $n$-shifted Lagrangian structure. Then the deformation to the normal cone $L\times\bA^1\rightarrow \bD_{L/X}$ of $L\rightarrow X$ carries the structure of an $n$-shifted Lagrangian morphism relative to $\bA^1$. Its fiber at a nonzero $\lambda\in\bA^1$ is the original $n$-shifted Lagrangian morphism $L\rightarrow X$ and its fiber at $0\in\bA^1$ is the zero section $L\rightarrow T^*[n] L$.
\end{maintheorem}

Similar to how the usual deformation to the normal cone replaces tubular neighborhoods available in differential geometry (i.e. for $C^\infty$ submanifolds), \cref{mainthm:deformationnormalcone} provides an analog of Weinstein's Lagrangian neighborhood theorem in $C^\infty$ symplectic geometry \cite[Theorem 9.3]{CannasDaSilva}.


\subsection*{Acknowledgements}

DC thanks Gabriele Vezzosi for enlightening discussions at an early stage of this project about 10 years ago, as well as David Kern for more recent conversations. 
He also acknowledges the hospitality of the Galileo Galilei Institute, where parts of the work have been done on the occasion of two workshops:  
``Geometry of Strings and Fields'' in 2013 and ``Emergent Geometries from Strings and Quantum Fields'' in 2023. 
He has received funding from the European Research Council (ERC) under the European Union's Horizon 2020 research and innovation programme (grant
agreement No 768679). PS thanks Tasuki Kinjo and Hyeonjun Park for useful discussions. We thank Benjamin Hennion, Ruoxi Li and the anonymous referee for useful comments which have improved the exposition.


\section{Preliminaries}

\subsection{Derived algebraic geometry}

We begin by recalling and fixing the notation for objects of derived algebraic geometry. We denote by $\cS$ the $\infty$-category of spaces. We denote by $\Map_{\cC}(-, -)$ the mapping space in an $\infty$-category $\cC$.

Throughout the paper we work over a $\mathbb{Q}$-algebra $\bK$. Let $\CAlg=\CAlg_\bK$ be the $\infty$-category of commutative dg algebras (cdgas) over $\bK$ and $\CAlg^c\subset \CAlg$ the full subcategory of connective commutative dg algebras. The $\infty$-category of derived affine schemes $\dAff$ is the opposite category of $\CAlg^c$. In other words, there is a tautological equivalence $\Spec\colon \CAlg^c_\bK\rightarrow \dAff^{\op}$.

We have the following $\infty$-categories of derived (pre)stacks:
\begin{itemize}
    \item $\dPSt$, the $\infty$-category of \defterm{derived prestacks}, is the $\infty$-category of accessible functors $\CAlg^c\rightarrow \cS$. For $S\in\dPSt$ we denote $\dPSt_S=\dPSt_{/S}$ and refer to the objects as derived $S$-prestacks.
    \item $\dSt\subset \dPSt$ is the $\infty$-category of \defterm{derived stacks} over $\bK$, i.e. accessible functors $\CAlg^c\rightarrow \cS$ satisfying \'etale descent. For $S\in\dSt$ we denote $\dSt_S=\dSt_{/S}$ and refer to the objects as derived $S$-stacks.
\end{itemize}

We have the following basic examples of derived (pre)stacks:
\begin{itemize}
    \item $\Bpre\bG_m$ is the classifying prestack of the group scheme $\bG_m$ parametrizes trivial line bundles.
    \item $\B\bG_m$ is the classifying stack obtained by applying sheafification to $\Bpre\bG_m$; it parametrizes all line bundles.
    \item We denote by $Pic^{\gr}\cong\B\bG_m\times\Z$ the derived stack of graded line bundles.
    \item For a derived prestack $S$ and derived $S$-prestacks $X, Y$ we denote by $\bMap_S(X, Y)\in\dPSt_S$ the relative mapping prestack. If $Y$ is a derived stack, so is $\bMap_S(X, Y)$.
\end{itemize}

For a derived affine scheme $S=\Spec A$ we define $\QCoh(S)=\Mod_A$, the $\infty$-category of $A$-modules. Denote $\cO_S=A$. For a morphism of derived affine schemes $f\colon T=\Spec B\rightarrow S=\Spec A$ there is an induced pullback functor $f^*\colon \QCoh(S)\rightarrow \QCoh(T)$ given by $M\mapsto M\otimes_A B$. Let $\Perf(S)\subset \QCoh(S)$ be the full subcategory of dualizable objects.

We define $\QCoh(X)$ and $\Perf(X)$ by the right Kan extensions, i.e.
\[\QCoh(X) = \lim_{S\in \dAff_{/X}^{\op}} \QCoh(S),\qquad \Perf(X) = \lim_{S\in \dAff_{/X}^{\op}} \Perf(S).\]
Both of these $\infty$-categories inherit symmetric monoidal structures and by \cite[Proposition 4.6.1.11]{HA} $\Perf(X)\subset \QCoh(X)$ is the full subcategory of dualizable objects. For a morphism $f\colon X\rightarrow Y$ of derived prestacks we have a pullback morphism $f^*\colon \QCoh(Y)\rightarrow \QCoh(X)$ and its right adjoint $f_*\colon \QCoh(X)\rightarrow \QCoh(Y)$. To leave $f$ implicit, we often denote
\[\Gamma_Y(X, -) = f_*(-).\]

$\QCoh(X)$ satisfies descent, so that the natural functor $\QCoh(L(X))\rightarrow \QCoh(X)$ is an equivalence (see \cite[Chapter 3, Corollary 1.3.8]{GR1}), where $L\colon \dPSt\rightarrow \dSt$ is the sheafification functor. In particular, if $X = \colim_i X_i$ is a colimit of derived stacks, then
\[\QCoh(X)\cong \lim_i \QCoh(X_i).\]

\begin{ex}\label{ex:QCohBGm}
Consider the derived prestack $\Bpre\bG_m$ obtained as the classifying prestack of the group scheme $\bG_m$. For an $\infty$-category $\cC$ denote by $\cC^{\gr}=\Fun(\Z, \cC)$ the $\infty$-category of $\Z$-graded objects in $\cC$. Then for any derived prestack $X$ we may identify
\[\QCoh(X)^{\gr}\cong \QCoh(X\times\Bpre\bG_m)\cong \QCoh(X)\otimes \QCoh(\Bpre\bG_m),\]
where the first equivalence is shown in \cite[Chapter 9, Proposition 1.3.3]{GR2} and the second equivalence follows from \cite[Proposition B.2.2]{GaitsgorySheaves} since $\QCoh(\Bpre\bG_m)$ is dualizable. Here the tensor product on the right-hand side is the tensor product of $k$-linear presentable $\infty$-categories as defined in \cite[Proposition 4.8.1.15]{HA}. Moreover, $\Perf(X\times\Bpre\bG_m)\subset \Perf(X)^{\gr}$ is the full subcategory consisting of graded perfect complexes which are concentrated in finitely many weights \cite[Proposition 5.8]{KPS}. For $V\in\QCoh(X)$ we denote by $V(p)\in\QCoh(X\times \Bpre\bG_m)$ the object corresponding under the equivalence $\QCoh(X\times \Bpre\bG_m)\cong \QCoh(X)^{\gr}$ to the quasi-coherent complex $V$ concentrated purely in weight $p$.
\end{ex}

\begin{lem}\label{lm:projectionformula}
Let $f\colon X\rightarrow Y$ be a morphism of derived prestacks, $\cF\in\QCoh(X)$ and $V\in\Perf(Y)$. Then the natural morphism
\[f_*\cF\otimes V\longrightarrow f_*(\cF\otimes f^*V)\]
is an isomorphism.
\end{lem}
\begin{proof}
Indeed, the composite
\[f_*(\cF\otimes f^*V)\xrightarrow{coev} f_*(\cF\otimes f^* V)\otimes V^\vee\otimes V\rightarrow f_*(\cF\otimes f^*(V\otimes V^\vee))\otimes V\xrightarrow{ev} f_*\cF\otimes V.\]
is the inverse.
\end{proof}

Given a derived prestack $X$ and a commutative algebra $\cA\in\CAlg(\QCoh(X))$, its \defterm{relative spectrum} $\Spec \cA\in\dPSt_X$ has the functor of points
\[(\Spec \cA)(T) = \Map_{\CAlg(\QCoh(T))}(f^*\cA, \cO_T)\]
for a derived affine scheme $T$ with a morphism $f\colon T\rightarrow X$. If $X$ is a derived stack, so is $\Spec \cA$.

Recall e.g. from \cite[Chapter 1.4]{GR2} or \cite[Definition B.10.2]{CHS} the notion of a derived prestack $X\rightarrow S$ admitting a \defterm{cotangent complex} $\bL_{X/S}\in\QCoh(X)^-$ relative to $S$. We denote by $\dPSt^{\bL}_S\subset \dPSt_S$ the full subcategory of derived $S$-prestacks which admit a cotangent complex $\bL_{X/S}\in\QCoh(X)^-$ relative to $S$ and by $\dPSt^{\perfdef}_S\subset \dPSt^{\bL}_S$ when $\bL_{X/S}\in\Perf(X)$. We use a similar notation $\dSt^{\perfdef}_S\subset \dSt^{\bL}_S\subset \dSt_S$ for derived stacks. For $X\in\dSt^{\perfdef}$ we denote by $\bT_X$, the \defterm{tangent complex}, the dual of $\bL_X$.

Let $n\geq -1$. Recall from \cite[Definition 1.3.3.1]{HAGII} the notions of \defterm{$n$-geometric stacks} $X\in\dSt$, \defterm{$n$-representable morphisms} $f\colon X\rightarrow Y$ and \defterm{$n$-smooth morphisms} $f\colon X\rightarrow Y$. We say a morphism $f\colon X\rightarrow Y$ is \defterm{geometric} if it is $n$-geometric for some $n$.

Any $n$-representable morphism $f\colon X\rightarrow Y$ admits a relative cotangent complex $\bL_{X/Y}\in\QCoh(X)^{\leq (n+1)}$ by \cite[Proposition 1.4.1.11]{HAGII}. If we moreover assume that $f$ is $n$-smooth, then $\bL_{X/Y}$ is perfect of Tor-amplitude $[0, n+1]$.

A \defterm{derived Artin stack} is a geometric stack locally of finite presentation. We denote by $\dArt\subset \dSt$ the full subcategory of derived Artin stacks. A derived Artin stack admits a perfect cotangent complex, so that $\dArt\subset \dSt^{\perfdef}$.

\subsection{Groupoids in derived algebraic geometry}\label{subsection1.2}

Here and below, a \defterm{(derived) groupoid stack} is a Segal groupoid object in $\dSt$ (a simplicial object satisfying the Segal and invertibility conditions), see \cite[Definition 1.3.1.6]{HAGII} and \cite[Chapter 6.1.2]{HTT}. A groupoid prestack is similarly a Segal groupoid object in $\dPSt$. For a groupoid stack $G_\bullet$ we denote by $|G_\bullet|\in\dSt$ its geometric realization (taken in $\dSt$). Similarly, for a groupoid prestack $G_\bullet$ we denote by $|G_\bullet|^{\pre}\in\dPSt$ the geometric realization taken in $\dPSt$. If $G_\bullet$ is a groupoid stack, we may identify $|G_\bullet|\cong L(|G_\bullet|^{\pre})$, where $L$ is the sheafification functor.

We denote by $\Gpd = \Gpd(\dSt)$ the $\infty$-category of Segal groupoid objects in $\dSt$. As a matter of notation, we will write $s,t\colon G_1\to G_0$ for the lowest face maps, and $e\colon G_0\to G_n$ for the sole degeneracy map starting from $G_0$.

\begin{itemize}
\item There is a functor
\[\bM\colon \Gpd\longrightarrow\dSt^{\Delta^1}\]
which sends a groupoid stack $G_\bullet$ to the map $G_0\to |G_\bullet|$.

\item There is another functor
\[\bN\colon\dSt^{\Delta^1}\longrightarrow\Gpd\]
that sends a map $X\to Y$ to its \v{C}ech nerve. It is right adjoint to $\bM$.
\end{itemize}

\begin{df}
A morphism $f\colon X\rightarrow Y$ of derived stacks is an \defterm{effective epimorphism} if the natural morphism $|\bN(f)|\rightarrow Y$ is an isomorphism.
\end{df}

We will use the following basic facts about effective epimorphisms of derived stacks.

\begin{prop}\label{prop:epimorphisms}
Let $f\colon X\rightarrow Y$ be a morphism of derived stacks. The following conditions are equivalent:
\begin{itemize}
    \item $f\colon X\rightarrow Y$ is an effective epimorphism.
    \item $\pi_0(f)\colon \pi_0(X)\rightarrow \pi_0(Y)$ is an epimorphism in the category of sheaves of sets, i.e. $f$ has a section \'etale locally.
\end{itemize}
\end{prop}
\begin{proof}
The $\infty$-category $\dSt$ is an $\infty$-topos. By \cite[Proposition 7.2.1.14]{HTT} a morphism $f\colon X\rightarrow Y$ in $\dSt$ is an effective epimorphism precisely when $\pi_0(f)\colon \pi_0(X)\rightarrow \pi_0(Y)$ is an effective epimorphism of sheaves of sets. By \cite[Theorem IV.7.8]{MacLaneMoerdijk} $\pi_0(f)$ is an effective epimorphism precisely if it is an epimorphism.
\end{proof}

\begin{prop}
The functor $\bM\colon \Gpd\rightarrow \dSt^{\Delta^1}$ is fully faithful with the essential image given by effective epimorphisms of derived stacks.
\label{prop:realizationfullyfaithful}
\end{prop}
\begin{proof}
The $\infty$-category $\dSt$ is an $\infty$-topos. Therefore, by \cite[Theorem 6.1.0.6]{HTT} the unit of the adjunction $id\rightarrow \bM\circ \bN$ is an equivalence. In particular, $\bM$ is fully faithful. The claim about the essential image is shown in \cite[Corollary 6.2.3.5]{HTT}.
\end{proof}

\begin{prop}\label{prop:coveringconservative}
Let $f\colon X\rightarrow Y$ be an effective epimorphism of derived stacks. Then:
\begin{itemize}
    \item The functor $f^*\colon \QCoh(Y)\rightarrow \QCoh(X)$ is conservative.
    \item An object $\cF\in\QCoh(Y)$ is perfect if, and only if, $f^*\cF$ is perfect.
\end{itemize}
\end{prop}
\begin{proof}
Let $G_\bullet = \bN(f)$ be the corresponding groupoid stack. Then $\QCoh(Y)\cong \lim_{[n]\in\Delta}(\QCoh(G_n))$.

\begin{enumerate}
    \item The assertion reduces to the claim that $\lim_{[n]\in\Delta}(\QCoh(G_n))\rightarrow \QCoh(G_0)$ is conservative. Now suppose $\alpha\colon x\rightarrow y$ is a morphism in $\lim_{[n]\in\Delta}(\QCoh(G_n))$ whose image in $\QCoh(G_0)$ is an equivalence. For any $n$ there is a morphism $[n]\rightarrow [0]$ in $\Delta$ and hence the image of $\alpha$ in $\QCoh(G_n)$ is an equivalence. But the collection of forgetful functors $\lim_{[n]\in\Delta}(\QCoh(G_n))\rightarrow \{\QCoh(G_n)\}_{n}$ is jointly conservative by \cite[Proposition 5.2.2.36(a)]{HA}.
    \item By \cite[Proposition 4.6.1.11]{HA} an object $\cF\in\QCoh(Y)$ is dualizable if, and only if, its image in $\QCoh(G_n)$ is dualizable for every $n$. Again using a morphism $[n]\rightarrow [0]$ we see that the latter condition is equivalent to the condition that the image of $\cF$ in $\QCoh(G_0)$ is dualizable.
\end{enumerate}
\end{proof}

We make the following observation about the cotangent complexes for groupoids.

\begin{lem}\label{lm:groupoidcotangent}
Let $G_\bullet$ be a groupoid stack such that $G_0$ and $G_1$ admit cotangent complexes.
\begin{enumerate}
    \item $G_n$ admits a cotangent complex for every $n$.
    \item If $|G_\bullet|$ admits a cotangent complex, then $\bL_{G_0/|G_\bullet|}\cong \bL_{G_0/G_1}[-1]$.
    \item If $|G_\bullet|$ admits a cotangent complex and the cotangent complexes of $G_0$ and $G_1$ are perfect, then $\bL_{|G_\bullet|}$ is perfect.
\end{enumerate}
\end{lem}
\begin{proof}$ $
\begin{enumerate}
    \item By the Segal conditions, $G_n$ is isomorphic to $G_1\times_{G_0}\dots \times_{G_0} G_1$ ($n$ times). Since $G_0$ and $G_1$ admit cotangent complexes, so does $G_n$.
    \item Let $p\colon G_0\rightarrow |G_\bullet|$ be the projection. Using the Cartesian diagram
    \[
    \xymatrix{
    G_1 \ar^{t}[r] \ar^{s}[d] & G_0 \ar^{p}[d] \\
    G_0 \ar^{p}[r] & |G_\bullet|
    }
    \]
    we get $t^*\bL_{G_0/|G_\bullet|}\cong \bL_s$, the relative cotangent complex of $s\colon G_1\rightarrow G_0$ Further pulling it back under the unit map $e\colon G_0\rightarrow G_1$ and using that $s\circ e=id_{G_0}$, we get $\bL_{G_0/|G_\bullet|}\cong e^*\bL_s$. The morphism $e\colon G_0\rightarrow G_1$ induces a fiber sequence
    \[e^*\bL_s\longrightarrow \bL_{G_0/G_0}\longrightarrow \bL_{G_0/G_1}.\]
    Since the middle term is zero, we obtain $e^*\bL_s\cong \bL_{G_0/G_1}[-1]$.
    \item The morphism $G_0\rightarrow |G_\bullet|$ induces a fiber sequence
    \[p^*\bL_{|G_\bullet|}\longrightarrow \bL_{G_0}\longrightarrow \bL_{G_0/|G_\bullet|}.\]
    By part (2) and the assumption, $p^*\bL_{|G_\bullet|}$ is perfect. By \cref{prop:coveringconservative} this implies that $\bL_{|G_\bullet|}$ is perfect.
\end{enumerate}
\end{proof}

\begin{df}
An \defterm{action of a groupoid stack $G_\bullet$ on a stack $X$} is the data of morphism of groupoid stacks $X_\bullet\to G_\bullet$ together with an isomorphism $X_0\cong X$ and such that for every morphism $[m]\rightarrow [n]$ the square
\[
\xymatrix{
X_n\ar[r]\ar[d] & X_m\ar[d] \\
G_n\ar[r] & G_m
}
\]
is Cartesian.
\end{df}

We write $G_\bullet-\dSt$ for the full $\infty$-subcategory of $(\dSt^{\Delta^{\op}})_{/G_\bullet}$ spanned by 
stacks together with a $G_\bullet$-action. 

\begin{ex}
Let $f\colon X\to Y$ be a morphism between derived stacks. Then we have the groupoid stack $\mathbb{N}(f)$. 
Assume we have another morphism $g\colon Y'\to Y$ and define $X':=X\times_Y^hY'\xrightarrow{\pi_2} Y'$. 
The morphism $g$ induces a morphism of groupoids $\mathbb{N}(\pi_2)\to\mathbb{N}(f)$, which is an action of
 $\mathbb{N}(f)$ on $X'$. This is illustrated by the following diagram:
\[
\xymatrix{
\dots \ar@<.8ex>[r] \ar[r] \ar@<-.8ex>[r] & X\times_Y Y'\times_Y X \ar[d] \ar@<.5ex>[r] \ar@<-.5ex>[r] & X\times_Y Y' \ar[d] \ar^{\pi_2}[r] & Y' \ar^{g}[d] \\
\dots \ar@<.8ex>[r] \ar[r] \ar@<-.8ex>[r] & X\times_Y X \ar@<.5ex>[r] \ar@<-.5ex>[r] &  X \ar^{f}[r] & Y
}
\]
\end{ex}

Let us recall the following result, which follows from the fact that $\dSt$ is an $\infty$-topos (see e.g.~\cite{HAGII}). 

\begin{prop}\label{prop:actionquotient}
Let $G_\bullet$ be a groupoid stack. Then the functor $X_\bullet\mapsto |X_\bullet|$ gives an equivalence 
\[G_\bullet-\dSt\xrightarrow{\sim} \dSt_{/|G_\bullet|}\,.\]
Similarly, if $G_\bullet$ is a groupoid prestack, then the functor $X_\bullet\mapsto |X_\bullet|^{\pre}$ gives an equivalence
\[G_\bullet-\dPSt\xrightarrow{\sim} \dPSt_{/|G_\bullet|^{\pre}}\,.\]
\end{prop}

\begin{ex}\label{ex:Gmstacks}
Let $G_\bullet$ be the nerve of the group scheme $\bG_m$, so that $G_n=\bG_m^n$. In this case \cref{prop:actionquotient} identifies
\begin{itemize}
    \item Derived stacks with a $\bG_m$-action and derived stacks over $\B\bG_m$
    \item Derived prestacks with a $\bG_m$-action and derived prestacks over $\Bpre\bG_m$.
\end{itemize}
Namely, given a derived stack $X$ with a $\bG_m$-action we have the derived stack $[X/\bG_m]\rightarrow \B\bG_m$. Conversely, given a derived stack $Y\rightarrow \B\bG_m$, the fiber product $pt\times_{\B\bG_m} Y$ is a derived stack with a $\bG_m$-action. The same construction works with derived prestacks if we replace the stack quotient $[X/\bG_m]$ by the prestack quotient $X/\bG_m$.
\end{ex}

For a derived $S$-prestack $X$ with a $\bG_m$-action, let $p^{\gr}\colon X/\bG_m\rightarrow S\times \Bpre\bG_m$ be the corresponding projection. For a $\bG_m$-equivariant quasi-coherent complex $V\in\QCoh(X)^{\bG_m}$ we denote by
\[\Gamma^{\gr}_S(X, V)\in\QCoh(S)^{\gr}\]
the object corresponding to $p^{\gr}_* V\in\QCoh(S\times \Bpre\bG_m)$ under the equivalence $\QCoh(S)^{\gr}\cong \QCoh(S\times \Bpre\bG_m)$ from \cref{ex:QCohBGm}.

\begin{df}
A groupoid stack $G_\bullet$ is \defterm{smooth} if $G_k$ belongs to $\dArt$ for every $k$ and all face maps are smooth. 
\end{df}

\begin{rmk}
If $G_0,G_1$ are in $\dArt$ and $s$ (or $t$) is smooth, then $G_\bullet$ is smooth.
\end{rmk}

By \cite[Proposition 1.3.4.2]{HAGII} it follows that if $G_\bullet$ is smooth, then $|G_\bullet|$ is a derived Artin stack.

\subsection{Sheaves of $\cO$-modules}

In this section we recall the notion of a sheaf of (non-quasi-coherent) $\cO_X$-modules on a derived prestack from \cite[Section B.4]{CHS}.

Let $\cM$ be the $\infty$-category of pairs $(A, M)$ of a cdga $A\in \CAlg^c_\bK$ and a dg $A$-module $M\in\Mod_A$ with morphisms $(A_1, M_1)\rightarrow (A_2, M_2)$ given by a pair of a morphism $f\colon A_1\rightarrow A_2$ in $\CAlg^c_\bK$ and a morphism $M_1\rightarrow M_2$ in $\Mod_{A_1}$. The projection $\cM\rightarrow \CAlg^c_\bK = \dAff^{\op}$ forms a coCartesian fibration. For a derived prestack $X$ denote by
\[\cM_X = \dAff_{/X}^{\op}\times_{\dAff} \cM\longrightarrow \dAff_{/X}^{\op}\]
the pullback coCartesian fibration over $\dAff_{/X}$.

\begin{df}
Let $X$ be a derived prestack.
\begin{itemize}
\item Let $\Mod^{\pre}_{\cO_X}$ (the $\infty$-category of \defterm{presheaves of $\cO_X$-modules}) be the $\infty$-category of sections of the coCartesian fibration $\cM_X\rightarrow \dAff_{/X}^{\op}$.
\item Let $\Mod_{\cO_X}\subset \Mod^{\pre}_{\cO_X}$ (the $\infty$-category of \defterm{sheaves of $\cO_X$-modules}) be the full subcategory consisting of objects satisfying \'etale descent, i.e. presheaves of $\cO_X$-modules $V$ such that for any \'etale cover $p\colon T\rightarrow S$ in $\dAff_{/X}$ the induced map
\[V_S\longrightarrow \lim_{[n]\in\Delta} V_{T_n}\]
is an equivalence, where $T_\bullet$ is the \v{C}ech nerve of $p$ and $V_S\in \QCoh(S)$ is the value of the presheaf of $\cO_X$-modules on $S\rightarrow X$.
\end{itemize}
\end{df}

\begin{rmk}
By \cite[Corollary 3.3.3.2]{HTT} $\QCoh(X)\subset \Mod^{\pre}_{\cO_X}$ is the full subcategory of coCartesian sections of $\cM_X\rightarrow \dAff_{/X}^{\op}$.
\end{rmk}

\begin{ex}
For any derived prestack $X$ there is an object $\cO_X\in \QCoh(X)$ which corresponds to the coCartesian section $(S\rightarrow X)\mapsto (S, \cO_S)$.
\end{ex}

Unpacking the definitions, a presheaf of $\cO_X$-modules is an assignment of a quasi-coherent sheaf $V_S$ for every $f\colon S\rightarrow X$, where $S$ is affine, and a morphism of quasi-coherent sheaves $g^* V_S\rightarrow V_T$ for every morphism $T\xrightarrow{g} S\rightarrow X$ of derived affine schemes. According to \cite[Proposition B.5.3]{CHS} we in fact have $\QCoh(X)\subset \Mod_{\cO_X}$ and quasi-coherent sheaves correspond to sheaves of $\cO_X$-modules such that the morphism $g^* V_S\rightarrow V_T$ is an equivalence.

For a morphism $f\colon X\rightarrow Y$ of derived prestacks there is a natural pullback functor $f^*\colon \Mod^{\pre}_{\cO_Y}\rightarrow \Mod^{\pre}_{\cO_X}$ which sends $((T\rightarrow Y)\mapsto V_T)$ to $((T\rightarrow X\rightarrow Y)\mapsto V_T)$ which preserves the subcategories of sheaves of $\cO_X$-modules and quasi-coherent sheaves. Let $\cQ^{\pre}\rightarrow \dPSt^{\op}$ be the coCartesian fibration classified by the functor $X\in\dPSt\mapsto \Mod^{\pre}_{\cO_X}$. Denote by $\cQ^{\desc}\subset \cQ^{\pre}$ the full subcategory which is classified by the functor $X\in\dSt\mapsto \Mod_{\cO_X}$ and by $\cQ\subset \cQ^{\desc}$ the full subcategory which is classified by the functor $X\in\dSt\mapsto \QCoh(X)$.

Explicitly, $\cQ^{\pre}$ has objects pairs $(X, V)$, where $X\in \dPSt$ is a derived prestack and $V\in \Mod^{\pre}_{\cO_X}$ is a presheaf of $\cO_X$-modules and morphisms $(X_1, V_1)\rightarrow (X_2, V_2)$ are morphisms $f\colon X_2\rightarrow X_1$ of derived prestacks together with a morphism $f^* V_1\rightarrow V_2$ of presheaves of $\cO_{X_2}$-modules. $\cQ\subset \cQ^{\pre}$ is the full subcategory of pairs $(X, V)$, where $X$ is a derived stack and $V$ is quasi-coherent.

Let $\cQ^{\pre}_{\op}\rightarrow \dPSt$ be the Cartesian fibration classified by the functor $X\in\dPSt\mapsto (\Mod^{\pre}_{\cO_X})^{\op}$. Explicitly, it has objects pairs $(X, V)$, where $X\in \dPSt$ and $V\in \Mod^{\pre}_{\cO_X}$ and morphisms $(X_1, V_1)\rightarrow (X_2, V_2)$ are morphisms $f\colon X_1\rightarrow X_2$ of derived prestacks together with a morphism $V_1\rightarrow f^* V_2$ of presheaves of $\cO_{X_2}$-modules. We denote by $\cQ_{\op}\subset \cQ^{\desc}_{\op}\subset \cQ^{\pre}_{\op}$ the Cartesian fibrations classified by $X\in\dSt\mapsto \QCoh(X)^{\op}$ and $X\in\dSt\mapsto (\Mod_{\cO_X})^{\op}$.

\begin{prop}
Let $X$ be a derived affine scheme. The inclusion $\QCoh(X)\subset \Mod^{\pre}_{\cO_X}$ admits a right adjoint $\Mod^{\pre}_{\cO_X}\rightarrow \QCoh(X)$ which sends a presheaf of $\cO_X$-modules $(S\rightarrow X)\mapsto V_S\in \QCoh(S)$ to the quasi-coherent sheaf $V_X\in\QCoh(X)$.
\label{prop:coherator}
\end{prop}
\begin{proof}
For an $\infty$-category $I$ recall the twisted arrow category $p\colon Tw(I)\rightarrow I\times I^{\op}$ (see \cite[Section 5.2.1]{HA}) whose fiber at $(i, j)\in I\times I^{\op}$ is $\Hom_I(i, j)$. Recall also the notion of the end of a bifunctor $G\colon I^{\op}\times I\rightarrow \cC$ from \cite{GepnerHaugsengNikolaus} given by
\[\int_{i\in I} G(i, i) = \lim_{Tw(I)^{\op}} G\circ p.\]

Let $\QCoh(-)\colon \dAff_{/X}^{\op}\rightarrow \mathbf{Cat}_\infty$ be the functor $S\mapsto \QCoh(S)$ classifying the coCartesian fibration $\cM_X\rightarrow \dAff_{/X}^{\op}$. We have the following alternative description of $\Mod^{\pre}_{\cO_X}$. By \cite[Proposition 7.1]{GepnerHaugsengNikolaus}
\[\Mod^{\pre}_{\cO_X} \cong \int_{S\in \dAff_{/X}} \Fun(\dAff^{\op}_{S//X}, \QCoh(S)),\]
where for a morphism $f\colon S\rightarrow X$, with $S$ is affine, $\dAff_{S//X}$ is the $\infty$-category of factorizations $S\rightarrow T\rightarrow X$, where $T$ is affine. Under this equivalence, a presheaf of $\cO_X$-modules $(S\rightarrow X)\mapsto V_S$ is sent to the functor $(S\xrightarrow{f} T\rightarrow X)\mapsto f^* V_T\in \QCoh(S)$.

For $S_1,S_2\in \dAff_{/X}$ consider the functor
\[\QCoh(S_2)\longrightarrow \Fun(\dAff^{\op}_{S_1//X}, \QCoh(S_2))\]
which sends a quasi-coherent sheaf $\cF\in \QCoh(S_2)$ to the constant diagram $(S_1\rightarrow T\rightarrow X)\mapsto \cF$. Its right adjoint 
\[\Fun(\dAff^{\op}_{S_1//X}, \QCoh(S_2))\longrightarrow \QCoh(S_2)\]
identifies with the limit functor which sends $F\colon \dAff^{\op}_{S_1//X}\rightarrow \QCoh(S_2)$ to $\lim_{T\in \dAff_{S_1//X}} F(T)$. Since $X$ is affine, $\dAff_{S_1//X}$ has a final object $S_1\rightarrow X\xrightarrow{id} X$, so the right adjoint sends $F\mapsto F(X)$. In particular, the right adjoint is compatible with transition functors in the diagram $S_1, S_2\mapsto \Fun(\dAff^{\op}_{S_1//X}, \QCoh(S_2))$. Taking the end, we get an adjunction
\[
\xymatrix{
\QCoh(X)\cong\lim_{S\in \dAff_{/X}^{\op}} \QCoh(S)\cong \int_{S\in \dAff_{/X}} \QCoh(S) \ar@<.5ex>[r] & \int_{S\in \dAff_{/X}} \Fun(\dAff^{\op}_{S//X}, \QCoh(S)), \ar@<.5ex>[l]
}
\]
where the identification on the left follows since the functor $Tw(I)\rightarrow I$ given by $(i\rightarrow j)\mapsto i$ is coinitial by \cite[Lemma A.3.6]{HaugsengMelaniSafronov} and $\dAff_{/X}$ has a final object $X\xrightarrow{id} X$. Moreover, the left adjoint coincides with the inclusion $\QCoh(X)\subset \Mod^{\pre}_{\cO_X}$ which proves the claim.
\end{proof}

\subsection{Total spaces}

In this section we introduce two versions of linear stacks and study their relationship.

\begin{df}
The \defterm{total space} is the functor
\[\Tot\colon \cQ^{\pre}_{\op}\longrightarrow \dPSt\]
given by sending $(X, V)$ to the derived prestack
\[\Tot_X(V)(S) = \Map_{\cQ^{\pre}_{\op}}((S, \cO_S), (X, V)).\]
\end{df}

\begin{df}
The \defterm{linear prestack} is the functor
\[\bV\colon \cQ^{\pre,\op}\longrightarrow \dPSt\]
given by sending $(X, V)$ to the derived prestack
\[\bV_X(V)(S) = \Map_{\cQ^{\pre}}((X, V), (S, \cO_S)).\]
\end{df}

Let us unpack these constructions for $V\in \Mod^{\pre}_{\cO_X}$:
\begin{itemize}
    \item For $S$ a derived prestack the space $\Map_{\dPSt}(S, \Tot_X(V))$ parametrizes maps $f\colon S\rightarrow X$ together with a morphism $\cO_S\rightarrow f^* V$ of presheaves of $\cO_S$-modules. If $S$ is affine, since $\cO_S$ is a quasi-coherent sheaf by \cref{prop:coherator} the latter morphism is the same as a morphism $\cO_S\rightarrow (f^* V)_S$ in $\QCoh(S)$.
    \item For $S$ a derived prestack the space $\Map_{\dPSt}(S, \bV_X(V))$ parametrizes maps $f\colon S\rightarrow X$ together with a morphism $f^* V\rightarrow \cO_S$ of presheaves of $\cO_S$-modules.
\end{itemize}

Applying the above description of the functor of points to the identity maps $\Tot_X(V)\rightarrow \Tot_X(V)$ and $\bV_X(V)\rightarrow \bV_X(V)$ we obtain the \defterm{tautological section} and the \defterm{tautological cosection}
\[s\colon \cO_{\Tot_X(V)}\longrightarrow p^* V, \qquad \tilde{s}\colon V\longrightarrow p_*\cO_{\bV_X(V)},\]
where $p\colon \Tot_X(V)\rightarrow X$ and $p\colon \bV_X(V)\rightarrow X$ are the projections.

Let us list some basic properties:
\begin{enumerate}
    \item For $V\in\QCoh(X)$ there is a natural equivalence
    \begin{equation}
    \bV_X(V)\cong \Spec \Sym(V).
    \end{equation}
    \item For $V\in\Perf(X)$ there is a natural equivalence
    \begin{equation}
    \bV_X(V)\cong \Tot_X(V^\vee),
    \end{equation}
    so that the dual $s^\vee\colon p^* V\rightarrow \cO_{\Tot_X(V^\vee)}$ of the tautological section coincides with the adjoint of the tautological cosection $s\colon V\rightarrow p_* \cO_{\bV_X(V)}$.
\end{enumerate}

Moreover, we have the following result on effective epimorphisms of total spaces, see \cite[Lemma 3.21]{SafronovPoissonLie}.

\begin{prop}
Let $f\colon X\rightarrow Y$ be an effective epimorphism of derived stacks. Let $E_X$ and $E_Y$ be quasi-coherent sheaves on $X$ and $Y$ respectively together with a morphism $E_X\rightarrow f^* E_Y$ whose homotopy fiber is connective. Then the induced morphism $\Tot_X(E_X)\rightarrow \Tot_Y(E_Y)$ is an effective epimorphism.
\label{prop:Totcovering}
\end{prop}

The total space and the linear prestack carry a canonical $\bG_m$-action. Namely, identifying derived prestacks with a $\bG_m$-action with derived prestacks over $\Bpre\bG_m$ using \cref{ex:Gmstacks}, for $V\in \Mod^{\pre}_{\cO_X}$ consider
\[\Tot_{X\times \Bpre\bG_m}(V\otimes \cO_{\Bpre\bG_m}(1))\rightarrow X\times \Bpre\bG_m,\]
Its pullback along $pt\rightarrow \Bpre\bG_m$ recovers $\Tot_X(V)$. The construction of the $\bG_m$-action on the linear prestack $\bV$ is identical. We will use the following fact shown in \cite[Lemma 1.2.2]{Park} (see also \cite[Proposition 2.11]{Monier} for the case $V$ perfect).

\begin{prop}\label{prop:functionslinearstack}
Let $X$ be a derived prestack and $V\in\QCoh(X)^-$ a bounded above quasi-coherent complex. Then the tautological cosection $V(1)\rightarrow \Gamma^{\gr}_X(\bV_X(V), \cO_{\bV_X(V)})$ induces an isomorphism
\[\Gamma^{\gr}_X(\bV_X(V), \cO_{\bV_X(V)})\cong \Sym(V(1))\in\CAlg(\QCoh(X)^{\gr}).\]
\end{prop}

The constructions $\Tot$ and $\bV$ often produce derived stacks as shown in the following statement.

\begin{prop}\label{prop:Vstack}
Suppose $X$ is a derived stack and $V\in \Mod_{\cO_X}$. Then $\Tot_X(V)$ is a derived stack. Similarly, if $X$ is a derived stack and $V\in\QCoh(X)$, then $\bV_X(V)$ is a derived stack.
\end{prop}
\begin{proof}
The two statements are proven analogously, so let us just show the statement for $\Tot$. Consider an \'{e}tale cover $p\colon T\rightarrow S$ of derived affine schemes and let $T_\bullet$ be its \v{C}ech nerve. There is a natural forgetful map $\Tot_X(V)\rightarrow X$. By functoriality we have a commutative diagram
\[
\xymatrix{
\Tot_X(V)(S) \ar[r] \ar[d] & \lim_{[n]\in\Delta} \Tot_X(V)(T_n) \ar[d] \\
X(S) \ar[r] & \lim_{[n]\in\Delta} X(T_n).
}
\]

Since $V$ satisfies \'{e}tale descent, this diagram is Cartesian. Since $X$ satisfies \'{e}tale descent, the bottom map is an equivalence. Therefore, the top map is an equivalence.
\end{proof}

The previous proposition implies that $\Tot$ and $\bV$ restrict to functors
\[\Tot\colon \cQ^{\desc}_{\op}\longrightarrow \dSt,\qquad \bV\colon \cQ^{\op}\longrightarrow \dSt.\]

If $V$ is perfect, the total space has the following alternative description.

\begin{prop}
Let $X$ be a derived prestack and $V\in\Perf(X)$ a perfect complex on $X$. Then:
\begin{enumerate}
    \item The tautological section is the morphism adjoint to the composite
    \[\cO_X\xrightarrow{coev} V^\vee\otimes V\xrightarrow{\tilde{s}\otimes id} (p_*p^*\cO_{\Tot_X(V)})\otimes V\rightarrow p_*p^* V.\]
    \item If $X$ admits a perfect cotangent complex, so does $\Tot_X(V)$.
    \item If $X$ is a derived Artin stack, so is $\Tot_X(V)$.
\end{enumerate}
\label{prop:TotPerf}
\end{prop}
\begin{proof}
Let $\eta\colon id\rightarrow p_*p^*$ and $\epsilon\colon p^*p_*\rightarrow id$ be the unit and counit of the adjunction. The first claim follows by examining the commutative diagram
\[
\xymatrix{
\cO_{\Tot_X(V)} \ar^-{coev}[r] & p^* V^\vee\otimes p^* V \ar^-{\eta\otimes id}[r] \ar@/_1.5pc/_{id}[drr] & p^*p_*p^* V^\vee \otimes p^* V \ar^-{s^\vee\otimes id}[r] \ar_{\epsilon\otimes id}[dr] & p^*p_*\cO_{\Tot_X(V)}\otimes p^* V \ar^-{\epsilon\otimes id}[r] & p^* V \\
&&& p^* V^\vee\otimes p^* V \ar_{s^\vee\otimes id}[ur] &
}
\]
and using the fact that the composite $\cO_{Tot_X(V)}\xrightarrow{coev} p^* V^\vee\otimes p^* V\xrightarrow{s^\vee\otimes id} p^* V$ coincides with $s$ by the definition of the dual.

The second claim follows from \cite[Proposition 2.25]{Grataloup}.

We will prove the third claim by induction. If $X$ is a derived affine scheme (i.e. $X$ is $(-1)$-geometric), the claim follows from \cite[Sub-lemma 3.9]{ToenVaquie}. Next, suppose the claim holds for any $n$-geometric stack $X$. If $Y$ is an $(n+1)$-geometric stack, there is an atlas $\{f_i\colon U_i\rightarrow Y\}$, where $U_i$ is $n$-geometric and $\coprod U_i\rightarrow Y$ is an effective epimorphism. Consider a Cartesian square
\[
\xymatrix{
\Tot_{U_i}(f^* V) \ar[r] \ar[d] & \Tot_Y(V) \ar[d] \\
U_i \ar^{f}[r] & Y
}
\]
By the induction assumption $\Tot_{U_i}(f^*V)$ is a derived Artin stack. Therefore, by \cite[Proposition 1.3.3.4]{HAGII} $\Tot_Y(V)$ is a derived Artin stack.
\end{proof}

\subsection{Functoriality of the cotangent complex}\label{sect:cotangentfunctoriality}

In this section we recall the functoriality of the cotangent complex of derived affine schemes and derived stacks.

\begin{prop}
There is a sheaf of $\cO_{pt}$-modules $\bL\in \Mod_{\cO_{pt}}$ whose value on a derived affine scheme $T$ is $\bL_T\in \QCoh(T)$ and whose value on a morphism $f\colon T_1\rightarrow T_2$ of derived affine schemes is $f^*\bL_{T_2}\rightarrow \bL_{T_1}$.
\end{prop}
\begin{proof}
This statement is essentially contained in the general formalism for the cotangent complex developed in \cite[Section 7.3]{HA}; let us provide the details.

Let $\cM^c$ be the $\infty$-category of pairs of a connective cdga $A$ and a connective $A$-module $M$. Let $p\colon \cM^c\rightarrow \CAlg^c_\bK$ be the forgetful functor. There is a functor $\cM^c\rightarrow \CAlg^c_\bK$ which sends $(A, M)$ to the square-zero extension $A\oplus M$. By \cite[Remark 7.3.2.15]{HA} its left adjoint is the cotangent complex functor $\bL\colon \CAlg^c_\bK\rightarrow \cM^c$ which satisfies $p\circ \bL \cong id$. Thus, $\bL$ defines a section of the coCartesian fibration $p\colon \cM\rightarrow \CAlg^c_\bK$, i.e. a presheaf of $\cO_{pt}$-modules.

Let us now show that $\bL$ satisfies \'etale descent. Suppose $p\colon U\rightarrow T$ is an \'etale cover and let $U_\bullet$ be the \v{C}ech nerve of $p$. All maps in $U_\bullet$ are \'etale as well, so the map
\[
\bL_T\longrightarrow \lim_{[n]\in\Delta} \bL_{U_n}
\]
reduces to the map
\[\bL_T\longrightarrow \lim_{[n]\in\Delta} p_n^* \bL_T,\]
where $p_n\colon U_n\rightarrow T$. This map is an equivalence due to faithfully flat descent.
\end{proof}

\begin{cor}
Let $S$ be a derived prestack which admits a cotangent complex. Then there is a sheaf of $\cO_S$-modules $\bL_{/S}\in \Mod_{\cO_S}$ whose value on a derived affine scheme $T$ is $\bL_{T/S}\in \QCoh(T)$ and whose value on a morphism $f\colon T_1\rightarrow T_2$ of derived affine schemes is $f^*\bL_{T_2/S}\rightarrow \bL_{T_1/S}$.
\end{cor}
\begin{proof}
The cotangent complex $\bL_S\in\QCoh(S)$ of $S$ admits a natural morphism $\bL_S\rightarrow p^*\bL$ in $\Mod_{\cO_S}$, where $p\colon S\rightarrow pt$. If we define its cofiber to be $\bL_{/S}\in\Mod_{\cO_S}$, then it satisfies the required properties.
\end{proof}

Let us now establish a similar functoriality of the cotangent complex of derived prestacks.

\begin{prop}
Let $S$ be a derived prestack. There is a functor
\[\bL\colon \dPSt_S^{\bL,\op}\longrightarrow \cQ^{\pre}\]
which sends $X$ to $(X, \bL_{X/S})$ and $f\colon X\rightarrow Y$ to $(Y,\bL_{Y/S})\rightarrow (X,\bL_{X/S})$. Moreover:
\begin{itemize}
    \item $\bL$ restricts to a functor
    \[\bL\colon \dSt_S^{\bL,\op}\longrightarrow \cQ.\]
    \item If $S$ admits a cotangent complex, $\bL$ defines functors
    \[\bL\colon \dPSt_S^{\bL, \op}\longrightarrow \cQ^{\pre}_{(\bL_{/S})/},\qquad \bL\colon \dSt_S^{\bL, \op}\longrightarrow \cQ_{(\bL_{/S})/}.\]
\end{itemize}
\label{prop:cotangentcomplexfunctoriality}
\end{prop}
\begin{proof}
Denote by $\Fun(\Delta^1, \dPSt)^\bL\subset \Fun(\Delta^1, \dPSt)$ the full subcategory of morphisms $X\rightarrow Y$ of derived prestacks which admit relative cotangent complex. Let $ev_0\colon \Fun(\Delta^1, \dPSt)\rightarrow \dPSt$ be the functor which sends $(X\rightarrow Y)\mapsto X$. Then by \cite[Lemma B.10.13]{CHS} we have a commutative diagram
\[
\xymatrix{
\Fun(\Delta^1, \dPSt)^{\bL, \op} \ar[rr] \ar_{ev_0}[dr] && \cQ^{\pre} \ar[dl] \\
& \dPSt^{\op}, &
}
\]
where the horizontal functor sends $(X\rightarrow Y)$ to $(X, \bL_{X/Y})$. Precomposing it with $i\colon \dPSt_S^\bL\rightarrow \Fun(\Delta^1, \dPSt)^{\bL}$ given by sending a derived $S$-prestack $X$ to $X\rightarrow S$, we get the first claim. Since $\bL_{X/S}$ is quasi-coherent, we get that $\bL$ restricts to a functor $\dSt_S^{\bL,\op}\rightarrow \cQ$.

Finally, if $S$ admits a cotangent complex, the functor $i\colon \dPSt^\bL\rightarrow \Fun(\Delta^1, \dPSt)^\bL$ factors through $\dPSt^\bL\rightarrow\Fun(\Delta^1, \dPSt)^\bL_{/(S\rightarrow pt)}$. Applying the cotangent complex gives the functor $\dPSt_S^{\bL, \op}\longrightarrow \cQ^{\pre}_{(\bL_{/S})/}$.
\end{proof}

For $n\in\Z$ we define the functor $T[n](-/S)$, the \defterm{$n$-shifted tangent bundle}, by the composite
\[T[n](-/S)\colon \dPSt^{\bL}_S\xrightarrow{\bL[-n]}(\cQ^{\pre})^{\op}\xrightarrow{\bV} \dPSt_S.\]
By \cref{prop:Vstack} and \cref{prop:TotPerf} it restricts to functors
\[T[n](-/S)\colon \dSt^{\bL}_S\longrightarrow \dSt_S,\qquad T[n](-/S)\colon \dArt\longrightarrow \dArt.\]

For a cospan $X\rightarrow Z\leftarrow Y$ the natural morphism
\[(T[n]X)\times_{T[n] Z} (T[n] Y)\longrightarrow T[n](X\times_Z Y)\]
is an isomorphism, so these functors preserve fiber products.

For us the shifted tangent bundle will arise using the following construction. Recall that if $S$ is a derived prestack and $C, X$ are derived $S$-prestacks, the \defterm{mapping prestack} $\bMap_S(C, X)$ satisfies the universal property
\[\Map_{\dPSt_S}(T, \bMap_S(C, X))\cong \Map_{\dPSt_S}(T\times_S C, X).\]

\begin{prop}\label{prop:maptangent}
Let $f\colon X\rightarrow S$ be a morphism of derived prestacks which admits a relative cotangent complex. Then there is a natural isomorphism
\[\bMap_S(\bA^1[-1]\times S, X)\cong T[1](X/S).\]
\end{prop}
\begin{proof}
For a derived affine scheme $T$ the space $\bMap_S(\bA^1[-1]\times S, X)(T)$ parametrizes morphisms $h\colon T\rightarrow S$ together with a morphism $\tilde{g}\colon \bA^1[-1]\times T\rightarrow X$ fitting into a commutative diagram
\[
\xymatrix{
\bA^1[-1]\times T \ar^{\tilde{g}}[rr] \ar_{\pi_{\bA^1[-1]}\times h}[dr] && X \ar^{f}[dl] \\
& S &
}
\]
Since $\bA^1[-1]\times T$ is the split square-zero extension $T[\cO_T[1]]$ of $T$ by $\cO_T[1]$, using the definition of the relative cotangent complex we get that $\bMap_S(\bA^1[-1]\times S, X)(T)$ parametrizes morphisms $g\colon T\rightarrow X$ together with a morphism $g^*\bL_{X/S}\rightarrow \cO_T[1]$. Similarly, by definition of the linear prestack, the space $T[1](X/S)(T)$ parametrizes morphisms $g\colon T\rightarrow X$ together with a morphism $g^*\bL_{X/S}[-1]\rightarrow \cO_T$.
\end{proof}

\begin{rmk}
For $S=pt$ the previous statement is shown in \cite[Proposition 1.4.1.6]{HAGII}.
\end{rmk}

\begin{cor}\label{cor:tangentcolimit}
Let $S$ be a derived affine scheme and $I$ a small $\infty$-category. Consider a diagram $X_\bullet\colon I\rightarrow \dPSt^{\bL}_S$ of derived prestacks which admit cotangent complexes and suppose that the colimit
\[X=\colim_{i\in I} X_i\in\dPSt_S\]
computed in $\dPSt_S$ admits a cotangent complex. Then the natural morphism
\[\colim_i T[1](X_i/S)\longrightarrow T[1](X/S),\]
where the colimit on the left is computed in $\bG_m-\dPSt_S$ is an isomorphism.
\end{cor}
\begin{proof}
First, the forgetful functor $\bG_m-\dPSt_S\rightarrow \dPSt_S$ is conservative. Next, we claim that this forgetful functor preserves colimits. Indeed, under the equivalence $\bG_m-\dPSt_S\cong \dPSt_{/S\times \Bpre\bG_m}$ it corresponds to the functor $\dPSt_{/S\times \Bpre\bG_m}\rightarrow \dPSt_S$ given by $Y\mapsto pt\times_{\Bpre\bG_m}Y$ which preserves colimits since they are universal in $\dPSt$. Therefore, it is enough to show that the natural morphism
\[\colim_i T[1](X_i/S)\longrightarrow T[1](X/S)\]
is an isomorphism, where the colimit on the left is computed in $\dPSt_S$.

By \cref{prop:maptangent} we have to show the same claim for the functor $\bMap_S(\bA^1[-1]\times S, -)\colon \dPSt_S\rightarrow \dPSt_S$. Colimits in $\dPSt_S$ are computed pointwise, so it is enough to show that
\[\colim_i \Map_S(\bA^1[-1]\times T, X_i)\longrightarrow \Map_S(\bA^1[-1]\times T, X),\]
where the colimit on the left is computed in $\cS$, is an isomorphism for every derived affine scheme $T$ equipped with a morphism to $S$. By definition
\[\Map_S(\bA^1[-1]\times T, X_i)\cong X_i(\bA^1[-1]\times T)\times_{S(T\times \bA^1[-1])} pt.\]
Since colimits in $\cS$ are universal, the natural morphism
\[\colim_i (X_i(\bA^1[-1]\times T)\times_{S(T\times \bA^1[-1])} pt)\longrightarrow X(\bA^1[-1]\times T)\times_{S(T\times \bA^1[-1])} pt)\]
is an isomorphism.
\end{proof}

\section{Shifted symplectic and Lagrangian structures}

\subsection{Shifted presymplectic and isotropic structures}

In this section we recall and slightly generalize the theory of shifted symplectic structures from \cite{PTVV,CPTVV,CHS}.

Let $A$ be a commutative dg algebra (cdga). Denote by $\Mod^{\gre}_A$ the $\infty$-category of graded mixed $A$-modules, i.e. a family of $A$-modules $E(p)$ for $p\in\Z$ together with a square-zero mixed structure $\epsilon\colon E(p)\rightarrow E(p+1)[1]$. The functor $\Mod^{\gre}_A\to \Mod^{\gr}_A$ that forgets the mixed structure $\epsilon$ is denoted $E\mapsto E^{\not\epsilon}$. We also denote by $A(p)[n]\in\Mod^{\gre}_A$ the graded mixed $A$-module $A[n]$ concentrated purely in weight $p$. For a derived prestack $X$ we denote by $\QCoh(X)^{\gre}$ the $\infty$-category of quasi-coherent sheaves equipped with a graded mixed structure and the $\infty$-category $\Perf(X)^{\gre}$ of graded mixed perfect complexes.

Denote by $\CAlg^{\gre}_A=\CAlg(\Mod^{\gre}_A)$ the $\infty$-category of graded mixed $A$-linear commutative dg algebras. Let $A\rightarrow B$ be a morphism of commutative differential graded algebras. Recall from \cite{PTVV,CPTVV} the graded mixed cdga
\[\DR(B/A)\in\CAlg^{\gre}_A,\]
i.e. the graded algebra of differential forms 
$\Sym_B(\bL_{B/A}[-1])$ equipped with the de Rham differential $\ddr$.

This definition extends to arbitrary derived prestacks as follows (see \cite[Section B.11]{CHS} and \cite{Park} for details):
\begin{itemize}
\item For a morphism $X\rightarrow S=\Spec A$ of derived prestacks, $\DR_S(X)$ is defined as a right Kan extension from its restriction to derived affine schemes. Explicitly, write $X$ as a colimit of affine prestacks $X_i=\Spec B_i$. Then
\[\DR_S(X) = \lim_i \DR_A(B_i)\in\CAlg^{\gre}_A.\]
\item Given a morphism of arbitrary derived prestacks $X\rightarrow S$ define
\[\DR_S(X) = \lim_{f\colon S'\rightarrow S} f_*(\DR_{S'}(X'))\in\CAlg(\QCoh(S)^{\gre}),\]
where the limit is taken over derived affine schemes mapping to $S$ and $X'=X\times_S S'$ is the base change.
\end{itemize}

By construction, given a commutative diagram
\[
\xymatrix{
X' \ar^{f'}[r] \ar^{\pi'}[d] & X \ar^{\pi}[d] \\
S' \ar^{f}[r] & S
}
\]
of derived prestacks, there is a natural pullback morphism
\[f^*\DR_S(X)\longrightarrow \DR_{S'}(X')\]
in $\CAlg(\QCoh(S')^{\gre})$. In particular, given morphisms $X\xrightarrow{f} Y$ and $Y\xrightarrow{g} S$ of derived prestacks by adjunction we obtain a morphism
\[\DR_S(X)\longrightarrow g_* \DR_Y(X).\]

\begin{df}\label{df:closedforms}
Let $\pi\colon X\rightarrow S$ be a morphism of derived prestacks and $\cL$ a graded line bundle on $S$.
\begin{itemize}
\item The space of \defterm{$\cL$-twisted $p$-forms on $X/S$} is
\[\cA^p_S(X; \cL) = \Map_{\QCoh(S)^{\gr}}(\cO_S(p)[-p], \DR_S(X)^{\not\epsilon}\otimes \cL).\]
\item The space of \defterm{$\cL$-twisted closed $p$-forms on $X/S$} is
\[\cA^{p, \cl}_S(X; \cL) = \Map_{\QCoh(S)^{\gre}}(\cO_S(p)[-p], \DR_S(X)\otimes \cL).\]
If $p=2$, we also refer to it as the space of \defterm{$\cL$-twisted presymplectic structures}.
\end{itemize}
\end{df}

The space $\cA^{p, \cl}_S(-;\cL)$ forms a derived prestack over $S$ that we denote by $\cA^{p, \cl}_S(\cL)$. An $\cL$-twisted closed $p$-form on $X/S$ is then equivalently a morphism $\omega\colon X\rightarrow \cA^{p, \cl}_S(\cL)$ of derived prestacks over $S$. For a pair $(X, \omega)$ we denote by $\overline{X}$ the pair $(X, -\omega)$.

\begin{rmk}
One is often interested in the case $\cL=\cO_S[n]$. We denote $\cA^{p, \cl}_S(\cO_S[n]) = \cA^{p, \cL}(n)$. In this case $\cO_S[n]$-twisted closed $p$-forms are exactly (relative) $n$-shifted closed $p$-forms as defined in \cite{PTVV,CHS}.
\end{rmk}

\begin{rmk}
The reference \cite{BouazizGrojnowski} considers an even more general case when $\cL$ is an $S$-relative $D$-module on $X$. Our twists correspond to relative $D$-modules pulled back from line bundles on $S$.
\end{rmk}

If $f\colon Y\to X$ is a map of derived $S$-prestacks and $\omega\colon X\to \Acl_S(\cL)$, we write $f^*\omega:=\omega\circ f$ for the induced $\cL$-twisted presymplectic structure on $Y$.

\begin{df}
An \defterm{$\cL$-twisted isotropic structure} on a map $L\to X$ of derived $S$-prestacks is a commuting square 
\[
\xymatrix{
L \ar[r]^-{f} \ar[d] & X \ar[d]^-{\omega_X} \\
S \ar[r]^-{0} & \Acl_S(\cL)
}
\]
\end{df}

\begin{rmk}\label{rmk-loops-for-forms}
There is a canonical identification between the based loop space of $\Acl_S(\cL[1])$ and $\Acl_S(\cL)$. As a consequence, one observes that $\cL[1]$-twisted isotropic structures on $X\to S$ are 
exactly $\cL$-twisted presymplectic structures on $X$.
\end{rmk}

Isotropic structures are special cases of isotropic correspondences, that we define now. 
\begin{df}
An \defterm{$\cL$-twisted isotropic correspondence} from a derived prestack $X\rightarrow S$ to a derived prestack $Y\rightarrow S$ is a commuting diagram 
\begin{equation}\label{isocorr}
\vcenter{
\xymatrix{
L \ar[r] \ar[d] & Y \ar[d]^-{\omega_Y} \\
X \ar[r]^-{\omega_X} & \Acl_S(\cL)
}
}
\end{equation}
of derived prestacks over $S$.
\end{df}
In other words, $\cL$-twisted isotropic correspondences are $\infty$-functors from $D$ to $(\dPSt_S)_{/\Acl_S(\cL)}$, where
\[D=\wedge:=a\leftarrow\emptyset\rightarrow b.\] 

Using that $\Acl_S(\cL)$ is an abelian group in $\dPSt_S$, one conversely gets that isotropic correspondences are also special cases of isotropic maps. Indeed,  
an $\cL$-twisted isotropic correspondence such as \eqref{isocorr} leads to an $\cL$-twisted structure on the map $L\to X\times Y$:  
\[
\xymatrix{
L \ar[r]^{f} \ar[d] & X\times Y \ar[d]^-{-\omega_X+\omega_Y} \\
S \ar[r]^-{0} & \Acl_S(\cL)
}
\]

We will particularly be interested in the following two cases of the base prestack $S$:
\begin{itemize}
\item $S=pt=\Spec \bK$ is a point. In this case $\cL=\bK[n]$ for some $n$ and such $\cL$-twisted presymplectic structures on $X\rightarrow pt$ are exactly $n$-shifted presymplectic structures as defined in \cite{PTVV}.
\item Let $S=\B\bG_m$ and $\cL=\cO_S(p)[n]$ be the line bundle corresponding to the one-dimensional $\bG_m$-representation concentrated in cohomological degree $-n$ and weight $p$. Let $\widetilde{X}\rightarrow S$ be a derived prestack and $X=\widetilde{X}\times_{\B\bG_m} pt$ its base change, which is a derived prestack with a $\bG_m$-action. Given an $\cO_S(p)[n]$-twisted presymplectic structure on $\widetilde{X}$ we can pull it back to an $n$-shifted presymplectic structure $\omega$ on $X$. In this case we say that $\omega$ is an \defterm{$n$-shifted presymplectic structure of weight $-p$}.
\end{itemize}

\subsection{Non-degeneracy of presymplectic structures}

A general and extensive treatment of non-degeneracy can be found in \cite[Section 2.8]{CHS}; we recollect here the definitions and facts we need in the present paper. To define non-degenerate $2$-forms, we have to require our derived prestacks to admit a cotangent complex. In this case the de Rham algebra has the following description.

\begin{thm}\label{thm:DRgr}
Let $\pi\colon X\rightarrow S$ be a morphism of derived prestacks and assume $X$ admits a cotangent complex relative to $S$. Then there is a natural isomorphism
\[\DR_S(X)^{\not\epsilon}\cong \Gamma_S(X, \Sym(\bL_{X/S}[-1]))\in\CAlg(\QCoh(X)^{\gr}).\]
\end{thm}
\begin{proof}
By definition,
\[\DR_S(X) = \lim_{f\colon S'\rightarrow S} f_*(\DR_{S'}(X')),\]
where the limit is taken over derived affine schemes $S'$ and we have a Cartesian square
\[
\xymatrix{
X' \ar^{f'}[r] \ar^{\pi'}[d] & X \ar^{\pi}[d] \\
S' \ar^{f}[r] & S
}
\]
Therefore 
\[\DR_S(X)^{\not\epsilon} = \lim_{f\colon S'\rightarrow S} f_*(\DR_{S'}(X'))^{\not\epsilon},\]
since $(-)^{\not\epsilon}$ preserves limits.
Now observe that 
\[S = \colim_{f\colon S'\rightarrow S} S',\]
where the colimit is taken in the $\infty$-category of derived prestacks over derived affine schemes $S'$ mapping to $S$. Colimits in derived prestacks are stable under base change, so we have
\[X = \colim_{f\colon S'\rightarrow S} X'.\]
Therefore, by \cite[Corollary B.5.6]{CHS} we have
\[\Sym(\bL_{X/S}[-1])\cong \lim_{f\colon S'\rightarrow S} f'_*f'^* \Sym(\bL_{X/S}[-1])\]
in $\QCoh(X)^{\gr}$. From this we obtain a natural isomorphism
\begin{align*}
\pi_*\Sym(\bL_{X/S}[-1])&\cong \pi_* \lim_{f\colon S'\rightarrow S} f'_*(f')^* \Sym(\bL_{X/S}[-1]) \\
&\cong \lim_{f\colon S'\rightarrow S} f_* \pi'_* \Sym(\bL_{X'/S'}[-1]),
\end{align*}
where we have used that $(f')^*\bL_{X/S}\cong \bL_{X'/S'}$ since the above square is Cartesian. As a result, the claim is reduced to the case $S$ is a derived affine scheme, which we assume from now on.

By \cref{cor:tangentcolimit} the functor $X\in\dPSt_S\mapsto T[1](X/S)\in\bG_m-\dPSt_S$ preserves colimits. Therefore, applying $\Gamma^{\gr}_S(T[1](X/S), \cO_{T[1](X/S)})$ and using \cref{prop:functionslinearstack} we get that the functor
\[X\in\dPSt_S\mapsto \pi_*\Sym(\bL_{X/S}[-1])\in\CAlg(\QCoh(S)^{\gr})\]
sends colimits to limits. By definition of $\DR_S(X)$ we are further reduced to the case when both $X$ and $S$ are affine. But this claim is shown in \cite[Proposition 1.3.17]{CPTVV}.
\end{proof}

\begin{rmk}
\Cref{thm:DRgr} was shown for geometric morphisms in \cite[Corollary B.11.11]{CHS} (see also \cite[Proposition 1.14]{PTVV} for the case $S=pt$ and $X$ a derived Artin stack).
\end{rmk}

\begin{cor}\label{cor:pullbackDRisomorphism}
Let
\[
\xymatrix{
X' \ar^{f'}[r] \ar^{\pi'}[d] & X \ar^{\pi}[d] \\
S' \ar^{f}[r] & S
}
\]
be a Cartesian square of derived prestacks, where $X$ admits a cotangent complex relative to $S$. Then the pullback morphism
\[f^*\DR_S(X)\longrightarrow \DR_{S'}(X')\]
is an isomorphism in $\CAlg(\QCoh(S')^{\gre})$.
\end{cor}
\begin{proof}
The forgetful functor from graded mixed modules to graded modules is conservative, so it is enough to establish the claim on the level of graded modules. 
\end{proof}

Let $D$ be a small category having an initial object $\emptyset$. Any $D$-diagram $X=(X_d)_{d\in D}\colon D\rightarrow \dPSt^{\perfdef}_S$ of derived prestacks leads to a $D$-diagram $\bT_{X/S}$ in $\QCoh(X_\emptyset)$, where $(\bT_{X/S})_d$ is the pull-back of $\bT_{X_d/S}$ to $X_\emptyset$ along the unique map $X_\emptyset\rightarrow X_d$. We denote by $D^{\rhd}$ the small category obtained from $D$ by formally adding a terminal object $*$.

\begin{ex}~
\begin{enumerate}
\item If $D$ is the terminal category having a single object $\emptyset$, then $D^{\rhd}=(\emptyset\to *)=\Delta^1$. 
\item A typical example of the category $D$ we will consider is
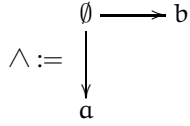
\[
\wedge:=\vcenter{\xymatrix{
\emptyset \ar[d] \ar[r] & b \\
a &
}}
\]
Then $D^{\rhd}=\wedge^{\rhd}$ is the commuting square category
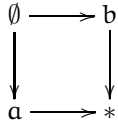
\[
\xymatrix{
\emptyset \ar[d] \ar[r] & b \ar[d] \\
a \ar[r] & \ast
}
\]
\item Another example of $D$ is the two commuting squares category 
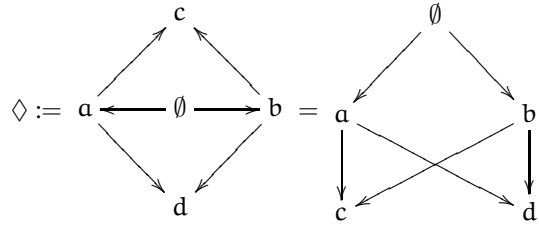
\[
\Diamond:=\vcenter{\xymatrix{
& c & \\
a \ar[ur] \ar[dr] & \emptyset \ar[l] \ar[r] & b \ar[dl] \ar[ul] \\
& d &
}}
= \vcenter{\xymatrix{
& \emptyset\ar[rd]\ar[ld] & \\
a \ar[d]\ar[rrd] & & b \ar[d]\ar[lld] \\
c & & d
}}
\]
Then $D^{\rhd}=\Diamond^{\rhd}$ is the four commuting squares category 
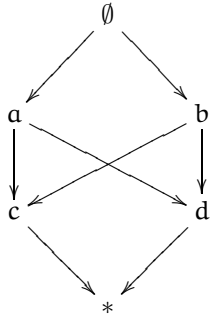
\[
\xymatrix{
& \emptyset\ar[rd]\ar[ld] & \\
a \ar[d]\ar[rrd] & & b \ar[d]\ar[lld] \\
c\ar[rd]  & & d\ar[ld] \\
& \ast & 
}
\]
\end{enumerate}
\end{ex}

We have a forgetful map $\Acl_S(\cL)\rightarrow \cA^2_S(\cL)$. In particular, thanks to \cref{thm:DRgr}, any $\cL$-twisted presymplectic structure $\omega\colon X\rightarrow \Acl_S(\cL)$ on a derived prestack $X\in\dPSt_S$ with a perfect cotangent complex leads to a morphism
\[\bT_{X/S}\longrightarrow \bL_{X/S}\otimes\pi^*\cL.\]

Hence for any $D$-diagram in $(\dPSt_S)_{/\Acl_S(\cL)}$, where each derived prestack admits a perfect cotangent complex relative to $S$, one may extend the $D$-diagram $\bT_{X/S}$ to a $D^{\rhd}$-diagram where, formally, $(\bT_{X/S})_*:= \bL_{X_\emptyset/S}\otimes\pi^*\cL$.

\begin{df}\label{def:nondegenerate}
A diagram $X\colon D\rightarrow (\dPSt_S)_{/\Acl_S(\cL)}$ of derived prestacks is \defterm{non-degenerate} if $X_d$ admits a perfect cotangent complex relative to $S$ for every $d\in D$ and the $D^{\rhd}$-diagram $\bT_{X/S}$ in $\QCoh(X_\emptyset)$ is a colimit diagram.
In other words, it is non-degenerate if the morphism $\colim_D(\bT_{X/S})\to \bL_{X_\emptyset/S}\otimes\pi^*\cL$ is an equivalence. 
\end{df}

\begin{ex}\label{examples-of-D}~
\begin{enumerate}
\item If $D=\cdot$, then the non-degeneracy condition is that $\bT_{X/S}\rightarrow \bL_{X/S}\otimes\pi^*\cL$ is a colimit diagram, which means that this map is an equivalence. This is precisely the condition that the $\cL$-twisted presymplectic structure on $X$ is non-degenerate (or $\cL$-twisted symplectic).

\item If $D=\wedge$, then one has an $\cL$-twisted isotropic correspondence $X\colon \wedge\to (\dPSt_S)_{/\Acl_S(\cL)}$. 
The non-degeneracy condition in the sense of \cref{def:nondegenerate} is that the diagram
\[
\xymatrix{
(\bT_{X/S})_\emptyset \ar[r] \ar[d] & (\bT_{X/S})_b \ar[d] \\
(\bT_{X/S})_a \ar[r] & \bL_{L/S}\otimes \pi^*\cL
}
\]
is coCartesian, which is precisely the condition that we have a non-degenerate $\cL$-twisted isotropic correspondence in the sense of \cite{CalaqueCotangent}. If, in addition, both $X_a$ and $X_b$ are $\cL$-twisted symplectic, then this is an $\cL$-twisted Lagrangian correspondence in the sense of \cite{PTVV, CalaqueTFT}.

\item If $D=\Diamond$, then non-degeneracy condition of a diagram $X\colon \Diamond\to (\dPSt_S)_{/\Acl_S(\cL)}$ in the sense of \cref{def:nondegenerate} 
is equivalent to the non-degeneracy of the $\cL[-1]$-twisted isotropic correspondence 
\begin{equation}\label{a-nice-isotropic-map}
\vcenter{
\xymatrix{
X_\emptyset \ar[r] \ar[d] & X_a\underset{X_c\times X_d}{\times}X_b \ar[d] \\
S \ar[r] & \Acl_S(\cL[-1])
}
}
\end{equation}
\begin{proof}
We first observe that $\colim_D\bT_{X/S}$ can be computed in two steps: 
\begin{itemize}
\item One first takes the colimit of $V$ of $(\bT_{X/S})_a\leftarrow (\bT_{X/S})_\emptyset\rightarrow (\bT_{X/S})_b$. 
There are two morphisms $V\to (\bT_{X/S})_c\oplus (\bT_{X/S})_d$ (exhibiting the commutativity of the two squares in $\bT_{X/S}$). 
\item One then takes the coequalizer of these two maps, that is, the cofiber of their difference. 
\end{itemize}
We therefore have a null-homotopic sequence $V\to (\bT_{X/S})_c\oplus (\bT_{X/S})_d \to \bL_{X_\emptyset/S}\otimes\pi^*\cL$, and the non-degeneracy of $X$ 
is equivalent to the requirement that this is a cofiber sequence.  
Note that $V\simeq \bT_{X_\emptyset/X_a\times X_b}[1]$. 
The $\cL[-1]$-twisted isotropic structure on $X_\emptyset\to X_a\underset{X_c\times X_d}{\times}X_b$ gives us another null-homotopic sequence 
$(\bT_{X/S})_\emptyset \to (\bT_{X/S})_a\underset{(\bT_{X/S})_c\times (\bT_{X/S})_d}{\times}(\bT_{X/S})_b\to \bL_{X_\emptyset/S}\otimes\pi^*\cL[-1]$, and the non-degeneracy 
of \eqref{a-nice-isotropic-map} is equivalent to the requirement that this is a cofiber sequence. 
Finally, observe that there is a morphism of null-homotopic sequences 
\begin{equation}\label{no-idea-for-a-name}
\vcenter{
\xymatrix{
\bT_{X_\emptyset/X_a\times X_b}[1] \ar[r] \ar[d] & (\bT_{X/S})_c\oplus (\bT_{X/S})_d \ar[r] \ar[d] & \bL_{X_\emptyset/S}\otimes\pi^*\cL \ar[d] \\
(\bT_{X/S})_\emptyset[1] \ar[r] & \big((\bT_{X/S})_a\underset{(\bT_{X/S})_c\times (\bT_{X/S})_d}{\times}(\bT_{X/S})_b\big)[1] \ar[r] & \bL_{X_\emptyset/S}\otimes\pi^*\cL
}
}
\end{equation}
whose cofiber is, up to a shift, $(\bT_{X/S})_a\times (\bT_{X/S})_b \overset{id}\to (\bT_{X/S})_a\times (\bT_{X/S})_b\to 0$, which is itself a cofiber sequence. 
Hence, the first line in \eqref{no-idea-for-a-name} is a cofiber sequence if and only if the second line is. 
\end{proof}
\end{enumerate}
\end{ex}

\begin{df}
An \defterm{$\cL$-twisted Lagrangian correspondence} is a non-degenerate $\cL$-twisted isotropic correspondence $X\leftarrow L\rightarrow Y$ such that the $\cL$-twisted presymplectic structures on $X$ and $Y$ are non-degenerate. An \defterm{$\cL$-twisted Lagrangian morphism} $L\to X$ is an $\cL$-twisted Lagrangian correspondence $S\leftarrow L\rightarrow X$.
\end{df}

\begin{prop}\label{prop:epimorphismnondegenerate}
Let $f\colon L\rightarrow X$ be an effective epimorphism in $\dSt_S$ equipped with an $\cL$-twisted isotropic structure. 
If the $\cL$-twisted isotropic structure on $f$ is non-degenerate, then the $\cL$-twisted presymplectic structure on $X$ is non-degenerate as well.
\end{prop}
\begin{proof}
Let $\pi_X\colon X\rightarrow S$ and $\pi_L\colon L\rightarrow S$ be the morphisms to $S$.

We have to prove that the morphism $\bT_{X/S}\rightarrow \bL_{X/S}\otimes\pi_X^*\cL$ induced by the $\cL$-twisted presymplectic structure on $X$ is an isomorphism. Since $f\colon L\rightarrow X$ is an effective epimorphism, by \cref{prop:coveringconservative} the pullback $f^*\colon \QCoh(X)\rightarrow \QCoh(L)$ is conservative. Therefore, it is enough to prove that the morphism $f^*\bT_{X/S}\rightarrow f^*\bL_{X/S}\otimes\pi_L^*\cL$ is an isomorphism.

This is exactly \cite[Lemma 1.3]{CalaqueCotangent}, but we repeat the argument here for the convenience of the reader. The $\cL$-twisted isotropic structure induces morphisms of fiber sequences
\[
\xymatrix{
\bT_{L/S} \ar[r] \ar[d] & f^*\bT_{X/S} \ar[r] \ar[d] & \bT_{L/X}[1] \ar[d] \\
\bL_{L/X}\otimes\pi_L^*\cL[-1] \ar[r] & f^*\bL_{X/S}\otimes\pi_L^*\cL \ar[r] & \bL_{L/S}\otimes \pi_L^*\cL
}
\]
where the outer vertical morphisms are dual to each other. If the $\cL$-twisted isotropic structure is non-degenerate, the outer vertical morphisms are isomorphisms and hence the middle morphism is an isomorphism as well.
\end{proof}

We also have a relative version of the previous claim.

\begin{prop}\label{prop:epimorphismrelativenondegenerate}
Consider a diagram $X\colon \Diamond \to (\dSt_S)_{/\Acl_S(\cL)}$ such that $X_d=S$ and $X_\emptyset\to X_a$ is an effective epimorphism. 
If both $X$ and the $\cL$-twisted isotropic structure on $X_b\rightarrow X_c$ are non-degenerate, 
then the $\cL$-twisted isotropic structure on $X_a\rightarrow X_c$ is non-degenerate as well. 
\end{prop}
\begin{proof}
We have to prove that the null-homotopic sequence $\bT_{X_a/S}\to X_{a\to c}^*\bT_{X_c/S}\to \bL_{X_a/S}\otimes\pi^*\cL$ is a cofiber sequence. 
Since $X_{\emptyset\to a}^*$ is exact and conservative (by \cref{prop:coveringconservative}), then it is enough to prove that $(\bT_{X/S})_a\to (\bT_{X/S})_c\to X_{\emptyset\to a}^*\bL_{X_a/S}\otimes\pi^*\cL$ 
is a cofiber sequence. We have the following commuting diagram: 
\[
\xymatrix{
f^*\bT_{X_a\underset{X_c}{\times}X_b/S} \ar[r]\ar[d] & (\bT_{X/S})_a \ar[r]\ar[d] & 0 \ar[d] \\
(\bT_{X/S})_b \ar[r]\ar[d] & (\bT_{X/S})_c \ar[r]\ar[d] & X_{\emptyset\to a}^*\bL_{X_a/S}\otimes\pi^*\cL \ar[d] \\
0 \ar[r] & X_{\emptyset\to b}^*\bL_{X_b/S}\otimes\pi^*\cL \ar[r] & f^*\bL_{X_a\underset{X_c}{\times}X_b/S}\otimes\pi^*\cL
}
\]
where $f\colon X_\emptyset\to X_a\underset{X_c}{\times}X_b$ is induced by $X_{\emptyset\to a}$ and $X_{\emptyset\to b}$. 
We now recollect what we know: 
\begin{itemize}
\item By assumption, $X$ is non-degenerate, which translates into the fact that the $\cL[-1]$-twisted isotropic structure on $f$ is non-degenerate (this is proven in 
the number (3) among \cref{examples-of-D}). This implies in particular that the big square is (co)Cartesian. In other words, this implies that the induced map 
$f^*\bT_{X_a\underset{X_c}{\times}X_b/S}\to f^*\bL_{X_a\underset{X_c}{\times}X_b/S}\otimes\pi^*\cL[-1]$ is an equivalence (this is \cite[Lemma 1.3]{CalaqueCotangent}, that we reproved in \cref{prop:epimorphismnondegenerate}). 
\item The upper left square is (co)Cartesian (the tangent to a fiber product is the fiber product of the tangents). 
\item The lower right square is (co)Cartesian, as it is dual to the previous one. Indeed, the non-degeneracy of the $\cL$-twisted isotropic structure on $X_{b\to c}$ 
implies in particular that $(\bT_{X/S})_c\simeq X_{\emptyset\to c}^*\bL_{X_c/S}\otimes\pi^*\cL$ (this is again \cite[Lemma 1.3]{CalaqueCotangent}). 
\item The lower left square is (co)Cartesian because the $\cL$-twisted isotropic structure on $X_{b\to c}$ is non-degenerate. 
\end{itemize}
Therefore, the upper right square is (co)Cartesian meaning that $(\bT_{X/S})_a\to (\bT_{X/S})_c\to X_{\emptyset\to a}^*\bL_{X_a/S}\otimes\pi^*\cL$ is a cofiber sequence. Using the fact that $\QCoh(X_a)\rightarrow \QCoh(X_\varnothing)$ is conservative, we get the result.
\end{proof}

\begin{rmk}
The only purpose of the requirement that $X_d=S$ in the above proposition is to make the notation lighter. 
It is not restrictive at all, as we can reduce any diagram $X$ to this case by using duality 
(replacing $X$ with $S$, and $X_c$ with $X_c\times_S \overline{X_d}$). 
\end{rmk}

\subsection{Orientations}

We will now define orientations in the relative context. For this let us first introduce some finiteness conditions, see \cite[Section B.10]{CHS} for more details.

\begin{df}
Let $\pi\colon C\rightarrow S$ be a morphism of derived prestacks.
\begin{itemize}
\item It is \defterm{universally cocontinuous} if for every morphism $f\colon S'\rightarrow S$ from a derived affine scheme $S$ the pushforward $f'_*\colon \QCoh(C')\rightarrow \QCoh(C')$ along the basechange $C'=C\times_S S'\rightarrow S'$ preserves colimits.
\item It is \defterm{universally perfect} if for every morphism $f\colon S'\rightarrow S$ from a derived affine scheme $S$ the pushforward $f'_*\colon \QCoh(C')\rightarrow \QCoh(S')$ along the basechange $C'=C\times_S S'\rightarrow S'$ preserves perfect complexes.
\item It is \defterm{$\cO$-compact} if it is universally cocontinuous and universally perfect.
\end{itemize}
\end{df}

\begin{df}
Let $f\colon C\rightarrow S$ be a universally cocontinuous morphism of derived prestacks and $\cL$ a graded line bundle on $C$. An \defterm{$\cL$-preorientation on $C$ over $S$} is a morphism
\[[C]\colon \Gamma_S(C, \cL)\longrightarrow \cO_S.\]
We also refer to the pair $(\cL, [C])$ as a \defterm{preorientation of $C$ over $S$}.
\end{df}

Suppose $C\rightarrow S$ is $\cO$-compact and equipped with a preorientation. Then we have the following structures:
\begin{itemize}
\item Consider a morphism $\sigma\colon S'\rightarrow S$ and the pullback diagram
\[
\xymatrix{
C' \ar^{\xi}[r] \ar^{f'}[d] & C \ar^{f}[d] \\
S' \ar^{\sigma}[r] & S
}
\]
Then the composite
\[
\Gamma_{S'}(C', \xi^* \cL)\cong \sigma^* \Gamma_S(C, \cL)\xrightarrow{[C]} \cO_{S'}
\]
endows $C'$ with a preorientation over $S'$ that we denote by $\sigma^*[X]$.
\item Given a perfect complex $E\in\Perf(C)$ there is a morphism
\[\Gamma_S(C, E^\vee\otimes\cL)\longrightarrow \Gamma_S(C, E)^\vee\]
adjoint to
\[\Gamma_S(C, E)\otimes \Gamma_S(C, E^\vee\otimes\cL)\rightarrow \Gamma_S(C, E\otimes E^\vee\otimes\cL)\rightarrow \Gamma_S(C, \cL)\xrightarrow{[C]} \cO_S.\]
\end{itemize}

\begin{df}\label{df-orientation}
Suppose $f\colon C\rightarrow S$ is an $\cO$-compact morphism equipped with an $\cL$-preorientation $[C]$.
\begin{itemize}
\item It is \defterm{weakly $\cL$-oriented} if for every perfect complex $E\in\Perf(X)$ the natural morphism
\[\Gamma_S(C, E^\vee\otimes \cL)\longrightarrow \Gamma_S(C, E)^\vee\]
is an isomorphism.
\item It is \defterm{$\cL$-oriented} if for every morphism of derived prestacks $\sigma\colon S'\rightarrow S$ the pullback $(C\times_S S', \sigma^*[C])$ is weakly $\sigma^*\cL$-oriented.
\end{itemize}
\end{df}

\begin{rmk}\label{rmk:pushforwardleftadjoint}
If the pushforward along $f\colon C\rightarrow S$ preserves perfect complexes, then $f_*\colon \Perf(C)\rightarrow \Perf(S)$ is a right adjoint to $f^*\colon \Perf(S)\rightarrow \Perf(C)$. One then automatically has that $\Perf(C)\ni E\mapsto (f_*(E^\vee))^\vee\in \Perf(S)$ is a left adjoint to $f^*$. Therefore, a morphism $f\colon C\rightarrow S$ is weakly $\cL$-oriented precisely if $f_*((-)\otimes\cL)$ defines a left adjoint to $f^*\colon \Perf(S)\rightarrow \Perf(C)$ with $[C]$ the counit of the adjunction.
\end{rmk}

The following are two basic examples of $\cL$-oriented prestack. Recall the notion of a Gorenstein derived scheme $X$ from \cite[Section 7.3]{GaitsgoryIndCoh} (for instance, any quasi-smooth derived scheme is Gorenstein). For such a scheme we can define a graded line bundle $\cK_X$ which plays the role analogous to the canonical bundle.

\begin{prop}\label{prop:Gorensteinorientation}
Let $X$ be a proper Gorenstein derived scheme. Then $X$ has a natural $\cK_X$-orientation.
\end{prop}
\begin{proof}
Since $X$ is proper, by \cite[Proposition 7.2.9]{GaitsgoryIndCoh} the pushforward $f_*\colon \QCoh(X)\rightarrow \QCoh(pt)$ along $f\colon X\rightarrow pt$ admits a right adjoint $f^!$. Moreover, by \cite[Corollary 7.2.5]{GaitsgoryIndCoh} we have $f^!(-)\cong \cK_X\otimes f^*(-)$. By \cite[Proposition 1.4]{NaefSafronov} it implies that $X$ is $\cO$-compact.

The counit of the adjunction defines a $\cK_X$-preorientation $[X]\colon \Gamma(X, \cK_X)\rightarrow \bK$. Using \cref{rmk:pushforwardleftadjoint} we see that $X$ is weakly $\cK_X$-oriented. But since $\QCoh(X\times S')\cong \QCoh(X)\otimes \QCoh(S')$ for any derived affine scheme $S'$, the same argument implies that $X\times S'\rightarrow S'$ is also weakly $\cK_X\boxtimes \cO_{S'}$-oriented, i.e. $X$ is $\cK_X$-oriented.
\end{proof}

Let $M\in\cS$ be a finitely dominated space (i.e. $M$ is compact in $\cS$) and consider the constant stack $M_\B$ with value $M$. Let $f\colon M_\B\rightarrow pt$ be the projection. In this case there is an orientation local system $\zeta_M\in\QCoh(M_\B)$, so that $\Hom(\zeta_M, -)\colon \QCoh(M_\B)\rightarrow \QCoh(pt)$ is the left adjoint to $f^*$.

\begin{df}
A finitely dominated space $M\in\cS$ is a \defterm{Poincar\'e space} if $\zeta_M$ is invertible, i.e. it is a shifted line bundle.
\end{df}

\begin{prop}\label{prop:Poincareorientation}
Let $M\in\cS$ be a Poincar\'e space. Then $M_\B$ has a natural $\zeta_M^{-1}$-orientation.
\end{prop}
\begin{proof}
If we denote by $f_\sharp\colon \QCoh(M_\B)\rightarrow \QCoh(pt)$ the left adjoint to $f^*$, then for $V\in\Mod_\bK$ we have
\begin{align*}
\Hom_{\Mod_\bK}(p_*\cF, V)&\cong \Hom_{\Mod_\bK}(p_\sharp(\cF\otimes \zeta_M), V)\\
&\cong \Hom_{\QCoh(M_\B)}(\cF\otimes \zeta_M, V\otimes \cO_{M_\B})\\
&\cong \Hom_{\QCoh(M_\B)}(\cF, V\otimes \zeta_M^{-1}).
\end{align*}
Therefore, there is a right adjoint $f^!$ to $f_*$ given by $f^!\bK = \zeta_M^{-1}$. The rest of the proof is identical to the proof of \cref{prop:Gorensteinorientation}.
\end{proof}

\begin{ex}
If $M$ is a closed topological $d$-manifold, it is a Poincar\'e space with the fiber of $\zeta_M$ at $x\in M$ isomorphic to $H^d(M, M\setminus\{x\}; \bK)[-d]$.
\end{ex}

We also have a relative version of orientations (also known as \defterm{boundary structures}).

\begin{df}
Let $f\colon C_1\rightarrow C_2$ be a morphism of derived $S$-prestacks, where both $C_1\rightarrow S$ and $C_2\rightarrow S$ are universally cocontinuous. Let $\cL$ be a graded line bundle on $C_2$. An \defterm{$\cL$-preorientation on $f$ over $S$} is a commutative diagram
\[
\xymatrix{
\Gamma_S(C_2, \cL)\ar[r] \ar[d] & \Gamma_S(C_1, f^*\cL) \ar^{[X]}[d] \\
0 \ar[r] & \cO_S
}
\]
\end{df}

Let us now introduce the notion of non-degeneracy for preoriented cospans analogous to non-degeneracy for isotropic spans. It is convenient to introduce the spaces of preorientations:
\begin{itemize}
\item For a derived prestack $S$ consider the functor
\[\cB_S\colon \dPSt_{S\times Pic^{\gr}}\longrightarrow \cS\]
given by sending a derived prestack $C\rightarrow S$ equipped with a graded line bundle $\cL$ to the space $\Map_{\QCoh(S)}(\Gamma_S(C, \cL), \cO_S)$.
\item For a derived prestack $S$ and an integer $d$ consider the functor
\[\cB(d)\colon \dPSt_S\longrightarrow \cS\]
given by sending $C\rightarrow S$ to $\Map_{\QCoh(S)}(\Gamma_S(C, \cO_C)[d], \cO_S)$.
\end{itemize}

\begin{rmk}
The Yoneda embedding $\dPSt_{S\times Pic^{\gr}}\hookrightarrow \Fun(\dPSt_{S\times Pic^{\gr}}, \cS)^{\op}$ given by \[(C, \cL)\mapsto \Map_{\dPSt_{S\times Pic^{\gr}}}((C, \cL), -)\] allows one to regard pairs $(C, \cL)$ consisting of derived $S$-prestacks $C$ equipped with graded line bundles $\cL$ as functors $\dPSt_{S\times Pic^{\gr}}\rightarrow \cS$. Then an $\cL$-preorientation on $C\rightarrow S$ is the same as a morphism $\cB_S\rightarrow (C, \cL)$ in $\Fun(\dPSt_{S\times Pic^{\gr}}, \cS)^{op}$.
\end{rmk}

Let $D$ be a small category with an initial object $\emptyset$. Let
\[C=(C_d)_{d\in D}\colon D^{\op}\longrightarrow \dPSt_{S\times Pic^{\gr}}\]
be a $D^{\op}$-diagram of $\cO$-compact derived $S$-prestacks $C_d$ equipped with graded line bundles $\cL_d$, and $E$ is a perfect 
complex on $C_{\emptyset}$. Then we have a $D$-diagram $\Gamma_S(C, E^\vee\otimes \cL)$ where we set $\Gamma_S(C, E^\vee\otimes\cL)_d = \Gamma_S(C_d, E^\vee_d\otimes\cL_d)$ and $E^\vee_d$ 
is the pullback of $E^\vee$ from $C_\emptyset$ to $C_d$ along the unique morphism.

Let $D^{\rhd}$ be the category obtained by formally adjoining to $D$ a final object $\ast$. If $C$ is a $D^{\op}$-diagram of $\cO$-compact prestacks equipped with preorientations, we can extend the $D$-diagram $\Gamma_S(C, E^\vee\otimes\cL)$ to a $D^{\rhd}$-diagram where we set 
$\Gamma_S(C, E^\vee\otimes\cL)_{\ast} = \Gamma_S(C_{\emptyset}, E)^\vee$.

\begin{df}\label{df-nd-for-orientations}
Consider a $D^{\op}$-diagram $C\colon D^{\op}\rightarrow (\dPSt_{S\times Pic^{\gr}})_{\cB_S/}$ of $\cO$-compact derived $S$-prestacks equipped with preorientations.
\begin{itemize}
\item The $D^{\op}$-diagram $C$ is \defterm{weakly oriented} if for every perfect complex $E$ on $C_{\emptyset}$ the $D^{\rhd}$-diagram $\Gamma_S(C, E^\vee\otimes\cL)$ is a colimit diagram.
\item The $D^{\op}$-diagram $C$ is \defterm{oriented} if for every morphism of derived prestacks $\sigma\colon S'\rightarrow S$ the pullback $D^{\op}$-diagram $C\times_S S'$ is weakly oriented.
\end{itemize}
\end{df}

\begin{ex}\label{ex-nd-for-orientations}
The three examples below are completely analogous to the ones of \cref{examples-of-D}. 
\begin{enumerate}
\item For $D=\cdot$ we get the notion of non-degeneracy for $\cL$-orientations from \cref{df-orientation}.
\item The case $D=\wedge$ is equivalent to having a relative $\cL$-orientation on $C_1:=C_a\coprod C_b\xrightarrow{f} C_\emptyset$. 
The non-degeneracy is expressed by the coCartesianity of a square, for every choice of perfect complex $E$ on $C_\emptyset$, that turns out to be equivalent 
to the fact that the null-homotopic sequence 
\[
\Gamma_S(C_\emptyset,E_\emptyset^\vee\otimes \cL_\varnothing)\to \Gamma_S(C_1,E_1^\vee\otimes \cL_1)\to \Gamma_S(C_\emptyset,E_\emptyset)^\vee
\]
is a cofiber sequence. From the data of an $\cL$-preorientation on this cospan we obtain a morphism of null-homotopic sequences
\[
\xymatrix{
\Gamma_S(C_\emptyset,E_\emptyset^\vee\otimes \cL_\varnothing) \ar[r]\ar[d]^{id} & \Gamma_S(C_1,E_1^\vee\otimes \cL_1) \ar[r]\ar[d] & \Gamma_S(C_\emptyset,E_\emptyset)^\vee \ar[d]^{id}\\
\Gamma_S(C_\emptyset,E_\emptyset^\vee\otimes \cL_\varnothing) \ar[r] & \Gamma_S(C_1,E_1)^\vee  \ar[r] & \Gamma_S(C_\emptyset,E_\emptyset)^\vee
}
\]
Therefore, if these are cofiber sequences (that happens under the non-degeneracy assumption) then the middle vertical arrow is an equivalence. 
In other words, the non-degeneracy of the diagram implies in particular the non-degeneracy of the pre-orientations on $C_a$ and $C_b$ (this is not necessarily the case 
for isotropic and pre-symplectic structures).  
Our notion of non-degeneracy therefore coincides with the one from \cite[Definition 2.8 \& \S4.2.1]{CalaqueTFT} for relative orientations and oriented cospans. 
\item For $D=\Diamond$, we let the reader check that the situation is completely analogous to what happens in \cref{examples-of-D}(3): the non-degeneracy condition 
for such a diagram ends up being equivalent to the requirement that the induced relative $\cL[1]$-preorientation on 
$C_a\underset{C_c\coprod C_d}{\coprod}C_b\to C_\emptyset$ is a relative $\cL[1]$-orientation.  

There is also an analogue of \cref{prop:epimorphismrelativenondegenerate} (the proof of which is left to the reader): if a $\Diamond^{\op}$-diagram $C$ of 
$\cL$-preoriented derived prestacks is $\cL$-oriented and if the $\cL$-preoriented cospan $C_c\rightarrow C_a\leftarrow C_d$ is $\cL$-oriented, 
then the $\cL$-preoriented cospan $C_c\rightarrow C_b\leftarrow C_d$ is $\cL$-oriented as well.
\end{enumerate}
\end{ex}

We can compose orientations as follows.

\begin{prop}\label{prop:orientationcomposition}
Consider a morphism of derived prestacks $f\colon S\rightarrow T$ equipped with an $\cM$-orientation $[S]\colon \Gamma_T(S, \cM)\rightarrow \cO_T$. Consider an oriented $D^{\op}$-diagram
\[C\colon D^{\op}\longrightarrow (\dPSt_{S\times Pic^{\gr}})_{\cB_S/}\]
defined by a graded line bundle $\cL_d$ on $C_d$ for each $d\in D$. Let $\cL'_d$ be the graded line bundle $\cL_d\otimes g_d^* \cM$, where $g_d\colon C_d\rightarrow S$. Then the composite
\[
\Gamma_T(S, \Gamma_S(C_\bullet, \cL'))\cong \Gamma_T(S, \Gamma_S(C_\bullet, \cL)\otimes \cM)\xrightarrow{[C_\bullet]} \Gamma_T(S, \cM)\xrightarrow{[S]} \cO_T
\]
defines an orientation on $C$ over $T$.
\end{prop}
\begin{proof}
The composite of $\cO$-compact morphisms is $\cO$-compact, so each morphism $C_d\rightarrow T$ for $d\in D$ is $\cO$-compact.

Consider a morphism $T'\rightarrow T$ of derived prestacks and the corresponding diagram
\[
\xymatrix{
C' \ar[r] \ar[d] & C \ar[d] \\
S' \ar[r] \ar^{f'}[d] & S \ar[d] \\
T' \ar[r] & T
}
\]
where all squares are Cartesian. Our goal is to prove that $C'$ is weakly oriented over $T'$. Consider a perfect complex $E$ on $C'_\varnothing$. Since $C'$ is weakly oriented over $S'$, we have that the $D^{\rhd}$-diagram $\Gamma_{S'}(C', E^\vee\otimes\cL)$ is a colimit diagram. Since $f_*$ preserves colimits and satisfies the projection formula, we obtain that $\Gamma_{T'}(C', E^\vee\otimes \cL')$ is a colimit diagram, i.e. $C'$ is weakly oriented over $T'$.
\end{proof}

\subsection{AKSZ construction}

In this section we discuss a relative version of the AKSZ construction from \cite{PTVV} combining results from \cite{GinzburgRozenblyum} and \cite{CHS}.

\begin{df}
Let $X\rightarrow C\rightarrow S$ be morphisms of derived prestacks. The \defterm{Weil restriction of $X$ along $f\colon C\rightarrow S$} is the derived prestack $\Res_{C/S}(X)$ satisfying the universal property
\[\Map_{\dPSt_S}(T, \Res_{C/S}(X))\cong \Map_{\dPSt_C}(T\times_S C, X).\]
\end{df}

\begin{rmk}
The mapping prestack is an example of a Weil restriction. Namely, given derived $S$-prestacks $C, X$ regard $\underline{X} = X\times_S C$ as a derived $C$-prestack. Then
\[\Res_{C/S}(\underline{X})\cong \bMap_S(C, X).\]
\end{rmk}

The Weil restriction comes with natural morphisms
\[
\xymatrix{
& \Res_{C/S}(X)\times_S C \ar^{ev}[dr] \ar_{id\times f}[dl] \ar_h[d] & \\
\Res_{C/S}(X) & C & X
}
\]

\begin{prop}\label{prop:Rescotangentperfect}
Let $f\colon C\rightarrow S$ be an $\cO$-compact morphism of derived prestacks and $X\rightarrow C$ a morphism which admits a perfect relative cotangent complex $\bL_{X/C}$. Then $\Res_{C/S}(X)\rightarrow S$ admits a perfect relative cotangent complex. Moreover, if $f$ is $\cL$-oriented, we have
\[\bL_{\Res_{C/S}(X)/S} \cong (id\times f)_*(h^* \cL\otimes ev^*\bL_{X/C}).\]
\end{prop}
\begin{proof}
Since $f\colon C\rightarrow S$ is $\cO$-compact, for any derived $S$-prestack $Y$ by \cref{rmk:pushforwardleftadjoint} the pullback functor $(id\times f)^*\colon \Perf(Y)\rightarrow \Perf(Y\times_S C)$ admits a left adjoint $(id\times f)_\sharp\colon \Perf(Y\times_S C)\rightarrow \Perf(Y)$. By \cite[Proposition B.3.7]{Rozenblyum} $\Res_{C/S}(X)\rightarrow S$ admits a relative cotangent complex given by the formula
\[\bL_{\Res_{C/S}(X)/S} = (id\times f)_\sharp\circ ev^* \bL_{X/C}\]
which is therefore perfect. The $\cL$-orientation on $f$ induces an isomorphism $(id\times f)_\sharp(-)\cong (id\times f)_*(h^*\cL\otimes(-))$, which gives the formula for the cotangent complex.
\end{proof}

Using the evaluation morphism, we may define an integration map as follows.

\begin{df}
Let $f\colon C\rightarrow S$ be a morphism of derived prestacks and $\cL$ a graded line bundle on $C$. Consider an $\cL$-preorientation $[C]$ on $C$ over $S$. The \defterm{integration map}
\[\int_{C/S}\colon f_* (\DR_C(X)\otimes \cL)\longrightarrow \DR_S(\Res_{C/S}(X))\]
in $\QCoh(S)^{\gre}$ is the composite
\begin{align*}
f_* (\DR_C(X)\otimes \cL) &\xrightarrow{ev^*} f_* (\DR_C(\Res_{C/S}(X)\times_S C)\otimes \cL) \\
&\xleftarrow{\sim} f_* (f^* \DR_S(\Res_{C/S}(X))\otimes \cL) \\
&\xleftarrow{\sim} \DR_S(\Res_{C/S}(X)) \otimes f_*\cL \\
&\xrightarrow{id\otimes [C]} \DR_S(\Res_{C/S}(X)),
\end{align*}
where the first backwards map is given by the pullback of differential forms which is an isomorphism by \cref{cor:pullbackDRisomorphism} and the second backwards map is given by the projection formula which is an isomorphism by \cite[Theorem B.8.12]{CHS}.
\end{df}

\begin{ex}\label{ex:integrationweight1}
Suppose $X\rightarrow C$ admits a relative cotangent complex $\bL_{X/C}$. Using \cref{thm:DRgr} we may identify the integration map in weight $1$ with
\[\int_{C/S}\colon f_*(\Gamma_C(X,\bL_{X/C})\otimes \cL)\longrightarrow \Gamma_S(\Res_{C/S}(X), \bL_{\Res_{C/S}(X)/S}).\]
Identifying the relative cotangent complex $\bL_{\Res_{C/S}(X)/S}$ using \cref{prop:Rescotangentperfect}, the right-hand side becomes isomorphic to $f_*(\Gamma_C(\Res_{C/S}(X)\times_S C, ev^*\bL_{X/C})\otimes \cL)$. Under these identifications the integration map is simply given by the pullback
\[ev^\star\colon \Gamma_C(X, \bL_{X/C})\longrightarrow \Gamma_C(\Res_{C/S}(X)\times_S C, ev^*\bL_{X/C})\]
along $ev$.
\end{ex}

We have the following relative version of the AKSZ construction.

\begin{thm}\label{thm:relativeAKSZ}
Let $S$ be a derived prestack equipped with a graded line bundle $\cM$. Let $D_1$ and $D_2$ be small categories having initial objects $\varnothing$. Let $C\colon D_1^{\op}\rightarrow (\dPSt_{S\times Pic^{\gr}})_{\cB_S/}$ be an oriented $D_1^{\op}$-diagram of derived $S$-prestacks and denote by $f\colon C_{\emptyset}\rightarrow S$ the natural projection. Let $X\colon D_2\rightarrow (\dPSt_{C_\varnothing})_{/\cA^{2, \cl}_{C_\varnothing}(\cL_\varnothing\otimes f^*\cM)}$ be a non-degenerate $D_2$-diagram of derived $C_\varnothing$-prestacks. Then the $D_1\times D_2$-diagram $\Res_{C_\bullet/S}(X_\bullet\times_{C_\varnothing} C_\bullet)$ equipped with the $\cM$-twisted presymplectic structure $\int_{C_\bullet/S} \omega_\bullet$ is non-degenerate.
\end{thm}
\begin{proof}
The proof of non-degeneracy of the diagram of mapping prestacks given in \cite[Proposition 3.4.2]{CHS} extends verbatim to Weil restrictions if we replace the computation of the cotangent complex of the mapping prestack given in \cite[Proposition B.10.21]{CHS} by the computation of the cotangent complex of the Weil restriction given in \cite[Corollary B.3.7]{Rozenblyum}.
\end{proof}

We will use the above theorem in the following way.

\begin{cor}\label{cor:relativeAKSZ}
Let $S$ be a derived prestack equipped with graded line bundles $\cL_1, \cL_2$.
\begin{enumerate}
\item Suppose $X$ is a derived $S$-prestack equipped with an $\cL_1\otimes \cL_2$-twisted symplectic structure over $S$ and $f\colon C\rightarrow S$ a derived $S$-prestack equipped with an $f^*\cL_2$-orientation. Then the mapping prestack $\bMap_S(C, X)$ carries a natural $\cL_1$-twisted symplectic structure.

\item Suppose $L\rightarrow X$, a morphism of derived $S$-prestacks, carries an $\cL_1\otimes \cL_2$-twisted Lagrangian structure over $S$ and $C_1\rightarrow C\leftarrow C_2$, a morphism of derived $S$-prestacks, carries an $f^*\cL_2$-orientation, where $f\colon C\rightarrow S$. Then the correspondence
\[
\twospan{\bMap_S(C, L)}{\bMap_S(C_1, L)\times_{\bMap_S(C_1, X)} \bMap_S(C, X)}{\bMap_S(C_2, L)}{\bMap_S(C_2, X)}{S}
\]
carries a natural structure of a $2$-fold $\cL_1$-twisted Lagrangian correspondence.
\end{enumerate}
\end{cor}
\begin{proof}
For the first claim we take $D_1=D_2=\cdot$. The pullback $X\times_S C$ carries an $f^*\cL_1\otimes f^*\cL_2$-twisted symplectic structure over $C$. Therefore, $\Res_{C/S}(X\times_S C)\cong \bMap_S(C, X)$ carries an $\cL_1$-twisted symplectic structure over $S$ by \cref{thm:relativeAKSZ}.

For the second claim we take $D_1=D_2=\wedge$. An $\cL_1\otimes \cL_2$-twisted Lagrangian structure on $L\rightarrow X$ is the same as a structure of a non-degenerate $\cL_1\otimes \cL_2$-twisted presymplectic structure on the span $S\leftarrow L\rightarrow X$. Again writing mapping prestacks in terms of Weil restrictions, \cref{thm:relativeAKSZ} implies that the natural $\cL_1$-twisted presymplectic structure on the $D_1\times D_2$-diagram
\[
\xymatrix{
S & \bMap_S(C_1, L) \ar[r] \ar[l] & \bMap_S(C_1, X) \\
S \ar[u] \ar[d] & \bMap_S(C, L) \ar[r] \ar[l] \ar[u] \ar[d] & \bMap_S(C, X) \ar[u] \ar[d]  \\
S & \bMap_S(C_2, L) \ar[r] \ar[l] & \bMap_S(C_2, X)
}
\]
defined using the integration map is non-degenerate. This is equivalent to the non-degeneracy of the $\cL_1$-twisted presymplectic structure on the diagram
\[
\twospan{\bMap_S(C, L)}{\bMap_S(C_1, L)\times_{\bMap_S(C_1, X)} \bMap_S(C, X)}{\bMap_S(C_2, L)}{\bMap_S(C_2, X)}{S}
\]
\end{proof}

\subsection{The $\infty$-category of shifted Lagrangian correspondences}\label{subsec-lagcor}

Given an $\infty$-category $\cC$ with finite limits, an abelian group object $A\in\cC$ and an integer $k\geq 1$ recall from \cite{Haugseng} the symmetric monoidal $(\infty, k)$-category $\Span_k(\cC; A)$ of iterated spans. Forgetting the symmetric monoidal structure, we also denote $\Span_k(\cC; A)=\Span_k(\cC_{/A})$. This $(\infty, k)$-category has the following structure:
\begin{itemize}
\item Its objects are morphisms objects of $\cC_{/A}$, i.e. morphisms $X\rightarrow A$ in $\cC$.
\item $1$-morphisms are functors $\wedge\rightarrow \cC_{/A}$, i.e. diagrams $X\leftarrow L\rightarrow Y$ over $A$.
\item $2$-morphisms (for $k\geq 2$) are functors $\Diamond\rightarrow \cC_{/A}$, i.e. commutative diagrams
\[
\twospan{C}{L}{M}{X}{Y}
\]
over $A$.
\end{itemize}

Fix a derived prestack $S$ and a graded line bundle $\cL$ on $S$. Then $\Acl_S(\cL)$ is an abelian group object in the $\infty$-category $\dPSt_S$ of derived prestacks over $S$ and hence we may consider the $(\infty, k)$-category of $\cL$-twisted isotropic correspondences.

We will also be interested in (non full) $(\infty, k)$-subcategories
\[\Lag^{S, \cL}_k\subset \Span_k(\dPSt; \Acl_S(\cL))\]
of $\cL$-twisted Lagrangian correspondences obtained by imposing non-degeneracy conditions on diagrams. The closure of Lagrangian correspondences under compositions was shown for $k=1$ in \cite{CalaqueTFT}, for $k=2$ in \cite{ABB} and for a general $k$ in \cite{CHS}.

When $S=pt$ we simply denote $\Lag^{pt, \cO_S[n]}_k = \Lag^n_k$.


\section{Shifted cotangent bundles}

By the results of \cite{CalaqueCotangent}, if $X$ is a derived Artin stack, the $n$-shifted cotangent bundle $T^*[n] X$ is $n$-shifted symplectic and if $f\colon X\rightarrow Y$ is a morphism of derived Artin stacks, we have an $n$-shifted Lagrangian correspondence
\[
\xymatrix{
& T^*[n] Y\times_Y X \ar[dl] \ar[dr] & \\
T^*[n] X && T^*[n] Y
}
\]
The goal of this section is to, first, remove the assumption that the derived stacks are Artin and, second, promote this construction to a functor out of the category of correspondences.

\subsection{Shifted cotangent bundles}\label{sect:shiftedcotangent}

In this section we recall the main results of \cite{CalaqueCotangent} as well as generalize them to the relative setting.

\begin{df}
Let $\pi\colon X\rightarrow S$ be a morphism of derived prestacks which admits a relative cotangent complex $\bL_{X/S}\in\QCoh(X)^-$ and $\cL$ a graded line bundle on $S$. The \defterm{$\cL$-twisted cotangent bundle} is
\[T^*_\cL(X/S) = \Tot_X(\bL_{X/S}\otimes \pi^* \cL)\xrightarrow{p} X.\]
For $\cL=\cO_S[n]$ we get the \defterm{$n$-shifted cotangent bundle $T^*[n](X/S)$}.
\end{df}

\begin{rmk}
Suppose $\cL=L[n]$ for an ungraded line bundle $L$ over $S$ and denote by $L^\times$ the total space of the corresponding $\bG_m$-torsor. Then $T^*_\cL(X/S)\cong T^*[n](X/S)\times^{\bG_m} L^\times$, where we consider the $\bG_m$-action on $T^*[n](X/S)$ by scaling the cotangent fibers with weight $1$.
\end{rmk}

Let $\tilde{\lambda}_X\in\Gamma_S(T^*_\cL(X/S), p^*\bL_{X/S})\otimes \cL$ be the tautological section. Define the \defterm{Liouville one-form} $\lambda_X\in\cA^1_S(T^*_\cL(X/S), \cL)$ to be the image of $\tilde{\lambda}_X$ under the pullback
\[\Gamma_S(T^*_\cL(X/S), p^*\bL_{X/S})\otimes \cL\longrightarrow \Gamma_S(T^*_\cL(X/S), \bL_{T^*_\cL(X/S)/S})\otimes \cL.\]
It has weight $1$ with respect to the $\bG_m$-action on $T^*_\cL(X/S)$. The $\cL$-twisted closed $2$-form $\ddr\lambda_X$ endows $T^*_\cL(X/S)$ with an $\cL$-twisted presymplectic structure relative to $S$ as well as an $\cL$-twisted isotropic fibration structure on the projection $p\colon T^*_\cL(X/S)\rightarrow X$.

For a morphism $f\colon X\rightarrow Y$ of derived $S$-prestacks admitting cotangent complexes relative to $S$ the morphism $f^*\bL_{Y/S}\rightarrow \bL_{X/S}$ induces a correspondence
\[
\xymatrix{
& T^*_\cL (Y/S)\times_Y X \ar[dl] \ar[dr] & \\
T^*_\cL (Y/S) && T^*_\cL (X/S)
}
\]
The images of the tautological sections $\tilde{\lambda}_X$ and $\tilde{\lambda}_Y$ in $\Gamma_S(T^*_\cL(Y/S)\times_Y X, p^* \bL_{X/S})\otimes \cL$ coincide, so the above correspondence has a natural structure of an $\cL$-twisted isotropic correspondence (relative to $S$). Let us now show the non-degeneracy of the shifted presymplectic structures.

\ 

We begin with a pair of useful lemmas. Consider morphisms of derived prestacks $f\colon X\rightarrow Y$ and $g\colon Y\rightarrow S$ which admit relative cotangent complexes. Let $I$ be the fiber of $\DR_S(X)\rightarrow g_*\DR_Y(X)$. As this is a morphism of graded mixed $\DR_S(X)$-modules, we have
\[I\in\Mod_{\DR_S(X)}(\QCoh(S)^{\gre}).\]
Consider the fiber sequence of relative cotangent complexes
\begin{equation}\label{eq:relativecotangent}
f^*\bL_{Y/S}\longrightarrow \bL_{X/S}\longrightarrow \bL_{X/Y}
\end{equation}
in $\QCoh(Y)$. Using \cref{thm:DRgr} we get
\[I(0) = 0,\qquad I(1)\cong \Gamma_S(X, f^*\bL_{Y/S})[-1].\]
Next, there is a commutative square
\[
\xymatrix{
\wedge^2 \bL_{X/S}\ar[r] \ar[d] & \wedge^2 \bL_{X/Y} \ar[d] \\
\bL_{X/S}\otimes \bL_{X/Y} \ar[r] & \bL_{X/Y}\otimes \bL_{X/Y}
}
\]
Taking fibers of the horizontal maps, we obtain a map
\[
\alpha\colon I(2)[2]\longrightarrow f^*\bL_{Y/S}\otimes \bL_{X/Y}
\]
such that the composite
\[f^*\wedge^2\bL_{Y/S}\longrightarrow I(2)[2]\xrightarrow{\alpha} f^*\bL_{Y/S}\otimes \bL_{X/Y}\]
is nullhomotopic because the composite \eqref{eq:relativecotangent} is nullhomotopic.

\begin{lem}\label{lm:dRdescription}
Let $f\colon X\rightarrow Y$ and $g\colon Y\rightarrow S$ be morphisms of derived prestacks which admit relative cotangent complexes. Then the composite
\[\Gamma_S(X, f^*\bL_{Y/S})\cong I(1)[1]\xrightarrow{\ddr} I(2)[2]\xrightarrow{\alpha} \Gamma_S(X, \bL_{X/Y}\otimes f^*\bL_{Y/S})\]
is homotopic to
\[\ddr\otimes id\colon g_*(\DR_Y(X)(0)\otimes \bL_{Y/S})\longrightarrow g_*(\DR_Y(X)(1)\otimes \bL_{Y/S})[1],\]
where in the last equation $\ddr$ denotes the de Rham differential relative to $Y$.
\end{lem}
\begin{proof}
The construction of the homotopy will be natural in $X\rightarrow Y\rightarrow S$, so as in the proof of \cref{thm:DRgr} we may reduce first to the case when $S$, $Y$ and $X$ are all affine, which we assume from now on. As a $\Gamma_S(Y, \DR_Y(X))(0)$-module, $I(1)$ is free on $\bL_Y$. Therefore, the mixed structure is uniquely determined by its value on $\Gamma_S(Y, \bL_{Y/S})$. Consider the commutative diagram
\[
\xymatrix{
\Gamma_S(Y, \bL_{Y/S}) \ar[r] \ar^{\ddr}[d] & I(1)[1] \ar^{\ddr}[d] & \\
\Gamma_S(Y, \wedge^2 \bL_{Y/S}) \ar[r] & I(2)[2] \ar^-{\alpha}[r] & \Gamma_S(X, f^*\bL_{Y/S}\otimes \bL_{X/Y}).
}
\]
As we have observed above, the bottom composite is null-homotopic, so the composite
\[\Gamma_S(Y, \bL_{Y/S})\longrightarrow \Gamma_S(X, f^*\bL_{Y/S})\cong I(1)[1]\xrightarrow{\ddr} I(2)[2]\xrightarrow{\alpha} \Gamma_S(X, \bL_{X/Y}\otimes f^*\bL_{Y/S})\]
has a canonical nullhomotopy, which proves the claim.
\end{proof}

\begin{rmk}
Given the morphism $\DR_S(X)\rightarrow g_*\DR_Y(X)$ in $\CAlg(\QCoh(S)^{\gre})$, one may apply deformation to the normal cone to obtain a filtration $I^0\leftarrow I^1\leftarrow \dots$, such that $I^0 = \DR_S(X)$ and $I^1 = I$. One may then strengthen \cref{lm:dRdescription} to the claim that the associated graded $\operatorname{gr} I^\bullet$ is isomorphic to $g_*(\DR_Y(X)\otimes \Sym(f^*\bL_{Y/S}[-1]))\in\CAlg(\QCoh(S)^{\gre, \gr})$ as a bigraded mixed cdga over $S$ with $f^*\bL_{Y/S}$ concentrated in weight $(0, 1)$. Using this perspective \cref{lm:dRdescription} describes the action of the mixed structure on $\operatorname{gr} I^\bullet$ on elements of weight $(1, 1)$.
\end{rmk}

\begin{lem}\label{lm:DRlinear}
Let $X$ be a derived prestack and $V\in\Perf(X)$ a perfect complex. Then
\[\DR_{X\times\Bpre\bG_m}(\Tot(V)/\bG_m)\cong \Sym(V^*(1)\oplus V^*(1)[-1]),\]
where the mixed structure is given by the identity map $V^*\rightarrow V^*$.
\end{lem}
\begin{proof}
Let $p^{\gr}\colon \Tot(V)/\bG_m\rightarrow X\times\Bpre\bG_m$ be the projection. By \cref{prop:functionslinearstack} we have
\[\DR_{X\times\Bpre\bG_m}(\Tot(V)/\bG_m)(0)\cong \Sym(V^*(1)).\]
But $\Sym(V^*(1)\oplus V^*(1)[-1])$ is the universal graded mixed cdga on $X\times\Bpre\bG_m$ whose weight $0$ part is $\Sym(V^*(1))$. Thus, there is a natural morphism of graded mixed cdgas
\begin{equation}
\Sym(V^*(1)\oplus V^*(1)[-1])\longrightarrow \DR_{X\times\Bpre\bG_m}(\Tot(V)/\bG_m)
\label{eq:DRlinearstack}
\end{equation}
on $X\times\Bpre\bG_m$. By \cref{thm:DRgr} we have
\[\DR_{X\times\Bpre\bG_m}(\Tot(V)/\bG_m)\cong p^{\gr}_* \Sym(\bL_{(\Tot(V)/\bG_m)/(X\times\Bpre\bG_m)}[-1])\]
By \cite[Proposition 2.25]{Grataloup} we have $\bL_{(\Tot(V)/\bG_m)/(X\times\Bpre\bG_m)}\cong (p^{\gr})^*(V^*(1))$. Thus,
\[\DR_{X\times\Bpre\bG_m}(\Tot(V)/\bG_m)\cong p^{\gr}_* (p^{\gr})^* \Sym(V^*(1)[-1]).\]
By the projection formula proved in \cref{lm:projectionformula} we deduce that \eqref{eq:DRlinearstack} is an isomorphism.
\end{proof}

\ 

We are now ready to prove the non-degeneracy of the twisted presymplectic structure on the cotangent bundle.

\begin{thm}\label{thm:cotangentnondegenerate}
If $X$ is a derived $S$-prestack which admits a perfect relative cotangent complex $\bL_{X/S}$, the $\cL$-twisted presymplectic structure on $T^*_\cL(X/S)$ is non-degenerate. If $X\rightarrow Y$ is a morphism of derived $S$-prestacks which admit perfect cotangent complexes relative to $S$, the $\cL$-twisted isotropic structure on the correspondence $T^*_\cL(Y/S)\leftarrow T^*_\cL(Y/S)\times_Y X\rightarrow T^*_\cL(X/S)$ is non-degenerate.
\end{thm}
\begin{proof}
As in the proof of \cite[Theorem 2.2]{CalaqueCotangent}, to prove that the $\cL$-twisted presymplectic structure $\ddr \lambda_X$ on $T^*_\cL(X/S)$ is symplectic, we just need to show that the $\cL$-twisted isotropic fibration $T^*_\cL(X/S)\rightarrow X$ is actually Lagrangian, i.e. the induced morphism $\bT_{T^*_\cL(X/S)/X}\rightarrow p^*(\bL_{X/S}\otimes\pi^*\cL)$ is an isomorphism. This morphism is induced by an element of $\psi_X\in \Gamma_S(T^*_\cL(X/S), \bL_{T^*_\cL(X/S)/X}\otimes p^*\bL_{X/S})\otimes\cL$ obtained as follows:
\begin{enumerate}
\item By construction $\lambda_X$, an $\cL$-twisted one-form linear along the fibers of $T^*_\cL(X/S)\rightarrow X$, lifts to an element $\tilde{\lambda}_X\in\Gamma_S(T^*_\cL(X/S), p^*\bL_{X/S})\otimes\cL$.
\item By \cref{lm:dRdescription} $\psi_X$ is obtained as the image of $\tilde{\lambda}_X$ under
\[\ddr\otimes id\colon \Gamma_S(T^*_\cL(X/S), p^*\bL_{X/S})\otimes\cL\longrightarrow \Gamma_S(T^*_\cL(X/S), \bL_{T^*_\cL(X/S)/X}\otimes p^*\bL_{X/S})\otimes\cL.\]
\end{enumerate}

Using the description of the relative de Rham algebra $\DR_X(T^*_\cL(X/S))$ from \cref{lm:DRlinear} we see that $\psi_X$ is obtained by applying $p^*$ to the canonical element of $(\bT_{X/S}\otimes\pi^*\cL^{-1})\otimes (\bL_{X/S}\otimes\pi^*\cL)$. Since it is non-degenerate, it proves the claim that $\bT_{T^*_\cL(X/S)/X}\rightarrow p^*(\bL_{X/S}\otimes\pi^*\cL)$ is an isomorphism.

The fact that $T^*_\cL(Y/S)\leftarrow T^*_\cL(Y/S)\times_Y X\rightarrow T^*_\cL(X/S)$ is an $\cL$-twisted Lagrangian correspondence is proven analogously, see \cite[Theorem 2.8]{CalaqueCotangent} for more details.
\end{proof}

\begin{ex}
For the final morphism $f\colon X\rightarrow Y=S$ the $\cL$-twisted Lagrangian correspondence from \cref{thm:cotangentnondegenerate} reduces to an $\cL$-twisted Lagrangian structure on the zero section $X\rightarrow T^*_\cL(X/S)$.
\end{ex}

\begin{ex}\label{ex:conormal}
Intersecting the $\cL$-twisted Lagrangian correspondence $T^*_\cL(Y/S)\leftarrow T^*_\cL(Y/S)\times_Y X\rightarrow T^*_\cL(X/S)$ from \cref{thm:cotangentnondegenerate} with the $\cL$-twisted Lagrangian structure on the zero section $X\rightarrow T^*_\cL(X/S)$ we obtain an $\cL$-twisted Lagrangian structure on the morphism $N^*_\cL(X/Y)\rightarrow T^*_\cL(X/S)$, where
\[N^*_\cL(X/Y) = Y\times_{T^*_\cL(Y/S)} (T^*_\cL(Y/S)\times_Y X)\]
is the \defterm{$\cL$-twisted conormal bundle}.
\end{ex}

The construction of cotangent bundles is compatible with the AKSZ construction in the following way.

\begin{prop}\label{prop:AKSZcotangent}
Let $S$ be a derived prestack equipped with a graded line bundle $\cM$. Let $f\colon C\rightarrow S$ be a derived $S$-prestack equipped with an $\cL$-orientation. Consider a morphism $\pi\colon X\rightarrow C$ which admits a perfect relative cotangent complex. Then there is an isomorphism of $\cM$-twisted symplectic prestacks
\[\Res_{C/S}(T^*_{\cL\otimes f^*\cM}(X/C))\cong T^*_\cM(\Res_{C/S}(X)/S),\]
where on the left we consider the $\cL$-twisted symplectic structure obtained using the AKSZ construction (\cref{thm:relativeAKSZ}).
\end{prop}
\begin{proof}
Let $T$ be a derived affine scheme. By definition, a morphism $T\rightarrow \Res_{C/S}(T^*_{\cL\otimes f^*\cM}(X/C))$ is the same as the following collection of data:
\begin{itemize}
\item A morphism $T\rightarrow S$ of derived prestacks.
\item A morphism $g\colon T\times_S C\rightarrow X$ of derived prestacks over $C$.
\item A morphism $s\colon \cO_{T\times_S C}\rightarrow g^*(\bL_{X/C}\otimes \pi^* (\cL\otimes f^*\cM))$ of perfect complexes.
\end{itemize}

Similarly, a morphism $T\rightarrow T^*_\cM(\Res_{C/S}(X)/S)$ is the same as the following collection of data:
\begin{itemize}
\item A morphism $T\rightarrow S$ of derived prestacks.
\item A morphism $\tilde{g}\colon T\rightarrow \Res_{C/S}(X)$ of derived prestacks over $S$.
\item A morphism $\tilde{s}\colon \cO_T\rightarrow \tilde{g}^*(\bL_{\Res_{C/S}(X)/S}\otimes\tilde{\pi}^*\cM)$ of perfect complexes, where $\tilde{\pi}\colon \Res_{C/S}(X)\rightarrow S$.
\end{itemize}

Consider the evaluation morphism
\[ev\colon \Res_{C/S}(X)\times_S C\longrightarrow X.\]
By \cref{prop:Rescotangentperfect} we may identify
\[\bL_{\Res_{C/S}(X)/S}\cong (id\times f)_* ev^*(\bL_{X/C}\otimes \pi^*\cL).\]

By the definition of the Weil restriction, the data of $\tilde{g}$ is equivalent to the data of a morphism
\[g\colon T\times_S C\xrightarrow{\tilde{g}\times id} \Res_{C/S}(X)\times_S C\xrightarrow{ev} X\]
of derived prestacks over $C$. Therefore, using the above description of the cotangent complex of $\Res_{C/S}(X)$ as well as the base change property for the diagram
\[
\xymatrix{
T\times_S C\ar^-{\tilde{g}\times id}[r] \ar^{id\times f}[d] & \Res_{C/S}(X)\times_S C \ar^{id\times f}[d] \\
T \ar^{\tilde{g}}[r] & \Res_{C/S}(X)
}
\]
which holds by \cite[Theorem B.8.12]{CHS} we see that the data of the morphism $\tilde{s}$ is equivalent to the data of a morphism
\[s'\colon \cO_T\longrightarrow (id\times f)_* g^*(\bL_{X/C}\otimes \pi^* (\cL\otimes f^*\cM))\]
which, by adjunction, is equivalent to the data of a morphism $s\colon \cO_{T\times_S C}\rightarrow g^*(\bL_{X/C}\otimes \pi^* (\cL\otimes f^*\cM))$. This finishes the construction of the isomorphism
\[\Res_{C/S}(T^*_{\cL\otimes f^*\cM}(X/C))\cong T^*_\cM(\Res_{C/S}(X)/S)\]
of derived prestacks. By construction, it is compatible with the forgetful maps to $\Res_{C/S}(X)$.

Let us now show that this isomorphism is compatible with $\cM$-twisted symplectic structures. Consider the Liouville one-forms
\begin{align*}
\lambda_X&\in\Gamma_C(T^*_{\cL\otimes f^*\cM}(X/C), \bL_{T^*_{\cL\otimes f^*\cM}(X/C)/C}\otimes \pi^*(\cL\otimes f^*\cM)),\\
\lambda_{\Res_{C/S}(X)}&\in \Gamma_S(T^*_{\cM}(\Res_{C/S}(X)/S), \bL_{T^*_{\cM}(\Res_{C/S}(X)/S)/S}\otimes \tilde{\pi}^*\cM).
\end{align*}
Since
\[\int_{C/S}\ddr \lambda_X\sim \ddr \int_{C/S} \lambda_X,\]
to show that the above isomorphism is compatible with $\cM$-twisted symplectic structures, it is enough to identify the $\cM$-twisted one-forms $\int_{C/S} \lambda_X$ and $\lambda_{\Res_{C/S}(X)}$.

Let $p\colon T^*_{\cL\otimes f^*\cM}(X/C)\rightarrow X$ and $\tilde{p}\colon T^*_\cM(\Res_{C/S}(X)/S)\rightarrow \Res_{C/S}(X)$ be the natural projections. Using the description of the integration map $\int_{C/S}$ in weight $1$ given in \cref{ex:integrationweight1} we obtain a commutative diagram
\[
\xymatrix{
\Gamma_C(T^*_{\cL\otimes f^*\cM}(X/C), p^*\bL_{X/C}\otimes \pi^*(\cL\otimes f^*\cM))\ar[r] \ar^{\int_{C/S}}[d] & \Gamma_C(T^*_{\cL\otimes f^*\cM}(X/C), \bL_{T^*_{\cL\otimes f^*\cM}(X/C)/C}\otimes \pi^*(\cL\otimes f^*\cM)) \ar^{\int_{C/S}}[d] \\
\Gamma_S(T^*_{\cM}(\Res_{C/S}(X)/S), \tilde{p}^*\bL_{\Res_{C/S}(X)/S}\otimes \tilde{\pi}^*\cM) \ar[r] & \Gamma_S(T^*_{\cM}(\Res_{C/S}(X)/S), \bL_{T^*_{\cM}(\Res_{C/S}(X)/S)/S}\otimes \tilde{\pi}^*\cM)
}
\]

Since $\lambda_X$ and $\lambda_{\Res_{C/S}(X)}$ are both obtained from the tautological sections $\tilde{\lambda}_X$ and $\tilde{\lambda}_{\Res_{C/S}(X)}$ under horizontal maps, it is enough to identify $\int_{C/S} \tilde{\lambda}_X$ and $\tilde{\lambda}_{\Res_{C/S}(X)}$. Fix a derived affine scheme $T$ together with a morphism $T\rightarrow T^*_\cM(\Res_{C/S}(X)/S)$, which has a description given at the beginning of the proof. By definition, the pullback of $\tilde{\lambda}_{\Res_{C/S}(X)}$ under $T\rightarrow T^*_\cM(\Res_{C/S}(X)/S)$ is given by $\tilde{s}$. Similarly, the pullback of $\int_{C/S} \tilde{\lambda}_X$ under the same map is given by
\[\cO_T\xrightarrow{s} (id\times f)_*(g^*(\bL_{X/C}\otimes \pi^*(\cL\otimes f^*\cM)))\cong \tilde{g}^*(\bL_{\Res_{C/S}(X)/S}\otimes \tilde{\pi}^*\cM).\]
But the isomorphism $\Res_{C/S}(T^*_{\cL\otimes f^*\cM}(X/C))\cong T^*_\cM(\Res_{C/S}(X)/S)$ is precisely constructed so that these two elements are equivalent.
\end{proof}

\begin{rmk}
The isomorphism $\Res_{C/S}(T^*_{\cL\otimes f^*\cM}(X/C))\cong T^*_\cM(\Res_{C/S}(X)/S)$ constructed in \cref{prop:AKSZcotangent} is uniquely characterized by the following two properties:
\begin{enumerate}
    \item The diagram
    \[
    \xymatrix{
    \Res_{C/S}(T^*_{\cL\otimes f^*\cM}(X/C)) \ar[rr] \ar_{\Res_{C/S}(p)}[dr] && T^*_\cM(\Res_{C/S}(X)/S) \ar^{\tilde{p}}[dl] \\
    & \Res_{C/S}(X) &
    }
    \]
    is commutative.
    \item The tautological one-forms $\int_{C/S} \tilde{\lambda}_X$ and $\tilde{\lambda}_{\Res_{C/S}(X)}$ are equivalent.
\end{enumerate}
\end{rmk}

\begin{ex}
Let $C$ be a smooth proper variety over $S=pt$ of dimension $d$. Let $\cK_C=K_C[d]$ be its graded canonical bundle. Recall from \cref{prop:Gorensteinorientation} that $X$ is naturally $\cK_C$-oriented. Let $G$ be an algebraic group and consider the moduli stack $\Bun_G(C) = \bMap(C, \B G)$ of $G$-bundles on $C$. Then by \cref{prop:AKSZcotangent} we get an isomorphism of $n$-shifted symplectic stacks
\[T^*[n]\Bun_G(C)\cong \Res_{C/pt}(T^*_{K_C[n+d]}((\B G\times C) / C)),\]
where $T^*_{K_C[n+d]}(\B G\times C/C)\cong T^*[n+d](\B G)\times^{\bG_m} K_C^\times$. The derived stack $\Res_{C/pt}(T^*_{K_C[n+d]}((\B G\times C) / C))$ parametrizes $G$-bundles $P\rightarrow C$ together with a section $\phi\in \Gamma(C, K_C\otimes \operatorname{coad} P)[n+d-1]$, where $\operatorname{coad} P$ is the coadjoint bundle. It is a shifted (and higher-dimensional) version of the moduli stack of Higgs bundles which arises in the case $d=1$ and $n=0$.
\end{ex}

\subsection{Functoriality of the cotangent bundle}

The cotangent complex has the opposite functoriality to that of the tangent complex; so, more work is needed to exhibit a functoriality of the cotangent bundle.

\begin{thm}
There is a product-preserving functor
\[T^*[n]\colon \dPSt^{\perfdef}\longrightarrow \Lag_1^n\]
which sends a derived prestack $X$ to $T^*[n] X$ equipped with its natural $n$-shifted symplectic structure and a morphism $X\rightarrow Y$ of derived prestacks to the Lagrangian correspondence
\[
\xymatrix{
& T^*[n] Y\times_Y X \ar[dl] \ar[dr] & \\
T^*[n] X && T^*[n] Y
}
\]
\end{thm}
\begin{proof}
Recall the Cartesian fibration $\cQ^{\pre,\op}\rightarrow \dPSt$. Its dual Cartesian fibration was denoted by $\cQ^{\pre}_{\op}\rightarrow \dPSt$. By \cite[Proposition 1.5]{BarwickGlasmanNardin} we may identify $\cQ^{\pre,\op}$ with the subcategory of $\Span_1(\cQ^{\pre}_{\op})$ with the same objects and whose morphisms consist of diagrams $(X, V_X)\leftarrow (Z, V_Z)\rightarrow (Y, V_Y)$, where $(Z, V_Z)\rightarrow (X, V_X)$ is a Cartesian morphism and $Z\rightarrow Y$ is an isomorphism in $\dPSt$. Therefore, we get a functor
\[\Tot\colon \cQ^{\pre,\op}\longrightarrow \Span_1(\cQ^{\pre}_{\op})\xrightarrow{\Tot} \Span_1(\dPSt).\]

By \cref{prop:cotangentcomplexfunctoriality} the cotangent complex defines a functor
\[\bL[n]\colon \dPSt^{\perfdef}\longrightarrow (\cQ^{\pre}_{\bL[n]/})^{\op}.\]

Post-composing it with $\Tot$ we obtain
\[T^*[n]\colon \dPSt^{\perfdef}\longrightarrow (\cQ^{\pre}_{\bL[n]/})^{\op}\longrightarrow \Span_1(\dPSt_{/\cA^1(n)})\longrightarrow \Span_1(\dPSt_{/\Acl(n)}),\]
where the last functor is given by applying the de Rham differential $\cA^1(n)\rightarrow \Acl(n)$.

By construction, this functor sends a derived prestack $X$ to $T^*[n] X$ and a morphism $X\rightarrow Y$ to the $n$-shifted isotropic correspondence $T^*[n] X\leftarrow T^*[n] Y\times_Y X\rightarrow T^*[n] Y$. By \cref{thm:cotangentnondegenerate} these are Lagrangian correspondences, i.e. $T^*[n]$ lands in the subcategory $\Lag_1^n\subset \Span_1(\dPSt_{/\Acl(n)})$.
\end{proof}

Next, we want to enhance the functoriality of the cotangent bundle with respect to spans of derived prestacks.

\begin{thm}\label{thm:spancotangent}
There is a symmetric monoidal functor
\[T^*[n]\colon \Span_1(\dPSt^{\perfdef})\longrightarrow \Lag_1^n\]
which sends a derived prestack $X$ to $T^*[n] X$ equipped with its natural $n$-shifted symplectic structure and a span $X\leftarrow Z\rightarrow Y$ of derived prestacks to the Lagrangian correspondence
\[
\xymatrix{
& N^*[n](Z/X\times Y) \ar[dl] \ar[dr] & \\
T^*[n] X && T^*[n] Y
}
\]
\end{thm}
\begin{proof}
Observe that $\Lag_1^n$ is the underlying $(\infty, 1)$-category of the $(\infty, 2)$-category $\Lag_2^n$ of iterated Lagrangian correspondences, so that we have a functor $T^*[n]\colon \dPSt^{\perfdef}\rightarrow \Lag_2^n$. As such, we may consider adjoints for the image of 1-morphisms under this functor. Consider the functor $\dPSt^{\perfdef}\rightarrow \Span_1(\dPSt^{\perfdef})$ which sends $X\rightarrow Y$ to the span $X\leftarrow X\rightarrow Y$. To establish a factorization
\[
\xymatrix{
\dPSt^{\perfdef} \ar[d] \ar^{T^*[n]}[dr] & \\
\Span_1(\dPSt^{\perfdef}) \ar@{-->}[r] & \Lag_2^n
}
\]
we will use the universal property of the $\infty$-category of spans established in \cite[Chapter 7.3]{GR1}. Namely, we have to check that the functor $T^*[n]\colon \dSt^{\perfdef}\rightarrow \Lag_2^n$ satisfies the left Beck--Chevalley condition of \cite[Chapter 7, Definition 3.1.2]{GR1}.

Any 1-morphism in $\Lag_2^n$, i.e. a Lagrangian correspondence $M_1\leftarrow N\rightarrow M_2$ admits a right adjoint given by the opposite correspondence $M_2\leftarrow N\rightarrow M_1$ as shown in \cite[Proposition 2.11.1]{CHS}. For a morphism $f\colon X\rightarrow Y$ denote by $T^*[n](f)$ the Lagrangian correspondence $T^*[n] X\leftarrow T^*[n]Y\times_Y X\rightarrow T^*[n] Y$ and by $T^*[n]^!(f)$ its right adjoint.

Now consider a Cartesian diagram
\[
\xymatrix{
W \ar^{\alpha_0}[r] \ar^{\beta_1}[d] & X \ar^{\beta_0}[d] \\
Y \ar^{\alpha_1}[r] & Z
}
\]
of derived prestacks. The Beck--Chevalley condition is that the natural base change 2-morphism $T^*[n](\beta_1)\circ T^*[n]^!(\alpha_0)\rightarrow T^*[n]^!(\alpha_1)\circ T^*[n](\beta_0)$ in $\Lag_2^n$ is an isomorphism. Let us unpack it. The left-hand side is given by the composite
\[
\xymatrix{
& T^*[n] X\times_X W \ar[dr] \ar[dl] && T^*[n] Y\times_Y W \ar[dl] \ar[dr] \\
T^*[n] X && T^*[n]W && T^*[n] Y
}
\]
which is
\[
\xymatrix{
& \left(T^*[n] X\times_X W\right) \times_{T^*[n]W} \left(T^*[n] Y\times_Y W\right) \ar[dl] \ar[dr] & \\
T^*[n] X && T^*[n] Y
}
\]

The right-hand side is given by the composite
\[
\xymatrix{
& T^*[n] Z\times_Z X \ar[dl] \ar[dr] && T^*[n] Z\times_Z Y \ar[dl] \ar[dr] & \\
T^*[n] X && T^*[n] Z && T^*[n] Y
}
\]
which is
\[
\xymatrix{
& T^*[n] Z\times_Z W \ar[dl] \ar[dr] & \\
T^*[n] X && T^*[n] Y
}
\]

The base change 2-morphism is given by the 2-fold span
\[
\xymatrix{
& T^*[n] Z\times_Z W \ar[dl] \ar[dr] & \\
T^*[n] Z\times_Z W && (T^*[n] X\times_X W) \times_{T^*[n] W} (T^*[n] Y\times_Y W)
}
\]
over $T^*[n] X$ and $T^*[n] Y$. But
\[
\xymatrix{
\alpha_0^* \beta_0^* \bL_Z \ar[r] \ar[d] & \alpha_0^* \bL_X \ar[d] \\
\beta_1^* \bL_Y \ar[r] & \bL_W
}
\]
is a Cartesian diagram in $\QCoh(W)$, so
\[T^*[n] Z\times_Z W\rightarrow (T^*[n] X\times_X W) \times_{T^*[n] W} (T^*[n] Y\times_Y W)\]
is an isomorphism.

Therefore, by \cite[Chapter 7.3, Theorem 3.2.2]{GR1} we obtain an extension of $T^*[n]\colon \dPSt^{\perfdef}\rightarrow \Lag_2^n$ to a functor $T^*[n]\colon \Span_1(\dPSt^{\perfdef})\rightarrow \Lag_2^n$. Its value on an object $X$ is $T^*[n] X$ and its value on a span $X\xleftarrow{f} Z\xrightarrow{g} Y$ is the composite $T^*[n](g)\circ T^*[n](f)$, i.e.
\[
\xymatrix{
& T^*[n] X\times_X Z \ar[dl] \ar[dr] && T^*[n] Y\times_Y Z \ar[dl] \ar[dr] & \\
T^*[n] X && T^*[n] Z && T^*[n] Y
}
\]
This precisely gives the Lagrangian correspondence $T^*[n]X \leftarrow N^*[n](Z/X\times Y)\rightarrow T^*[n] Y$ as claimed.
\end{proof}

\begin{ex}
Given a morphism of derived prestacks $Z\rightarrow X$ we may regard it as a correspondence $pt\leftarrow Z\rightarrow X$. Applying $T^*[n]$ to this correspondence we obtain a Lagrangian correspondence $pt\leftarrow N^*[n](Z/X)\rightarrow T^*[n] X$, i.e. an $n$-shifted Lagrangian morphism $N^*[n](Z/X)\rightarrow T^*[n] X$.
\end{ex}


\section{Symplectic groupoids in derived algebraic geometry}

In this section we introduce $n$-shifted symplectic groupoids, generalizing the theory of \cite{CDW}. For simplicity, we work with derived stacks over the point, but everything is valid in the relative setting we have discussed in the previous section. Moreover, we describe an AKSZ procedure to obtain $n$-shifted symplectic groupoids from what we call $d$-oriented co-groupoids.


\subsection{Shifted symplectic groupoids}

Observe that, since $\mathcal{A}^{2,\mathrm{cl}}(A, n)$ is obtained via the Dold--Kan correspondence, 
it is actually a simplicial $k$-module. Therefore, it makes sense to consider the groupoid stack
\[
\B_{\bullet}\Acl(n): \,A \longmapsto \B_{\bullet}\Acl(A, n)\,.
\]
\begin{df}
An \defterm{$n$-shifted presymplectic groupoid} is a groupoid stack $G_\bullet$ together with a map of groupoid stacks
\[\theta_\bullet\colon G_{\bullet} \longrightarrow \B_{\bullet}\Acl(n)\,.\]
\end{df}

\begin{rmk}\label{rem-crazy}
Note that the identification 
\[
\Omega_{0}\big(\Acl(n+1)\big):=pt\underset{\Acl(n+1)}{\times}pt\simeq\Acl(n)
\]
gives rise to an equivalence of groupoid stacks
\[\bN\big(pt\overset{0}{\longrightarrow}\Acl(n+1)\big)\simeq \B_{\bullet}\Acl(n)\,.\]

Therefore, since $\bM$ is left adjoint to $\bN$ (see \cref{subsection1.2}), one can equivalently define an $n$-shifted presymplectic groupoid 
as a groupoid stack $G_\bullet$ together with an $(n+1)$-shifted isotropic structure $(\omega_\theta,\eta_\theta)$ 
on $\bM(G_\bullet)=|e|\colon G_0\to |G_\bullet|$. That is to say, we have a square
\[
\xymatrix{
G_0 \ar[r]^{|e|} \ar[d] & |G_{\bullet}| \ar[d]^{\omega_\theta} \\
pt \ar[r] & \Acl(n+1) 
}
\]
defining an $n$-shifted presymplectic structure $\omega_\theta$ on $|G_\bullet|$ and a specific path $\eta_\theta$ in 
\[
\Map_{\dSt}\big(G_0,\Acl(n+1)\big)=\Acl(G_0,n+1)
\]
making the above square commute.
\end{rmk}

Observe that if $\theta_\bullet\colon G_{\bullet} \longrightarrow \B_{\bullet}\Acl(n)$ is an $n$-shifted presymplectic 
groupoid, then $\theta_{k}\colon G_k \to \Acl(n)^k$ corresponds to $k$ $n$-shifted presymplectic structures 
$\theta^1_k,\dots,\theta^k_k$ on $G_k$.

\begin{lem}\label{lm:presymplecticgroupoidisotropic}
The correspondence 
\[
\xymatrix{
G_k \ar[r] \ar[d]& G_1 \\
G_1^k & \\
}
\]
comes equipped with an $n$-shifted isotropic structure $\gamma_{k}(\theta)$, for each $k\geq 0$.    
\end{lem}
\begin{proof}
The map $\theta_\bullet$ induces a commuting diagram 
\[
\xymatrix{
G_k \ar[r] \ar[d] \ar^{\theta_k}[rd]& G_1 \ar^{\theta_1}[rd] &  \\
G_1^k \ar_{\theta_1^k}[rd] & \Acl(n)^k \ar[r]\ar@{=}[d] & \Acl(n) \\
& \Acl(n)^k &}
\]
which, after composing with the sum map $\sum:\Acl(n)^k\to\Acl(n)$, gives a commuting square 
\[
\xymatrix{
G_k \ar[rr] \ar[d] && G_1 \ar^{\theta_1}[d]\\
G_1^k\ar_{\sum\circ\theta_1^k}[rr] && \Acl(n) \\
}
\]
This gives the desired isotropic structure $\gamma_k(\theta)$. 
\end{proof}

\begin{df}\label{df:symplecticgroupoid}
An \defterm{$n$-shifted symplectic groupoid} is an $n$-shifted presymplectic groupoid $\theta_\bullet\colon G_{\bullet} \longrightarrow \B_{\bullet}\Acl(n)$ satisfying the following properties:
\begin{enumerate}
    \item The stacks $G_0$ and $G_1$ admit perfect cotangent complexes (in particular, it implies that $G_k$ admit perfect cotangent complexes for all $k$)
    \item The $n$-shifted isotropic structure $\gamma_2(\theta)$ on $G_1^2\leftarrow G_2\rightarrow G_1$ is $n$-shifted Lagrangian (in particular, it implies that $(G_1, \theta_1)$ is $n$-shifted symplectic).
\end{enumerate}
\end{df}

Suppose $f\colon G_0\rightarrow G_{-1}$ is an $(n+1)$-shifted isotropic morphism. In other words, we have a commutative diagram
\[
\xymatrix{
G_0 \ar^{f}[r] \ar[d] & G_{-1} \ar[d] \\
pt \ar^-{0}[r] & \Acl(n+1)
}
\]
Taking the nerves horizontally, we obtain a morphism of groupoid stacks $\bN(f)_\bullet\rightarrow \B_\bullet\Acl(n)$. 
In other words, $\bN(f)$ becomes an $n$-shifted presymplectic groupoid.
\begin{prop}\label{easy-proposition}
Suppose $f\colon G_0\rightarrow G_{-1}$ is an $(n+1)$-shifted Lagrangian morphism. Then the $n$-shifted presymplectic groupoid $\bN(f)$ is $n$-shifted symplectic.
\end{prop}
\begin{proof}
Let $G_\bullet = \bN(f)$. Consider the $(n+1)$-shifted Lagrangian correspondence $pt\leftarrow G_0\rightarrow G_{-1}$ as a 1-morphism from $pt$ to $G_{-1}$ in $\Lag_2^{n+1}$. By \cite[Proposition 2.11.1]{CHS} it admits a right adjoint given by the $(n+1)$-shifted Lagrangian correspondence $G_{-1}\leftarrow G_0\rightarrow pt$ with the counit of the adjunction given by the iterated Lagrangian correspondence
\[
\twospan{G_0}{G_0\times G_0}{G_{-1}}{G_{-1}}{G_{-1}}
\]
Therefore, $G_1$ viewed as a 1-morphism from $pt$ to $pt$ in $\Lag_2^{n+1}$ is a monad. This implies two facts:
\begin{enumerate}
    \item First, $G_1$ is an object of $\Lag_1^n$, i.e. $G_1$ is $n$-shifted symplectic.
    \item Second, the multiplication 2-morphism of the monad, which is
    \[
    \twospan{G_2}{G_1\times G_1}{G_1}{pt}{pt}
    \]
    is an iterated Lagrangian correspondence. Thus, $G_1^2\leftarrow G_2\rightarrow G_1$ is an $n$-shifted Lagrangian correspondence.
\end{enumerate}
\end{proof}

\begin{rmk}
An alternative way to see that
\[
G_2=G_0\underset{G_{-1}}{\times}G_0\underset{G_{-1}}{\times}G_0\longrightarrow G_1\times G_1\times \overline{G_1}
\]
is Lagrangian is to use \cite[Theorem 3.1]{BB} with $L=M=N=G_0$ and $S=G_{-1}$. Here the opposite $n$-shifted symplectic structure appears on the last $G_1$ because the order of the factors in the self-intersection have been exchanged compared to \cite{BB} (where the author uses the cyclic order on $L,M,N$). 
\end{rmk}

\begin{prop}\label{another-easy-proposition}
Let $G_\bullet$ be an $n$-shifted symplectic groupoid. 
The $(n+1)$-shifted isotropic structure $\gamma_k(\theta)$ is non-degenerate for all $k\in\mathbb{Z}_+$.
\end{prop}
\begin{proof}
Assume that $\theta_1$ and $\gamma_2(\theta)$ are non-degenerate. 
We first observe that $\gamma_1(\theta)$ is non-degenerate, simply because $\theta_1$ is. 

We then prove that $\gamma_k(\theta)$ is non-degenerate for all $k\geq2$. We proceed by induction on $k$. We have the following composition of $n$-shifted isotropic correspondences: 
$$
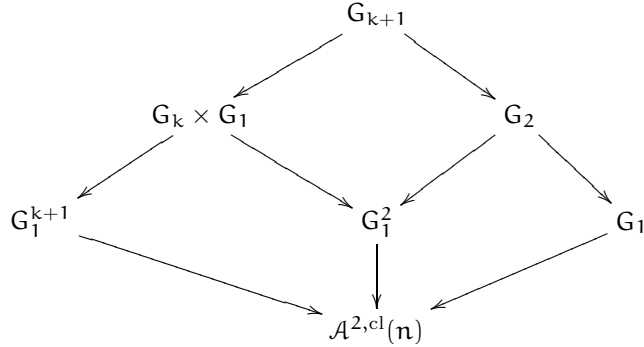
\xymatrix{
&&G_{k+1} \ar[rd]\ar[ld] && \\
& G_k\times G_1 \ar[rd]\ar[ld] && G_2 \ar[rd]\ar[ld] & \\
G_1^{k+1} \ar[rrd] && G_1^2 \ar[d] && G_1 \ar[lld] \\
&& \Acl(n) &&
}
$$
If the first correspondence (which is $\gamma_k(\theta)\times\gamma_1(\theta)$) and the second one (which is $\gamma_2(\theta)$) are non-degenerate, then their composition 
(which is $\gamma_{k+1}(\theta)$) is non-degenerate as well (see \cite[Theorem 4.4]{CalaqueTFT}).

We finally prove that $\gamma_0(\theta)$ is non-degenerate. We deduce it from the following composition of $n$-shifted isotropic correspondences: 
$$
\xymatrix{
&&G_1 \ar[rd]\ar[ld] && \\
& G_0 \ar[rd]\ar[ld] && G_2 \ar[rd]\ar[ld] & \\
pt~~~ \ar[rrd] && G_1 \ar[d] && G_1\times \overline{G_1} \ar[lld] \\
&& \Acl(n) &&
}
$$
The composed $n$-shifted isotropic correspondence is $\gamma_1(\theta)$ and is thus non-degenerate. 
Knowing that the second one (which is $\gamma_2(\theta)$) is non-degenerate, we get that the pull-back of the first one (which is $\gamma_0(\theta)$) along $G_1\to G_0$ is non-degenerate. Since $G_1\to G_0$ admits a section $G_0\to G_1$, we obtain that $\gamma_0(\theta)$ itself is non-degenerate. 
\end{proof}
There is a kind of converse to \cref{easy-proposition}: 
\begin{thm}\label{thm-symplectic-quotient}
Let $G_\bullet$ be an $n$-shifted presymplectic groupoid, such that $G_0$, $G_1$ and $|G_\bullet|$ admit a perfect cotangent complex. Then the following are equivalent: 
\begin{itemize}
    \item[(a)] $G_\bullet$ is an $n$-shifted symplectic groupoid.
    \item[(b)] The $(n+1)$-shifted isotropic structure $\eta_\theta$ on the morphism $\bM(G_\bullet)=(G_0\to |G_\bullet|)$ is $(n+1)$-shifted Lagrangian.
    \item[(c)] The $n$-shifted isotropic structure $\gamma_0(\theta)$ on the morphism $G_0\to G_1$ is $n$-shifted Lagrangian.
\end{itemize}
In this case the $(n+1)$-shifted presymplectic structure on $|G_\bullet|$ is non-degenerate as well.
\end{thm}
\begin{proof}
We first observe that (a)$\Rightarrow$(c) follows from \cref{another-easy-proposition}. 

We then prove (c)$\Leftrightarrow$(b). 
By \cref{lm:groupoidcotangent} we have an equivalence $\bT_{G_0/|G_\bullet|}\cong \bT_{G_0/G_1}[1]$ which induces a commutative diagram
\[
\xymatrix{
\bT_{G_0/|G_\bullet|} \ar^{\sim}[rr] \ar[dr] && \bT_{G_0/G_1}[1] \ar[dl] \\
& \bL_{G_0}[n]
}
\]
where the morphism on the left is induced by the $(n+1)$-shifted isotropic structure $\eta_\theta$ on $G_0\rightarrow |G_\bullet|$ 
and the morphism on the right is induced by the $n$-shifted isotropic structure $\gamma_0(\theta)$ on $G_0\rightarrow G_1$. 
Therefore, non-degeneracy for the two isotropic structures is equivalent. Non-degeneracy of the $(n+1)$-shifted presymplectic structure on $|G_\bullet|$ is automatic thanks to \cref{prop:epimorphismnondegenerate}.

We will now prove (b)$\Rightarrow$(a). Knowing that $G_0\to |G_\bullet|$ is $(n+1)$-shifted Lagrangian, and using \cref{easy-proposition}, we get that $G_\bullet=\bN(G_0\to |G_\bullet|)$ is an $n$-shifted symplectic groupoid. 
\end{proof}

\begin{ex}[Usual notion] \label{ex:symplecticgroupoid}
Let $G$ be a symplectic groupoid in the usual sense (see \cite{CDW}). 
This means that $G$ is a genuine smooth groupoid scheme with smooth source and target maps, together with an ordinary symplectic structure on the scheme $G_1$ of arrows such that the graph of the composition $G_1\underset{G_0}{\times}G_1\subset G_1\times G_1\times\overline{G_1}$ is Lagrangian. Then $G_\bullet:=\underbrace{G_1\underset{G_0}{\times}\cdots\underset{G_0}{\times}G_1}_{\bullet\text{ times}}$ is obviously a 
$0$-shifted symplectic groupoid in our sense. In particular, we get that $|G_\bullet|=[G_0/G_1]$ is $1$-shifted symplectic and $G_0\rightarrow |G_\bullet|$ is $1$-shifted Lagrangian (see also \cite[Proposition 3.32]{SafronovPoissonLie}).
\end{ex}


\subsection{Shifted Lagrangian groupoids}

Let $f\colon L_\bullet\longrightarrow G_\bullet$ a morphism of groupoid stacks. 

\begin{df}
An \defterm{$n$-shifted isotropic structure on $f$} is a commutative square of groupoid stacks
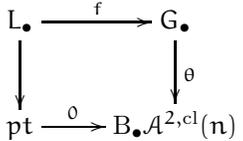
\[
\xymatrix{
L_\bullet \ar^{f}[r] \ar[d] & G_\bullet \ar^{\theta}[d] \\
pt \ar^-{0}[r] & \B_{\bullet}\Acl(n)
}
\]
\end{df}

Observe that an $n$-shifted isotropic structure on $f$ induces in particular an $n$-shifted isotropic structure on $f_1\colon L_1\to G_1$. Moreover, 
using the adjunction $\bM\dashv \bN$, the induced morphism of stacks $|f|\colon |L_\bullet|\to |G_\bullet|$ inherits a canonical $(n+1)$-shifted isotropic 
structure.

If $f\colon L_\bullet\rightarrow G_\bullet$ is an $n$-shifted isotropic morphism of groupoids, then there is a functor $\boxtimes\to \dSt_{/\Acl(n)}$ 
given by the commutative diagram
\begin{equation}\label{eq:isotropicgroupoidnondegenerate}
\twospan{L_2}{L_1\times_{G_0} L_1}{L_1}{G_1}{pt}
\end{equation}
over $\Acl(n)$.

\begin{df}\label{dfn-lagrangian-groupoid}
An \defterm{$n$-shifted Lagrangian structure} on $f$ is an $n$-shifted isotropic structure on $f$ such that
\begin{enumerate}
\item $G_\bullet$ is an $n$-shifted symplectic groupoid.
\item $L_0$ admits a perfect cotangent complex.
\item The $n$-shifted isotropic structure on $L_1\to G_1$ is $n$-shifted Lagrangian. 
\item The diagram \eqref{eq:isotropicgroupoidnondegenerate} is non-degenerate, meaning that the $(n-1)$-shifted isotropic structure 
on $L_2\to L_1\times_{G_1} (L_1\times_{G_0} L_1)$ is $(n-1)$-shifted Lagrangian.
\end{enumerate}
\end{df}

\begin{rmk}
If $G_\bullet$ is an $n$-shifted symplectic groupoid and $L_1\rightarrow G_1$ is $n$-shifted Lagrangian, then $L_1\times_{G_0} L_1\rightarrow G_1$ is $n$-shifted Lagrangian, being a composition of the $n$-shifted Lagrangian morphism $L_1\times L_1\rightarrow G_1\times G_1$ and the $n$-shifted Lagrangian correspondence $G_1\times G_1\leftarrow G_2\rightarrow G_1$.
\end{rmk}

\begin{prop}\label{prop-lagr-groupoid-from-2-lag}
Given a $2$-morphism 
\[
\twospan{L_0}{L_{-1}}{G_0}{G_{-1}}{pt}
\]
in $\Lag_2^{n+1}$, the induced morphism of $n$-shifted isotropic groupoids $\bN(L_0\to L_{-1})\to \bN(G_0\to G_{-1})$ is $n$-shifted Lagrangian. 
\end{prop}
\begin{proof}
Let $L_\bullet = \bN(L_0\to L_{-1})$ and $G_\bullet = \bN(G_0\to G_{-1})$. First, we already know from \cref{easy-proposition} that, since $G_0\to G_{-1}$ is $(n+1)$-shifted Lagrangian, then $G_\bullet$ is an $n$-shifted symplectic groupoid. The rest of the proof is analogous to \cref{easy-proposition}. Namely, we have two objects $G_0, L_{-1}$ in the $(\infty, 2)$-category $\Hom_{\Lag_3^{n+1}}(pt, G_{-1})$ together with a morphism $L_0$ from $G_0$ to $L_{-1}$. This morphism admits a right adjoint, which is again $L_0$ viewed as a morphism from $L_{-1}$ to $G_0$. Therefore, $L_1$ viewed as a 1-morphism from $G_0$ to $G_0$ in $\Hom_{\Lag_3^{n+1}}(pt, G_{-1})$ is a monad. This implies two facts:
\begin{enumerate}
    \item First, $L_1\rightarrow G_0\times_{G_{-1}} G_0\cong G_1$ is $n$-shifted Lagrangian.
    \item Second, the multiplication 2-morphism of the monad is a certain 3-morphism in $\Lag_3^{n+1}$. Rearranging the diagram it becomes the 2-morphism \eqref{eq:isotropicgroupoidnondegenerate} in $\Lag_2^n$.
\end{enumerate}
\end{proof}

\begin{thm}\label{thm-lagrangian-quotient}
Let $f\colon L_\bullet \to G_\bullet$ be a morphism between smooth groupoids equipped with an $n$-shifted isotropic structure. Suppose that $L_1$, $L_0$, $|L_\bullet|$, $G_1$, $G_0$ and $|G_\bullet|$ admit a perfect cotangent complex. Then the following are equivalent:
\begin{itemize}
    \item[(a)] The morphism $f\colon L_\bullet\rightarrow G_\bullet$ is $n$-shifted Lagrangian.
    \item[(b)] The $(n+1)$-shifted isotropic structure on $G_0\to|G_\bullet|$ is $(n+1)$-shifted Lagrangian and the commutative diagram of stacks
    \begin{equation}\label{eq:isotropicgroupoidquotientnondegenerate}
    \twospan{L_0}{|L_\bullet|}{G_0}{|G_\bullet|}{pt}
    \end{equation}
    over $\Acl(n+1)$ is non-degenerate.
    \item[(c)] The $n$-shifted isotropic structures on $G_0\rightarrow G_1$ is $n$-shifted Lagrangian and the commutative diagram of stacks
    \begin{equation}\label{eq:isotropicgroupoidunitnondegenerate}
    \twospan{L_0}{L_1}{G_0}{G_1}{pt}
    \end{equation}
    over $\Acl(n)$ is non-degenerate.
\end{itemize}
In this case the $(n+1)$-shifted isotropic structure on $|L_\bullet|\rightarrow |G_\bullet|$ is $(n+1)$-shifted Lagrangian as well.
\end{thm}
\begin{proof}
The implication (a)$\Rightarrow$(c) is proven as in \cref{another-easy-proposition}. Namely, consider $L_0$ as a 2-morphism
\[
\xymatrix{
pt \rtwocell^{G_0}_{L_1} & G_1
}
\]
in $\Span_2(\dArt_{/\Acl(n)})$ and $L_2$ as a 2-morphism
\[
\xymatrix{
pt \rtwocell^{L_1}_{L_1\times_{G_0} L_1} & G_1.
}
\]
Their composition is $L_1$ viewed as a 2-morphism
\[
\xymatrix{
pt \rtwocell^{G_0}_{L_1\times_{G_0} L_1} & G_1.
}
\]
Since $L_1\rightarrow G_1$ is $n$-shifted Lagrangian, the diagonal $L_1\rightarrow L_1\times_{G_1} L_1$ is $(n-1)$-shifted Lagrangian. In particular, the above composition is non-degenerate. Since \eqref{eq:isotropicgroupoidnondegenerate} is non-degenerate, it implies that the isotropic structure \eqref{eq:isotropicgroupoidunitnondegenerate} pulled back along $L_1\rightarrow L_0$ is non-degenerate. Since this morphism admits a section $L_0\rightarrow L_1$, this implies that \eqref{eq:isotropicgroupoidunitnondegenerate} itself is non-degenerate.

Next we prove (c)$\Leftrightarrow$(b). In \cref{thm-symplectic-quotient} we have shown that the non-degeneracy of the $(n+1)$-shifted isotropic structure on $G_0\rightarrow |G_\bullet|$ is equivalent to the non-degeneracy of the $n$-shifted isotropic structure on $G_0\rightarrow G_1$. Similarly to that proof, by \cref{lm:groupoidcotangent} we have an equivalence $\bT_{L_0/|L_\bullet|\times_{|G_\bullet|} G_0}\cong \bT_{L_0/L_1\times_{G_1} G_0}[1]$ which induces a commutative diagram
\[
\xymatrix{
\bT_{L_0/|L_\bullet|\times_{|G_\bullet|} G_0} \ar^{\sim}[rr] \ar[dr] && \bT_{L_0/L_1\times_{G_1} G_0}[1] \ar[dl] \\
& \bL_{L_0}[n-1]
}
\]
where the morphism on the left is induced by the $n$-shifted isotropic structure on $L_0\rightarrow |L_\bullet|\times_{|G_\bullet|} G_0$ 
and the morphism on the right is induced by the $(n-1)$-shifted isotropic structure on $L_0\rightarrow L_1\times_{G_1} G_0$. Therefore, non-degeneracy of the two shifted isotropic structures is equivalent.

Finally, we prove (b)$\Rightarrow$(a). Since $L_0\rightarrow |L_\bullet|\times_{|G_\bullet|} G_0$ is $n$-shifted Lagrangian and $L_0\rightarrow |L_\bullet|$ is an effective epimorphism, by \cref{prop:epimorphismnondegenerate} we get that the $n$-shifted presymplectic structure on $|L_\bullet|\times_{|G_\bullet|} G_0$ is $n$-shifted symplectic. Since the $(n+1)$-shifted isotropic structure on $G_0\rightarrow |G_\bullet|$ is $(n+1)$-shifted Lagrangian by assumption and the $n$-shifted presymplectic structure is obtained by a shifted isotropic intersection, we get that the restriction of the $(n+1)$-shifted isotropic structure on $|L_\bullet|\rightarrow |G_\bullet|$ along $|L_\bullet|\times_{|G_\bullet|} G_0\rightarrow |L_\bullet|$ is non-degenerate. Since $G_0\rightarrow |G_\bullet|$ is an effective epimorphism, by \cref{prop:epimorphismnondegenerate} this implies that the $(n+1)$-shifted isotropic structure on $|L_\bullet|\rightarrow |G_\bullet|$ is $(n+1)$-shifted Lagrangian. Thus,
\[
\twospan{L_0}{L_{-1}}{G_0}{G_{-1}}{pt}
\]
is a 2-morphism in $\Lag_2^{n+1}$ and hence by \cref{prop-lagr-groupoid-from-2-lag} $L_\bullet\rightarrow G_\bullet$ is $n$-shifted Lagrangian.
\end{proof}

\begin{ex}
Let $G_\bullet$ be an $n$-shifted symplectic groupoid. Let $L_\bullet=G_0$ (the constant groupoid). Then the morphism $e\colon G_0\rightarrow G_\bullet$ carries a natural $n$-shifted Lagrangian structure.
\label{ex:constantgroupoidLagrangian}
\end{ex}

\begin{ex}[Usual notion]
We have the following relative version of \cref{ex:symplecticgroupoid}. Let $G\rightrightarrows M$ be a smooth groupoid scheme with smooth source and target maps with $G_\bullet$ the corresponding simplicial scheme. Let $L\rightrightarrows C$ be a smooth subgroupoid, i.e. $L\subset G$ and $C\subset M$ are smooth closed subschemes and the source and target maps are smooth. Let $L_\bullet$ be the corresponding simplicial scheme and $f\colon L_\bullet\rightarrow G_\bullet$ the natural inclusion morphism. A $0$-shifted isotropic structure on $f\colon L_\bullet\rightarrow G_\bullet$ is the same as a multiplicative closed two-form $\omega$ on $G$ such that $\omega|_L = 0$. Using \cref{thm-lagrangian-quotient} the $0$-shifted isotropic structure on $f$ is $0$-shifted Lagrangian if the following conditions are satisfied:
\begin{enumerate}
    \item $G$ is a symplectic groupoid, i.e. $(G, \omega)$ is symplectic and the inclusion $G\times_M G\rightarrow G\times G\times \overline{G}$ is Lagrangian.
    \item $L\rightarrow G$ is Lagrangian, i.e. $L\rightrightarrows C$ is a Lagrangian subgroupoid of $G\rightrightarrows M$. Note that in \cref{thm-lagrangian-quotient}(c) besides the condition that $L\rightarrow G$ is Lagrangian there is also the condition that $C\rightarrow L\times_G M$ is $(-1)$-shifted Lagrangian. But this condition is equivalent to the fact that the natural morphism $N_{C/L}\rightarrow N^*_{C/M}$ is an isomorphism, which is automatic for Lagrangian subgroupoids.
\end{enumerate}
\end{ex}

Suppose $L_\bullet\rightarrow G_\bullet$ and $M_\bullet\rightarrow G_\bullet$ are two morphisms of groupoid stacks. Then the simplicial stack
\[X_\bullet = L_\bullet\times_{G_\bullet} M_\bullet\]
is again a groupoid stack.

Now suppose $L_\bullet\rightarrow G_\bullet$ and $M_\bullet\rightarrow G_\bullet$ carry $n$-shifted isotropic structures such that the two $n$-shifted presymplectic structures on $G_\bullet$ are equivalent. In other words, consider the diagram
\[
\xymatrix{
L_\bullet \ar[d] \ar[dr] && M_\bullet \ar[d] \ar[dl] \\
pt \ar^{0}[dr] & G_\bullet \ar[d] & pt \ar^{0}[dl] \\
& \B_{\bullet}\Acl(n) &
}
\]

Then we get a morphism of groupoid stacks
\[X_\bullet\longrightarrow pt \times_{\B_{\bullet}\Acl(n)} pt\cong \B_{\bullet}\Acl(n-1).\]
In other words, $X_\bullet$ becomes an $(n-1)$-shifted presymplectic groupoid.

\begin{prop}
Suppose $L_\bullet\rightarrow G_\bullet$ and $M_\bullet\rightarrow G_\bullet$ carry $n$-shifted Lagrangian structures such that the two $n$-shifted presymplectic structures on $G_\bullet$ are equivalent. Then the $(n-1)$-shifted presymplectic groupoid $L_\bullet\times_{G_\bullet} M_\bullet$ is $(n-1)$-shifted symplectic.
\label{prop:Lagrangianintersection}
\end{prop}
\begin{proof}
Since $X_1 = L_1\times_{G_1} M_1$ is a Lagrangian intersection, the $(n-1)$-shifted presymplectic structure on $X_1$ is $(n-1)$-shifted symplectic.

To show that $X_1\times X_1\leftarrow X_2\rightarrow X_1$ is an $(n-1)$-shifted Lagrangian correspondence, we use the fact that $\Lag_2^{n}$ is a sub-$(\infty,2)$-category of $\Span_2(\dArt_{/\Acl(n)})$. $L_2$ defines a $2$-morphism
\[
\xymatrix@C=3cm{
pt \rtwocell^{L_1}_{(L_1\times L_1)\times_{G_1\times G_1} G_2} & G_1
}
\]
in $\Lag_2^n$ and $M_2$ defines a $2$-morphism
\[
\xymatrix@C=3cm{
G_1 \rtwocell^{M_1}_{G_2\times_{G_1\times G_1}(M_1\times M_1)} & pt
}
\]

Their horizontal composition $L_2\times_{G_1} M_2$ therefore defines a $2$-morphism
\[
\xymatrix@C=3cm{
pt \rtwocell^{X_1}_{(L_1\times L_1)_{G_1\times G_1} G_2\times_{G_1\times G_1} (M_1\times M_1)} & pt
}
\]

$G_2$ defines a $2$-morphism
\[
\xymatrix@C=2cm{
G_1\times G_1 \rtwocell^{G_2\times_{G_1} G_2}_{G_1\times G_1} & G_1\times G_1
}
\]
which is a counit of the adjunction. Composing this $2$-morphism with the previous one we obtain that $X_2$ defines a $2$-morphism
\[
\xymatrix{
pt \rtwocell^{X_1}_{X_1\times X_1} & pt
}
\]

Since the individual 2-morphisms all belong to $\Lag_2^n$, so does the composite. Thus, the natural $(n-1)$-shifted presympelctic structure on $X_\bullet$ is $(n-1)$-shifted symplectic.
\end{proof}

\begin{ex}
Suppose $G_\bullet$ is an $n$-shifted symplectic groupoid. By \cref{ex:constantgroupoidLagrangian} the constant groupoid $e\colon G_0\rightarrow G_\bullet$ has an $n$-shifted Lagrangian structure. Therefore, its self-intersection $X_\bullet = G_0\times_{G_\bullet} G_0$ is an $(n-1)$-shifted symplectic groupoid. Note that $X_0 \cong G_0$.
\end{ex}

\subsection{Tangent and cotangent groupoids}

We will now show how to obtain natural atlases of tangent and cotangent bundles of derived Artin stacks. Recall that in \cref{sect:cotangentfunctoriality} we have introduced the functor $T[n]\colon \dSt^\bL\rightarrow \dSt$ sending a derived stack $X$ which admits a cotangent complex to its $n$-shifted tangent bundle.

\begin{df}
Let $G_\bullet$ be a groupoid stack, such that $G_0$ and $G_1$ admit cotangent complexes. The \defterm{$n$-shifted tangent groupoid of $G_\bullet$} is the groupoid $k\mapsto T[n](G_k)$.
\end{df}

Under appropriate smoothness assumptions, taking geometric realizations commutes with taking the shifted tangent groupoids.

\begin{prop}
Let $G_\bullet$ be a groupoid stack and denote the natural projection by $f\colon G_0\rightarrow G_{-1}=|G_\bullet|$. Moreover, assume that $G_0$ and $G_{-1}$ admit cotangent complexes and $f$ is an $m$-smooth morphism for some $m\geq -1$. Take any $n\geq 0$ and consider the $n$-shifted tangent groupoid $T[n](G_\bullet)$. Then there is an isomorphism of derived stacks
\[|T[n]G_\bullet|\cong T[n]G_{-1}.\]
\end{prop}
\begin{proof}
We have that $T[n](G_\bullet)$ is the nerve of $T[n]G_0\rightarrow T[n] G_{-1}$. So, to prove the claim, by \cref{prop:realizationfullyfaithful} it is enough to show that $T[n]G_0\rightarrow T[n]G_{-1}$ is an effective epimorphism. Consider the morphism $\bT_{G_0}[n]\rightarrow f^*\bT_{G_{-1}}[n]$; its homotopy fiber is $\bT_{G_0/G_{-1}}$. Since $G_0\rightarrow G_{-1}$ is $m$-smooth, we have that $\bL_{G_0/G_{-1}}$ is perfect of Tor-amplitude $[0, m+1]$. Therefore, $\bT_{G_0/G_{-1}}$ is perfect of Tor-amplitude $[-m-1, 0]$ and $\bT_{G_0/G_{-1}}[n]$ is perfect or Tor-amplitude $[-m-1-n, -n]$. For any $n\geq 0$ this complex is connective and hence by \cref{prop:Totcovering} the morphism $T[n]G_0\rightarrow T[n]G_{-1}$ is an effective epimorphism.
\end{proof}

\begin{ex}
Consider a $(-1)$-smooth effective epimorphism $G_0\rightarrow G_{-1}$, where $G_0$ is a smooth scheme. Then $G_\bullet$ is a smooth groupoid $G_0\times_{G_{-1}} G_0\rightrightarrows G_0$. Taking $n=0$ we see that $T(G_0\times_{G_{-1}} G_0)\rightrightarrows TG_0$, known as the \defterm{tangent groupoid} (see \cite[Example 1.1.16]{Mackenzie}), is the groupoid presenting the stack $TG_{-1}$.
\end{ex}

Next, let us describe $n$-shifted cotangent groupoids. For this recall from \cref{ex:conormal} that if $X\rightarrow Y$ is a morphism of derived stacks which admit perfect cotangent complexes, there is a natural $n$-shifted Lagrangian structure on the $n$-shifted conormal bundle $N^*[n](X/Y)\rightarrow T^*[n]Y$.

\begin{df}
Let $f\colon G_0\rightarrow G_{-1}$ be a morphism of derived stacks which admit perfect cotangent complexes and denote $G_\bullet = \bN(f)$. The \defterm{$n$-shifted cotangent groupoid of $G_\bullet$} is the groupoid stack $T^*[n](G_\bullet)$ given by $\bN(N^*[n+1](G_0/G_{-1})\rightarrow T^*[n+1]G_{-1})$.
\end{df}

By \cref{thm-symplectic-quotient} we see that the $n$-shifted cotangent groupoid is an $n$-shifted symplectic groupoid.

\begin{rmk}
The above is a somewhat roundabout description of the shifted cotangent groupoid given in terms of the morphism $G_0\rightarrow G_{-1}$ instead of the groupoid $G_\bullet$. A more direct construction will be given in the upcoming paper of the first author with David Kern.
\end{rmk}

Unpacking the definition, we see that for a groupoid
\[
\xymatrix{
\dots \ar@<.8ex>[r] \ar[r] \ar@<-.8ex>[r] & G_1 \ar@<.5ex>[r] \ar@<-.5ex>[r] & G_0
}
\]
its $n$-shifted cotangent groupoid is given by
\[
\xymatrix{
\dots \ar@<.8ex>[r] \ar[r] \ar@<-.8ex>[r] & T^*[n] G_1 \ar@<.5ex>[r] \ar@<-.5ex>[r] & N^*[n](G_0/G_1).
}
\]

\begin{prop}\label{prop:shiftedcotangentgroupoid}
Let $G_\bullet$ be a groupoid stack and denote the natural projection by $f\colon G_0\rightarrow G_{-1}=|G_\bullet|$. Moreover, assume that $G_{-1}$ admits a perfect cotangent complex and $G_0$ is $m$-smooth for some $m\geq -1$. Then for any $n\geq m+1$ there is an isomorphism of $(n+1)$-shifted symplectic stacks
\[|T^*[n](G_\bullet)|\cong T^*[n+1] G_{-1}.\]
\end{prop}
\begin{proof}
By \cref{prop:realizationfullyfaithful} we have to show that $N^*[n+1](G_0/G_{-1})\rightarrow T^*[n+1] G_{-1}$ is an effective epimorphism. Consider the morphism $\bL_{G_0/G_{-1}}[n]\rightarrow f^* \bL_{G_{-1}}[n+1]$. Its homotopy fiber is $\bL_{G_0}[n]$. Since $Z$ is $m$-smooth, $\bL_{G_0}[n]$ is perfect of Tor-amplitude $[-n, m+1-n]$. For any $n\geq m+1$ this complex is connective and hence by \cref{prop:Totcovering} the morphism $N^*[n+1](G_0/G_{-1})\rightarrow T^*[n+1] G_{-1}$ is an effective epimorphism.
\end{proof}

\begin{ex}
Consider a smooth effective epimorphism $G_0\rightarrow G_{-1}$, where $G_0$ is a smooth scheme. Then $G_\bullet$ is a smooth groupoid $G_0\times_{G_{-1}} G_0\rightrightarrows G_0$. Let $L$ be its Lie algebroid. Then the $0$-shifted cotangent groupoid of $G_\bullet$ is the \defterm{cotangent groupoid} $T^*(G_0\times_{G_{-1}} G_0)\rightrightarrows L^*$ (\cite[Section 11.3]{Mackenzie}). By \cref{prop:shiftedcotangentgroupoid} it presents the $1$-shifted symplectic stack $T^*[1] G_{-1}$.
\end{ex}


\subsection{Orientations on co-groupoids}

Let us now define objects dual to symplectic groupoids, which will appear as sources in the AKSZ construction. Let $\cC$ be an $\infty$-category which admits finite colimits.

\begin{df}
A \defterm{Segal co-groupoid in $\cC$} is a Segal groupoid in $\cC^{\op}$.
\end{df}

Denote by $\partial^k$ and $\sigma^k$ the coboundary and codegeneracy maps in a Segal co-groupoid. Unpacking the definition, a Segal co-groupoid is a cosimplicial object $C^\bullet\colon \Delta\rightarrow \cC$ satisfying the following conditions:
\begin{enumerate}
\item For every $n>0$ the natural morphism
\[\underbrace{C^1\coprod_{C^0} \dots \coprod_{C^0} C^1}_{n\text{ times}}\longrightarrow C^n\]
is an isomorphism.

\item The morphism
\[\partial^0\coprod \partial^1\colon C^1\coprod_{C^0} C^1\longrightarrow C^2\]
is an isomorphism.
\end{enumerate}

We will mainly be interested in Segal co-groupoids in $\dSt$ which we call \defterm{co-groupoid stacks}.

\begin{ex}
Let $f\colon X\rightarrow Y$ be a morphism in $\cC$. We can construct its \v{C}ech co-nerve $\bN^{co}(f)$ such that $\bN^{co}(f)_n = \underbrace{Y\coprod_X \dots \coprod_X Y}_{n\text{ times}}$. Then $\bN^{co}(f)$ is a Segal co-groupoid.
\end{ex}

Let $\cB(d)\colon \dSt\rightarrow \cS$ be the functor which sends $C\mapsto \Map(\Gamma(C, \cO_C)[d], \bK)$. We naturally have $\Omega_0\big(\cB(d)\big)\simeq \cB(d+1)$. We are now ready to define orientations on co-groupoids. Consider the co-groupoid
\[
\B^\bullet \cB(d) = \bN^{co}(\cB(d-1)\xrightarrow{0} \emptyset)
\]
in $\Fun(\dSt, \cS)^{\op}$. Note that $\B^0 \cB(d)\cong \emptyset$ and $\B^1 \cB(d)\cong \cB(d)$.
If we view $\B^\bullet \cB(d)$ as a simplicial object in $\Fun(\dSt, \cS)$ then it is equivalent 
to the nerve of the abelian group $\cB(d)$. The Yoneda embedding provides a fully faithful functor $\dSt\hookrightarrow \Fun(\dSt, \cS)^{\op}$.

\begin{df}
Let $C^\bullet$ be a co-groupoid stack. A \defterm{pre-orientation of dimension $d$} on $C^\bullet$ is a morphism of co-groupoids
\[\B^\bullet \cB(d)\longrightarrow C^\bullet.\]
\end{df}

\begin{rmk}
A pre-orientation of dimension $d$ on $C^\bullet$ is equivalently a map of simplicial complexes
\[
\Gamma(C^\bullet, \cO_{C^\bullet})\rightarrow \bN_\bullet(0\rightarrow \bK[1-d])\simeq \B_\bullet(\bK[-d]).
\]
\end{rmk}

As in \cref{lm:presymplecticgroupoidisotropic} we obtain a pre-orientation $[C^1]$ of dimension $d$ on $C^1$ and a relative pre-orientation $\alpha_k$ of dimension $d$ on the cospan $(C^1)^{\coprod k}\rightarrow C^k\leftarrow C^1$.

\begin{df}\label{df:orientedcogroupoid}
A \defterm{$d$-oriented co-groupoid} is a co-groupoid stack $C^\bullet$ equipped with a pre-orientation of dimension $d$ and which satisfies the following conditions:
\begin{enumerate}
\item The stacks $C^0$ and $C^1$ are $\cO$-compact (in particular, it implies that $C^k$ are $\cO$-compact for all $k$).
\item The relative pre-orientation $\alpha_2$ of dimension $d$ on the cospan $C^1\coprod C^1\rightarrow C^2\leftarrow C^1$ is non-degenerate (in particular, it implies that $[C^1]$ is an orientation of dimension $d$).
\end{enumerate}
\end{df}

A source of $d$-oriented co-groupoids is provided by the following construction. Let $f\colon X\rightarrow Y$ be a morphism equipped with a relative pre-orientation of dimension $d-1$. In other words, we have a commutative diagram
\[
\xymatrix{
\cB(d-1) \ar[d] \ar^{0}[r] & \emptyset \ar[d]\\
X \ar^{f}[r] & Y
}
\]

Taking the co-nerve of horizontal morphisms, we obtain a morphism $\B^\bullet \cB(d)\rightarrow \bN^{co}(f)$, i.e. $\bN^{co}(f)$ is naturally equipped with a pre-orientation of dimension $d$.

\begin{prop}
Suppose $f\colon C^{-1}\rightarrow C^0$ is equipped with an orientation of dimension $d-1$. Then $\bN^{co}(f)$ is a $d$-oriented co-groupoid.
\end{prop}
\begin{proof}
The proof is analogous to the proof of \cref{easy-proposition}. Let $C^\bullet = \bN^{co}(f)$. Consider the $(\infty, 2)$-category $\Or_2^{d-1}$ of $(d-1)$-oriented cospans from \cite[Section 2.7]{CHS}. The $(d-1)$-oriented cospan $\emptyset\rightarrow C^0\leftarrow C^{-1}$ defines a morphism from $\emptyset$ to $C^{-1}$ in $\Or_2^{d-1}$. By \cite[Proposition 2.11.1]{CHS} it admits a right adjoint given by the $(d-1)$-oriented cospan $C^{-1}\rightarrow C^0\leftarrow \emptyset$ with the counit of the adjunction given by the iterated oriented cospan
\[
\twocospan{C^0}{C^0\coprod C^0}{C^{-1}}{C^{-1}}{C^{-1}}
\]
Therefore, $C^1$ viewed as a 1-morphism from $\emptyset$ to $\emptyset$ in $\Or_2^{d-1}$ is a monad. This implies two facts:
\begin{enumerate}
    \item First, $C^1$ is an object of $\Or_1^d$, i.e. $[C^1]$ is an orientation of dimension $d$.
    \item Second, the multiplication 2-morphism of the monad, which is
    \[
    \twocospan{C^2}{C^1\coprod C^1}{C^1}{\emptyset}{\emptyset}
    \]
    is an iterated oriented cospan. Thus, $C^1\coprod C^1\rightarrow C^2\leftarrow C^1$ is an oriented cospan.
\end{enumerate}
\end{proof}

\begin{ex}
Consider the morphism $\emptyset\rightarrow pt$. It naturally carries a relative orientation of dimension $-1$. For instance, we may consider it as a 
morphism $(S^{-1})_\B\rightarrow (D^0)_\B$. Therefore, the co-groupoid $pt\rightrightarrows pt\coprod pt$ is $0$-oriented.
\label{ex:Cechnervecogroupoid}
\end{ex}


\subsection{AKSZ--PTVV construction}

If $C$ is a derived stack with an orientation of dimension $d$ and $X$ is an $n$-shifted symplectic stack, then $\bMap(C, X)$ carries an $(n-d)$-shifted symplectic structure. In this section we generalize this result to $n$-shifted symplectic groupoids.

Let $C^\bullet$ be a co-groupoid equipped with a pre-orientation of dimension $d$ and $X$ a derived stack equipped with an $n$-shifted presymplectic structure. Consider the correspondence
\[
\xymatrix{
& \bMap(C^\bullet, X)\times C^\bullet \ar^{\pi}[dl] \ar^{ev}[dr] & \\
\bMap(C^\bullet, X) && X
}
\]

Then $ev^*\omega_X$ naturally defines an element in the end
\[ev^*\omega_X\in \int_{[m]\in\Delta} \Acl\big(\bMap(C^m, X)\times C^m, n\big).\]

Recall that a pre-orientation of dimension $d$ on $C^\bullet$ is a morphism of simplicial complexes
\[\Gamma(C^\bullet, \cO)\longrightarrow \bN(0\rightarrow \bK[1-d]).\]

Now suppose $C^\bullet$ are $\cO$-compact. Then we get an integration map
\[\int \DR(\bMap(C^\bullet, X)\times C^\bullet)\rightarrow \int \DR(\bMap(C^\bullet, X)) \otimes \Gamma(C^\bullet, \cO)
\rightarrow \int \DR(\bMap(C^\bullet, X))\otimes \bN(0\rightarrow \bK[1-d]).\]

On the level of closed 2-forms, this gives an integration map
\[\int \Acl(\bMap(C^\bullet, X)\times C^\bullet, n)\rightarrow \Map(\bMap(C^\bullet, X), \B_\bullet\Acl(n-d)).\]

In other words, $\bMap(C^\bullet, X)$ becomes an $(n-d)$-shifted presymplectic groupoid.

\begin{thm}\label{thm:AKSZsymplectic}
Suppose $C^\bullet$ is a $d$-oriented co-groupoid and $X$ an $n$-shifted symplectic stack. Then $\bMap(C^\bullet, X)$ is an $(n-d)$-shifted symplectic groupoid.
\end{thm}
\begin{proof}
We need to show that the $(n-d)$-shifted isotropic correspondence $\bMap(C^1\coprod C^1, X)\leftarrow \bMap(C^2, X)\rightarrow \bMap(C^1, X)$ is $(n-d)$-shifted Lagrangian, which follows from \cite[Theorem 2.9]{CalaqueTFT}.
\end{proof}

\begin{ex}\label{ex-coor-conerve}
Let $f\colon C^{-1}\to C^0$ be a morphism equipped with a relative $(d-1)$-orientation. 
By \cite[Theorem 2.9]{CalaqueTFT}, if $X$ is $n$-shifted symplectic, then the $(n-d+1)$-shifted isotropic structure on 
$f^*_X\colon \bMap(C^0,X)\to \bMap(C^{-1},X)$ is $(n-d+1)$-shifted Lagrangian. 
Applying \cref{thm:AKSZsymplectic} to the $d$-oriented co-groupoid $C^\bullet=\bN^{co}(f)$, we get a $(n-d)$-shifted symplectic groupoid 
$\bMap(C^\bullet,X)$, that is actually equivalent to the $(n-d)$-shifted symplectic groupoid $\bN(f^*_X)$. 
\end{ex}

Now suppose $L\rightarrow X$ carries an $n$-shifted isotropic structure and $C^\bullet$ is a co-groupoid equipped with a pre-orientation of dimension $d$ as before. Then we get a diagram
\[
\xymatrix{
\bMap(C^\bullet, L) \ar[d] \ar[r] & \bMap(C^\bullet, X) \ar[d] \\
\ast \ar[r] & \B_\bullet \Acl(n-d)
}
\]
In other words, the morphism of groupoids $\bMap(C^\bullet, L)\rightarrow \bMap(C^\bullet, X)$ carries an $(n-d)$-shifted isotropic structure.

\begin{thm}\label{thm:AKSZLagrangian}
Suppose $C^\bullet$ is a $d$-oriented co-groupoid and $L\rightarrow X$ an $n$-shifted Lagrangian morphism. Then $\bMap(C^\bullet, L)\rightarrow \bMap(C^\bullet, X)$ is an $(n-d)$-shifted Lagrangian groupoid.
\end{thm}
\begin{proof}
We need to show that the diagram
\[
\twospan{\bMap(C^2, L)}{\bMap(C^1\coprod C^1, L)\times_{\bMap(C^1\coprod C^1, X)} \bMap(C^2, X)}{\bMap(C^1, L)}{\bMap(C^1, X)}{pt}
\]
of stacks over $\Acl(n-d)$ is non-degenerate. This is the content of \cref{cor:relativeAKSZ}.
\end{proof}

\begin{cor}\label{cor:groupoidAKSZ}
Suppose $C^\bullet$ is a $d$-oriented co-groupoid and $L\rightarrow X$ an $n$-shifted Lagrangian morphism, where $X$ is $n$-shifted symplectic. Then
\[
\bMap(C^\bullet\rightarrow C^0, L\rightarrow X) = \bMap(C^0, X) \times_{\bMap(C^\bullet, X)} \bMap(C^\bullet, L)
\]
is an $(n-d-1)$-shifted symplectic groupoid.
\end{cor}
\begin{proof}
Since $\bMap(C^\bullet, X)$ is an $(n-d)$-shifted symplectic groupoid by \cref{thm:AKSZsymplectic}, we have a constant $(n-d)$-shifted Lagrangian groupoid 
$\bMap(C^0, X)\rightarrow \bMap(C^\bullet, X)$ as in \cref{ex:constantgroupoidLagrangian}. By \cref{thm:AKSZLagrangian} 
$\bMap(C^\bullet, L)\rightarrow \bMap(C^\bullet, X)$ is also an $(n-d)$-shifted Lagrangian groupoid. 
Therefore, by \cref{prop:Lagrangianintersection} we have that
\[
\bMap(C^0, X) \times_{\bMap(C^\bullet, X)} \bMap(C^\bullet, L)
\]
is an $(n-d-1)$-shifted symplectic groupoid.
\end{proof}
\begin{ex}
We borrow the assumptions and notation from \cref{ex-coor-conerve}, and further assume that there is a Lagrangian morphism $L\to X$. Then (again using \cref{cor:relativeAKSZ}, as in the proof of Theorem \ref{thm:AKSZLagrangian}) 
we get a $2$-morphism 
\[
\twospan{\bMap(C^0, L)}{\bMap(C^0, X)}{\bMap(C^{-1}, L)}{\bMap(C^{-1}, X)}{pt}
\]
in $\Lag_2^{n-d+1}$. By \cref{prop-lagr-groupoid-from-2-lag}, we obtain an $(n-d)$-shifted Lagrangian groupoid $\bN(f^*_L)\to \bN(f^*_X)$, that coincides with 
the $(n-d)$-shifted Lagrangian groupoid $\bMap(C^\bullet,L)\to \bMap(C^\bullet,X)$ from \cref{thm:AKSZLagrangian}. 
Then the $(n-d-1)$-shifted symplectic groupoid $\bMap(C^\bullet\rightarrow C^0, L\rightarrow X)$ from \cref{cor:groupoidAKSZ} appears is the nerve of 
the $(n-d)$-shifted Lagrangian morphism $\bMap(C^0,L)\to \bMap(C^{-1},L)\underset{\bMap(C^{-1},X)}{\times}\bMap(C^0,X)$. 
\end{ex}


\section{Deformation to the normal cone}

In this section we explain how to perform deformation to the normal cone for Lagrangian morphisms.


\subsection{Orientations of the zero locus a section}

Let $\cL\rightarrow X$ be a line bundle over a derived stack and $s\in\Gamma(X, \cL)$ a section. In this section we will define (relative) orientations on the zero locus of $s$. It will be convenient to phrase the situation in a universal way.

Let $\Theta=[\bA^1/\bG_m]$, where $\bG_m$ acts on the coordinate of $\bA^1$ with weight $1$. $\Theta$ is the moduli space of pairs of a line bundle with a section. Let $\cO_\Theta(n)$ be the trivial line bundle on $\bA^1$ corresponding to the representation of $\bG_m$ of weight $n$. For instance, $\cO_\Theta(1)$ is the universal line bundle over $\Theta$. Let $i\colon \B\bG_m\rightarrow \Theta$ be the inclusion of the origin, which is the zero locus of the canonical section $x\colon \cO_\Theta\rightarrow \cO_\Theta(1)$.

\begin{prop}\label{prop:orientedzerolocus}
There is an $\cO_\Theta(1)[-1]$-orientation on the morphism $i\colon \B\bG_m\rightarrow \Theta$ over $\Theta$.
\end{prop}
\begin{proof}
Since $i$ is schematic and quasi-compact, it is universally cocontinuous by \cite[Chapter 3, Proposition 2.2.2]{GR1}. Since $i$ is proper and quasi-smooth, it is universally perfect by \cite[Lemma 2.2]{ToenProper}. Therefore, $i$ is $\cO$-compact. The morphisms $\Theta\xrightarrow{id} \Theta$ and $\varnothing\rightarrow \Theta$ are obviously $\cO$-compact.

We have a fiber sequence
\[
\cO_\Theta\xrightarrow{x} \cO_\Theta(1)\rightarrow i_*i^*\cO_\Theta(1).
\]
The connecting homomorphism $i_*i^*\cO_\Theta(1)\rightarrow \cO_\Theta[1]$ defines an $i^*\cO_\Theta(1)[-1]$-preorientation on $\B\bG_m$ and the resulting nullhomotopy $\cO_\Theta(1)\rightarrow i_*i^*\cO_\Theta(1)\rightarrow \cO_\Theta[1]$ defines an $\cO_\Theta(1)[-1]$-preorientation on the morphism $\B\bG_m\rightarrow \Theta$ over $\Theta$.

Next, we prove that the morphism $\B\bG_m\rightarrow \Theta\leftarrow \varnothing$ is oriented over $\Theta$. By \cite[Proposition 2.2.5]{CHS} we have to show that for every morphism $\sigma\colon S'\rightarrow \Theta$ from a derived affine scheme $S'$ the pullback $i'=(i\times id)\colon \B\bG_m\times_\Theta S'\rightarrow S'$ is weakly oriented over $S'$. Such a morphism $\sigma\colon S'\rightarrow \Theta$ corresponds to a line bundle $\cL$ with a section $s$ with $\B\bG_m\times_\Theta S'$ identifying with the zero locus of $s$. For every perfect complex $E\in\Perf(S')$ we have a fiber sequence
\[E\otimes \cL^{-1}\xrightarrow{s} E\rightarrow i_* i^* E.\]
Therefore, to show that the morphism $i'$ is weakly oriented over $S'$ we have to show that the pairing
\[(E^\vee\otimes \cL)\otimes (E\otimes \cL^{-1})\longrightarrow \cO_{S'}\]
is non-degenerate, which is obvious. Therefore, $i$ is oriented over $\Theta$.
\end{proof}

Using the observation of \cref{ex-nd-for-orientations}(2), we see that $\B\bG_m$ carries an $i^*\cO_\Theta(1)[-1]$-orientation over $\Theta$. We can identify the fibers of the morphism $i$ from \cref{prop:orientedzerolocus} over $\Theta$ as follows:
\begin{itemize}
\item The fiber of $\B\bG_m\rightarrow \Theta$ over a nonzero $\lambda\in\bA^1$ is $\varnothing\rightarrow pt$ which has an orientation of degree $-1$ corresponding to the obvious orientation of degree $0$ of $pt$.
\item The fiber of $\B\bG_m\rightarrow \Theta$ over the origin is $\bA^1[-1]\rightarrow pt$.
\end{itemize}

\begin{rmk}
We have $\bA^1[-1] = \Spec(\bK[y])$, where $y$ is a variable of degree $-1$. The $(-1)$-orientation on $\bA^1[-1]$ is given by the morphism $[\bA^1[-1]]\colon \cO(\bA^1[-1])\cong k[y]\rightarrow k[1]$ given by $a + by\mapsto b$. The boundary structure on $\bA^1[-1]\rightarrow pt$ is given by the obviously commuting diagram
\[
\xymatrix{
\cO(pt)\cong k \ar[r] \ar[d] & \cO(\bA^1[-1])\cong k[y] \ar^{[\bA^1[-1]]}[d] \\
0 \ar[r] & k[1]
}
\]
\end{rmk}

\begin{cor}\label{cor:relativeorientationsection}
Let $X$ be a derived prestack equipped with a line bundle $\cL$ and a section $s$. Let $i\colon Z\rightarrow X$ be the zero locus of $s$. Then the inclusion $Z\rightarrow X$ has a natural $\cL[-1]$-orientation over $X$.
\end{cor}
\begin{proof}
The data of a line bundle and a section specifies a morphism $X\rightarrow \Theta$. Pulling back the $\cO_\Theta(1)[-1]$-orientation of $\B\bG_m\rightarrow \Theta$ along this morphism gives rise to an $\cL[-1]$-orientation of $Z\rightarrow X$.
\end{proof}

To relate the above orientation on the inclusion $Z\rightarrow X$ of the zero section to more classical notions, we may pushforward the above relative orientation. Recall from \cref{prop:Gorensteinorientation} that for a Gorenstein derived scheme $X$ we have a graded canonical bundle $\cK_X$. Let $K_X$ be the ungraded canonical bundle, so that $\cK_X=K_X[d]$ for some $d$.

\begin{prop}
Let $X$ be a proper Gorenstein derived scheme equipped with a section $s$ of $K_X^{-1}$. Then the inclusion of its zero locus $Z\rightarrow X$ carries an orientation of dimension $(d-1)$.
\end{prop}
\begin{proof}
By \cref{cor:relativeorientationsection} we see that the inclusion $Z\rightarrow X$ carries a $K_X^{-1}[-1]$-orientation over $X$. By \cref{prop:Gorensteinorientation} $X$ itself carries a $K_X[d]$-orientation over $pt$. Composing the two, using \cref{prop:orientationcomposition} we obtain that $Z\rightarrow X$ carries an $\cO[d-1]$-orientation over $pt$, i.e. an orientation of dimension $(d-1)$.
\end{proof}

\begin{rmk}
The above orientation of the zero locus of a section of the anticanonical line bundle coincides with the one constructed in \cite[Section 3.2.1]{CalaqueTFT}.
\end{rmk}

\subsection{Deformation space}\label{sect:deformationspace}

Given a morphism $f\colon L\rightarrow X$ of derived prestacks one can form its deformation to the normal cone. We will use the modern definition of the deformation to the normal cone via Weil restriction following the ideas of \cite{Simpson} further developed in \cite[Chapter 9]{GR2}, \cite{KhanRydh} and \cite{Hekking}.

\begin{df}
Let $f\colon L\rightarrow X$ be a morphism of derived prestacks. Its \defterm{deformation to the normal cone} is the relative mapping prestack
\[\uD_{L/X} = \bMap_{X\times \Theta}(X\times \B\bG_m, L\times \Theta).\]
\end{df}

\begin{rmk}
Unpacking the definition, the $T$-points of the derived prestack $\uD_{L/X}$ parametrize the following collection of data:
\begin{itemize}
    \item A morphism $g\colon T\rightarrow X$ of derived prestacks.
    \item A line bundle $\cL$ over $T$ together with a section $s$.
    \item A morphism $g_0\colon s^{-1}(0)\rightarrow L$ fitting into a commutative diagram
    \[
    \xymatrix{
    s^{-1}(0) \ar^{g_0}[r] \ar[d] & L \ar^{f}[d] \\
    T \ar^{g}[r] & X
    }
    \]
\end{itemize}
\end{rmk}

Let us list some properties of the deformation space:
\begin{enumerate}
\item There are a few equivalent ways to define the prestack $\uD_{L/X}$:
\begin{align*}
\uD_{L/X}&\cong \bMap_{\Theta}(\B\bG_m, L\times \Theta)\times_{\bMap_{\Theta}(\B\bG_m, X\times \Theta)} (X\times \Theta) \\
&\cong \Res_{\B\bG_m/\Theta}(L\times \B\bG_m)\times_{\Res_{\B\bG_m/\Theta}(X\times \B\bG_m)} (X\times \Theta) \\
&\cong \Res_{(X\times \B\bG_m)/(X\times\Theta)}(L\times\B\bG_m)
\end{align*}
\item If $L$ and $X$ are derived Artin stacks, then $\uD_{L/X}$ is a derived Artin stack by \cite[Theorem 5.1.1]{HalpernLeistnerPreygel} (for $1$-geometric stacks) and \cite{HKR} (in general).
\item There are morphisms of derived prestacks $L\times \Theta\rightarrow \uD_{L/X}\rightarrow X\times \Theta$ which compose to $f\times id$.
\item The fiber of $\uD_{L/X}$ over a nonzero $\lambda\in\bA^1$ is $\bMap_X(\varnothing, L)\cong X$.
\item The fiber of $\uD_{L/X}$ over the origin $0\in\bA^1$ is $\bMap_X(X\times \bA^1[-1], L)$. If $L\rightarrow X$ admits a relative cotangent complex, using \cref{prop:maptangent} we may identify it with $T[1](L/X)$.
\end{enumerate}

\begin{rmk}
The deformation space as defined in \cite[Theorem 4.1.13]{KhanRydh} and \cite[Definition 6.3.1]{Hekking} is given by
\[\bD_{L/X} = \uD_{L/X}\times_\Theta \bA^1.\]
\end{rmk}

\begin{rmk}
Consider the morphism $pt\rightarrow \bA^1$ over $\bA^1$ which is naturally $\bG_m$-equivariant. Its \v{C}ech conerve, i.e. the co-groupoid
\[
\xymatrix{
\bA^1 \ar@<.5ex>[r] \ar@<-.5ex>[r] & \bA^1\coprod_{pt} \bA^1 \ar@<.8ex>[r] \ar[r] \ar@<-.8ex>[r] & \bA^1\coprod_{pt} \bA^1\coprod_{pt} \bA^1 \ar@<1.2ex>[r] \ar@<.4ex>[r] \ar@<-.4ex>[r] \ar@<-1.2ex>[r] & \ldots,
}
\]
is exactly the co-groupoid $\Bifurc^\bullet$ from \cite[Chapter 9, Section 2.2]{GR2}. The \v{C}ech nerve of
\[L\times \bA^1\longrightarrow \bD_{L/X}\cong \Res_{pt/\bA^1}(L)\times_{\Res_{pt/\bA^1}(X)} (X\times \bA^1)\]
can thus be identified with
\[\Res_{\Bifurc^\bullet/\bA^1} (L\times \Bifurc^\bullet) \times_{\Res_{\Bifurc^\bullet/\bA^1} (X\times \Bifurc^\bullet)} (X\times \bA^1)\]
using that for every base stack $S$ the contravariant functor $C\mapsto \Res_{C/S}(X\times_S C)$ sends colimits in $C$ to limits. This groupoid is precisely the groupoid $\mathcal{R}_\bullet$ from \cite[Chapter 9, equation (2.2)]{GR2} used to define the deformation to the normal cone.
\end{rmk}

Our goal in this section is to endow $\uD_{L/X}$ with a relative symplectic structure assuming that $L\rightarrow X$ has a Lagrangian structure.

\begin{prop}\label{thm:Weilcorrespondence}
Suppose $L\rightarrow X$ is a morphism of derived prestacks equipped with an $n$-shifted Lagrangian structure. Then the correspondence
\[
\xymatrix{
& \bMap_\Theta(\B\bG_m, X\times \Theta) & \\
X\times \Theta \ar[ur] \ar[dr] & L\times \Theta \ar[l] \ar[r] & \bMap_\Theta(\B\bG_m, L\times \Theta) \ar[ul] \ar[dl] \\
& \Theta & 
}
\]
has a natural structure of a 2-fold $\cO_\Theta(-1)[n+1]$-twisted Lagrangian correspondence. In particular, the morphism
\[L\times\Theta\longrightarrow \uD_{L/X}\cong \bMap_\Theta(\B\bG_m, L\times \Theta)\times_{\bMap_\Theta(\B\bG_m, X\times \Theta)} (X\times\Theta)\]
has a natural structure of a $\cO_\Theta(-1)[n]$-twisted Lagrangian morphism.
\end{prop}
\begin{proof}
By \cref{prop:orientedzerolocus} we have an $\cO_\Theta(1)[-1]$-orientation on the morphism $\B\bG_m\rightarrow \Theta$ and an $\cO_\Theta(0)[n]$-twisted Lagrangian morphism $L\times \Theta\rightarrow X\times \Theta$. Therefore, the claim follows by the AKSZ construction described in \cref{cor:relativeAKSZ}.
\end{proof}

Individual fibers of the 2-fold Lagrangian correspondence from \cref{thm:Weilcorrespondence} may be identified as follows:
\begin{itemize}
    \item The fiber at $1\in\bA^1$ is the 2-fold $(n+1)$-shifted Lagrangian correspondence
\[
\xymatrix{
& \bMap(\varnothing, X)\cong pt & \\
X \ar[ur] \ar[dr] & L \ar[l] \ar[r] & \bMap(\varnothing, L)\cong pt \ar[ul] \ar[dl] \\
& pt & 
}
\]
which recovers the original $n$-shifted Lagrangian morphism $L\rightarrow X$.
    \item The fiber at $0\in\bA^1$ is the 2-fold $(n+1)$-shifted Lagrangian correspondence
\[
\xymatrix{
& \bMap(\bA^1[-1], X) & \\
X \ar[ur] \ar[dr] & L \ar[l] \ar[r] & \bMap(\bA^1[-1], L) \ar[ul] \ar[dl] \\
& pt & 
}
\]
\end{itemize}

Our goal is to identify the latter 2-fold $(n+1)$-shifted Lagrangian correspondence explicitly. By \cref{prop:maptangent} we may identify $\bMap(\bA^1[-1], X)\cong T[1]X$, so that the restriction along $\bA^1[-1]\rightarrow pt$ corresponds to the inclusion of the zero section $X\rightarrow T[1] X$. Denote by
\[ev\colon \bA^1[-1]\times T[1] X\longrightarrow X\]
the corresponding evaluation map. By \cref{prop:functionslinearstack} we may identify
\[\Gamma^{\gr}_X(T[1] X, \cO_{T[1]X})\cong \bigoplus_p \Sym^p(\bL_X[-1]).\]
For a $p$-form $\xi$ on $X$ we denote by $\overline{\xi}$ the corresponding function on $T[1] X$ of weight $p$.

\begin{prop}\label{prop:tangentintegration}
Let $X$ be a derived prestack which admits a cotangent complex. Then the composite
\[\Acl(X, n)\xrightarrow{ev^*} \Acl(\bA^1[-1]\times T[1] X, n)\xrightarrow{\int_{\bA^1[-1]}} \Acl(T[1] X, n+1)\]
fits into a commutative diagram
\[
\xymatrix{
\Acl(X, n) \ar[r] \ar[d] & \Acl(T[1] X, n+1) \\
\cA^2(X, n) \ar[r] & \cA^1(T[1] X, n+1)\ar^{\ddr}[u]
}
\]
where $\cA^2(X, n)\rightarrow \cA^1(T[1] X, n+1)$ is induced by
\[\wedge^2 \bL_X[n]\rightarrow \bL_X\otimes \bL_X[n]\rightarrow p_*p^*\bL_X[n+1]\]
given by $\xi\wedge\eta\mapsto \xi\overline{\eta} + \overline{\xi}\eta$.
\end{prop}
\begin{proof}
Recall that $\mathbb{A}^{1}[-1]=\Spec(\bK[y])$, with $y$ a degree $-1$ variable. 
Hence, $\DR(\mathbb{A}^{1}[-1])$ is the following graded mixed cdga: as a graded cdga, it is $\bK[y,z]$, where $y$ has degree $-1$ and weight $0$, 
and $z$ has degree $0$ and weight $1$, the internal differential is zero, and the mixed differential is defined by $\epsilon(y)=z$. 
For a derived prestack $Y$ an element $\omega$ of degree $d$ and weight $p$ in 
$\DR(\mathbb{A}^{1}[-1]\times Y)\simeq \DR(\mathbb{A}^{1}[-1])\otimes \DR(Y)$ is of the following form: 
\[
\omega=\sum_{j\geq0}(\omega_{0,j}z^j+\omega_{1,j}yz^j)\,
\]
with $\omega_{0,j}$ (resp.~$\omega_{0,j}$) of degree $d$ and weight $p-j$ (resp.~degree $d+1$ and weight $p-j$) in $\DR(Y)$. 

The integration morphism $\int_{\mathbb{A}^{1}[-1]}:\DR(\mathbb{A}^{1}[-1]\times Y)\to \DR(Y)$ sends such an $\omega$ to $\omega_{1,0}$. 

\medskip

The composed map 
\[
\Acl(\bA^{1}[-1]\times Y, n)\overset{\int_{\bA^{1}[-1]}}{\longrightarrow} \Acl(Y, n+1)\longrightarrow \cA^{1, \cl}(Y, n)
\]
is null-homotopic: borrowing the above notation for $\omega$, the null-homotopy $h$ is given by $h(\omega):=\omega_{0,1}$. 
Indeed: 
\begin{eqnarray*}
h\big((d+\epsilon)(\omega)\big) & = & \sum_{j\geq0}h\big(\epsilon(\omega_{0,j}z^j+\omega_{1,j}yz^j\big) \\
& = & \sum_{j\geq0}h\big((d+\epsilon)(\omega_{0,j})z^j+(d+\epsilon)(\omega_{1,j})yz^j+\omega_{1,j}z^{j+1}\big) \\
& = & (d+\epsilon)(\omega_{0,1})+\omega_{1,0} = (d+\epsilon)\big(h(\omega)\big)+\int_{\mathbb{A}^{1}[-1]}\omega\,.
\end{eqnarray*}
Now observe that the square 
\[
\xymatrix{
\cA^1(n+1) \ar[r]^-{\ddr} \ar[d] & \Acl(n+1) \ar[d] \\
0 \ar[r] & \cA^{1, \cl}(n)
}\]
is Cartesian. Therefore, the above null-homotopy induces a lift 
\[
\lambda\colon \Acl(\bA^{1}[-1]\times Y, n)\longrightarrow \cA^1(Y, n+1)
\]
of the integration morphism $\int_{\bA^{1}[-1]}$, given by $\lambda(\omega)=\omega_{0,1}^1$ (that is the weight $1$ component of $\omega_{0,1}$). 
In particular, for a closed $2$-form $\omega$, $\lambda(\omega)$ only depends on the underlying $2$-form of $\omega$, and 
\[
\int_{\bA^{1}[-1]}\omega=\ddr\big(\lambda(\omega)\big)\,.
\]
In other words, we obtain a commutative diagram
\[
\xymatrix@C=2cm{
\Acl(\bA^1[-1]\times Y, n) \ar^-{\int_{\bA^1[-1]}}[r] \ar[d] & \Acl(Y, n+1) \\
\cA^2(\bA^1[-1]\times Y, n) \ar^{\lambda}[r] & \cA^1(Y, n+1) \ar^{\ddr}[u]
}
\]

Let now $Y=\bMap(\bA^{1}[-1],X)\cong T[1]X$. Using that $\cA^2(X, n)$ sends colimits in $X$ to limits in $X$, the claim is reduced to the case $X$ a derived affine scheme.

For functions the pullback along the evaluation map $ev\colon T[1]X\times\bA^1[-1]\to X$ is given by the composition of 
\begin{eqnarray*}
ev^*\colon \mathcal O(X) & \longrightarrow & \mathcal O(T[1]X)\otimes \bK[y] \\
f & \longmapsto & \pi^*f\otimes 1+\overline{\ddr f}\otimes y\,,
\end{eqnarray*}
where we identify graded functions on $T[1]X$ and differential forms on $X$. The pullback
\[
ev^*\colon \DR(X) \longrightarrow \DR(T[1] X\times\bA^1[-1])\cong \DR(T[1] X)\otimes \bK[y, z]
\]
is uniquely determined by its value on functions using the universal property of $\DR(X)$. The compatibility of $ev^*$ with the mixed structure implies that for $1$-forms we get
\[ev^*(\xi) = \pi^*\xi\otimes 1+\ddr\overline{\xi}\otimes y+\overline{\xi}\otimes z+\overline{\ddr\xi}\otimes yz.\]
For a pair of $1$-forms $\xi, \eta$ the compatibility of $ev^*$ with multiplication implies that
\[ev^*(\xi\wedge\eta)_{0, 1} = \xi\overline{\eta}+\overline{\xi}\eta.\]
\end{proof}

Recall from \cref{sect:shiftedcotangent} and \cite{CalaqueCotangent} that there is a tautological one-form $\lambda_X\in\cA^1(T^*[n+1]X, n+1)$ equipped with a nullhomotopy $\eta_X$ of its pullback to the zero section $X\rightarrow T^*[n+1] X$; moreover, $\ddr\lambda_X$ defines an $(n+1)$-shifted symplectic structure on $T^*[n+1]X$ and $\ddr \eta_X$ defines an $(n+1)$-shifted Lagrangian structure on $X\rightarrow T^*[n+1] X$.

\begin{prop}
The fiber of the 2-fold Lagrangian correspondence from \cref{thm:Weilcorrespondence} at $0\in\bA^1$ is equivalent, as a 2-fold Lagrangian correspondence, to
\[
\twospan{L}{X}{N^*[n+1](L/X)}{T^*[n+1]X}{pt}
\]
\end{prop}
\begin{proof}
Both the 2-fold Lagrangian correspondence
\[
\twospan{L}{X}{N^*[n+1](L/X)}{T^*[n+1]X}{pt}
\]
and
\[
\twospan{L}{X}{T[1]L}{T[1]X}{pt}
\]
carry a $\bG_m$-action given by scaling the fibers of the cotangent and conormal bundles and $\bA^1[-1]$ with all differential forms of weight $1$. Therefore, we merely need to identify the $(n+1)$-shifted Lagrangian structure on the morphism
\[N^*[n+1](L/X)\longrightarrow T^*[n+1] X\]
with the one on
\[T[1]L\longrightarrow T[1] X\]
as there is a unique graded extension of this Lagrangian structure to the whole 2-fold correspondence.

Let $\omega_X$ be the $2$-form underlying the $n$-shifted symplectic structure on $X$ and $h_L$ the homotopy $f^*\omega_X\sim 0$ underlying the $n$-shifted Lagrangian structure on $f\colon L\rightarrow X$. By \cref{prop:tangentintegration} the $(n+1)$-shifted symplectic structure on $T[1] X$ is obtained by applying $\ddr$ to the image of $\omega_X$ under the composite
\[
\Gamma(X, \wedge^2\bL_X)[n]\longrightarrow \Gamma(X, \bL_X\otimes \bL_X)[n]\longrightarrow \Gamma(X, p_*p^*\bL_X)[n+1]\longrightarrow\Gamma(T[1] X, \bL_{T[1]X})[n+1].
\]

Similarly, the $(n+1)$-shifted symplectic structure on $T^*[n+1] X$ is obtained by applying $\ddr$ to the image of the canonical element under the composite
\[
\Gamma(X, \bT_X\otimes \bL_X)\longrightarrow \Gamma(X, \tilde{p}_*\tilde{p}^*\bL_X)[n+1]\rightarrow\Gamma(T^*[n+1] X, \bL_{T^*[n+1]X})[n+1],
\]
where $\tilde{p}\colon T^*[n+1]X\rightarrow X$ is the projection. Identifying $T^*[n+1]X\cong T[1]X$ using $\omega_X$ identifies the element of $\Gamma(X, \bL_X\otimes \bL_X)[n]$ with the canonical element of $\Gamma(X, \bT_X\otimes \bL_X)$. Therefore, the isomorphism $T[1] X\cong T^*[n+1] X$ is compatible with $(n+1)$-shifted symplectic structures.

The computation of the $(n+1)$-shifted Lagrangian structure on $f\colon T[1]L\rightarrow T[1]X$ is performed similarly. Namely, the $n$-shifted Lagrangian structure $h_L$ allows us to construct a lift in the diagram
\[
\xymatrix{
\bT_L \ar[r] \ar@{-->}[d] & f^*\bT_X \ar[d] & \\
\bL_{L/X}[n-1] \ar[r] & f^*\bL_X \ar[r] & \bL_L
}
\]
Under this isomorphisms $\bT_L\cong \bL_{L/X}[n-1]$ and $\bT_X\cong \bL_X[n]$ the nullhomotopy $h_L$ of the image of $\omega_X$ under
\[
\Gamma(X, \bL_X\otimes \bL_X)[n]\longrightarrow \Gamma(X, \bL_L\otimes \bL_L)[n]
\]
is identified with the nullhomotopy of the image of the canonical element under
\[
\Gamma(X, \bT_X\otimes \bL_X)\longrightarrow \Gamma(X, \bL_{L/X}[-1]\otimes \bL_L)
\]
used in the definition of the $(n+1)$-shifted Lagrangian structure on $N^*[n+1](L/X)\rightarrow T^*[n+1]X$ in \cite{CalaqueCotangent}.
\end{proof}

We obtain a proof of (a slightly modified version of) \cite[Conjecture 3.2]{CalaqueCotangent}. Note that \cite[Conjecture 3.2]{CalaqueCotangent} is missing the twist by $\cO_\Theta(-1)$. 

\begin{cor}\label{cor:calaqueconjecture}
Let $L\rightarrow X$ be a morphism of derived prestacks equipped with an $n$-shifted Lagrangian structure. Then the deformation to the normal cone $L\times\Theta\rightarrow \uD_{L/X}$ of $L\rightarrow X$ carries the structure of an $\cO_\Theta(-1)[n]$-twisted Lagrangian morphism. Its fiber at a nonzero $\lambda\in\bA^1$ is the original $n$-shifted Lagrangian morphism $L\rightarrow X$ and its fiber at $0\in\bA^1$ is the zero section $L\rightarrow T^*[n] L$.
\end{cor}

\begin{rmk}
Pulling back $\uD_{L/X}\rightarrow \Theta$ along $\bA^1\rightarrow \Theta$ we obtain that the morphism
\[L\times \bA^1\longrightarrow \bD_{L/X}\]
carries the structure of an $n$-shifted Lagrangian morphism relative to $\bA^1$.
\end{rmk}

\printbibliography
\end{document}